\documentclass[11pt,namelimits,sumlimits,a4paper]{amsart}
\usepackage{cite}
\usepackage{comment}

\usepackage{amssymb,amsmath}
\usepackage[mathscr]{eucal}
\usepackage{slashed}

\usepackage{graphicx}
\usepackage{amsmath,amsthm}
\usepackage{graphics}
\usepackage{color}
\usepackage{epsfig}
\usepackage{amssymb,amsmath}
\usepackage[mathscr]{eucal}
\usepackage[latin1]{inputenc}
\usepackage[T1]{fontenc}
\usepackage[UKenglish]{babel}
\usepackage{amsfonts}
\usepackage{fancyhdr}
\usepackage{graphicx}
\usepackage{amsmath}
\usepackage{amsthm}
\usepackage{amsmath,amscd}
\usepackage{latexsym}
\usepackage{cite}
\usepackage{amssymb,amsmath}

\usepackage{comment}
\usepackage[mathscr]{eucal}
\usepackage{slashed}
\usepackage{times,psfrag}
\usepackage{mathtools}
\usepackage{hyperref}
\usepackage{multirow}
\usepackage{longtable}
\usepackage{bm}
\usepackage[all]{xypic}
\usepackage{fancyhdr}
\usepackage{float}

\usepackage{esint}

\usepackage[]{graphicx}
\usepackage{graphics}
\usepackage{color}
\usepackage{epsfig}
\usepackage{layout}
\usepackage{enumitem}

\usepackage{lmodern}
\usepackage[rgb]{xcolor}
\usepackage[draft,author={Lorenzo Foscolo}]{pdfcomment}

\numberwithin{equation}{section}
\newtheorem{theorem}[equation]{Theorem}
\newtheorem{lemma}[equation]{Lemma}
\newtheorem{prop}[equation]{Proposition}
\newtheorem{corollary}[equation]{Corollary}

\theoremstyle{definition}
\newtheorem{definition}[equation]{Definition}

\theoremstyle{remark}
\newtheorem{remark}[equation]{Remark}
\newtheorem*{remark*}{Remark}

\newcounter{mtheorem}

\newtheoremstyle{mystyle}%                % Name
  {}%                                     % Space above
  {}%                                     % Space below
  {\itshape}%                             % Body font
  {}%                                     % Indent amount
  {\bfseries}%                            % Theorem head font
  {.}%                                    % Punctuation after theorem head
  { }%                                    % Space after theorem head, ' ', or \newline
  {}%                                     % Theorem head spec (can be left empty, meaning `normal')

\theoremstyle{mystyle}
\newtheorem{mtheorem}[mtheorem]{Theorem}

\newcommand{\abs}[1]{\lvert#1\rvert}

\newcommand{\ie}{\emph{i.e.} }

\newcommand{\cf}{\emph{cf.} }

\newcommand{\beq}{\begin{equation}}
\newcommand{\eeq}{\end{equation}}
\newcommand{\bea}{\begin{eqnarray}}
\newcommand{\eea}{\end{eqnarray}}

\newcommand{\C}{\mathbb{C}}

\newcommand{\R}{\mathbb{R}}

\newcommand{\Z}{\mathbb{Z}}
\newcommand{\N}{\mathbb{N}}

\newcommand{\PP}{\mathbb{P}}
\newcommand{\Sph}{\mathbb{S}}
\newcommand{\ra}{\rightarrow}

\newcommand{\vol}{\operatorname{Vol}}

\newcommand{\Real}{\operatorname{Re}}
\newcommand{\Imag}{\operatorname{Im}}

\newcommand{\tu}[1]{\textup{#1}}

\newcommand{\unitary}[1]{\textup{U$(#1)$}}

\newcommand{\sunitary}[1]{\textup{SU$(#1)$}}

\newcommand{{\isomgtc}}{\ensuremath{\sunitary{2}^3 \rtimes S_3}}

\newcommand{\Hess}{\operatorname{Hess}}
\newcommand{\hol}{\operatorname{hol}}
\newcommand{\Hol}{\operatorname{Hol}}
\newcommand{\Lam}{\Lambda}
\DeclareMathOperator{\tr}{tr}

\def\co{\colon\thinspace}

\begin{document}

\title{Cohomogeneity one solutions of the IIB system}

\author[L. Foscolo]{Lorenzo Foscolo}
\address{Sapienza Universit\`a di Roma, Piazzale A. Moro 5, 00185 Roma, Italia}
\email{lorenzo.foscolo@uniroma1.it}

\author[M. Garcia-Fernandez]{Mario Garcia-Fernandez}
\address{Instituto de Ciencias Matem\'aticas (CSIC-UAM-UC3M-UCM)\\ Nicol\'as Cabrera 13--15, Cantoblanco\\ 28049 Madrid, Spain}
\email{mario.garcia@icmat.es}

\thanks{LF was partially supported by the Ministero dell'Universit\`a e della Ricerca of Italy -- Avviso FIS2 grant FIS-2023-0395. MGF was partially supported by the Spanish Ministry of Science and Innovation, through the `Severo Ochoa Programme for Centres of Excellence in R\&D' (CEX2023-001347-S), and under grants PID2022-141387NB-C22 and PID2025-174260NB-C21, and by the European Union's Horizon 2020 research and innovation programme under the Marie Sk\l odowska-Curie grant agreement No 101273232 SURF FLOW}

\maketitle

\begin{abstract}
The IIB system is a system of PDEs for an $\sunitary{3}$--structure and a positive function on a 6-manifold. Its solutions describe conformally balanced pluriclosed Hemitian metrics on complex 3-folds with holomorphically trivial canonical line bundle. In particular, solutions of the IIB system are steady solitons for the generalized Ricci-flow and Bismut--Hermitian--Einstein metrics.

In this paper we study cohomogeneity one solutions of the IIB system. We construct 1-parameter families (up to scaling symmetries) of complete non-compact non-K\"ahler solutions on the smoothing of the conifold with controlled geometry at infinity. The generic member of the family is asymptotically conical with tangent cone at infinity the Calabi--Yau cone metric on the conifold. As a limit of the family we recover an explicit solutions known in the physics literature as the Chamseddine--Volkov/Maldacena--Nu\~nez solution, which has an exotic asymptotic geometry. We show that there are no complete cohomogeneity one solutions on crepant resolutions of the conifold (or its quotient), but we also establish the existence of an analogous 1-parameter family of forward complete solutions defined on exterior domains of the conifold with prescribed incomplete behaviour along the interior boundary and similar asymptotic behaviour at infinity.

As a byproduct of these existence results, we produce infinitely many complete non-compact non-K\"ahler Bismut--Hermitian--Einstein metrics in complex dimension 3 with full holonomy of the Bismut connection, in contrast to the holonomy reduction forced upon compact examples. We also find infinitely many such complete non-compact full-holonomy non-K\"ahler examples (with at least two ends) in complex dimension 2 by revisiting a construction due to Callan--Harvey--Strominger in the physics literature.
\end{abstract}

\section{Introduction}

Motivated by the moduli space for K3 surfaces, M. Reid proposed that there may exist an irreducible moduli space of algebraic Calabi--Yau manifolds, with varying topology \cite{Reid87}. The basic underlying principle of \emph{Reid's fantasy} was earlier described by Clemens and Friedman, who showed that one can construct new, possibly non-K\"ahler complex manifolds with trivial canonical bundle, by contracting a collection of disjoint $(-1,-1)$ curves and then smoothing the resulting ordinary double point singularities. The key open question is whether all complex threefolds with trivial canonical bundle can be connected by a sequence of such \emph{conifold transitions}. A natural approach to this problem is via metric degenerations, using Yau's solution of the Calabi-Conjecture \cite{Yau76} and, in the realm of K\"ahler geometry, this has been implemented to some extent in \cite{RZ11,Song15,Tosatti09}. However, conifold transitions (and flops) may often end up in a non-K\"ahler complex manifold, where the powerful techniques of K\"ahler geometry do not apply.

In a series of papers, Fu, Li, and Yau proposed to implement this metric approach to Reid's fantasy using natural equations in string theory \cite{LiYau05,FLY12}, today known as the Hull--Strominger system \cite{Hull86,Strominger86}. The relevance of these equations is that the metric is hermitian, but no longer K\"ahler, and its torsion $H= - d^c\omega$ is coupled, via the \emph{Bianchi identity}, to an exact four-form source produced by an instanton, with an intensity which is proportional to the \emph{slope parameter} in string theory. Partial results in this direction have been obtained in \cite{Chuan12,CPY24,FPS26}, though the main question about solvability of the Bianchi identity through conifold transitions remains open.

In this setup, a natural, simpler, question is to study the limit in which the \emph{slope parameter} of the Hull--Strominger system tends to zero, under conifold transitions. Formally, the corresponding equations are given by
\begin{equation}\label{eq:IIBintro}
d(e^{-2\phi}\Omega)=0, \qquad d(e^{-2\phi}\omega^2)=0, \qquad dd^c\omega= 0,
\end{equation}
where $(\omega,\Omega)$ is an $\tu{SU}(3)$--structure on a six-dimensional smooth manifold $M$, and $\phi$ is a smooth function called the \emph{dilaton}. The corresponding system of PDEs is known as the \emph{IIB system}, first introduced by Gra\~na--Minasian--Petrini--Tomasiello in \cite{GMPT05} and later studied by Tseng--Yau in \cite{TY14}. In this limit, the equations make contact with the equations of motion of Type II string theory (see Proposition \ref{prop:soliton}), relating to the seminal studies of conifold transitions in the string theory literature \cite{ABM97,GMS95} (see \cite{Vafa} for the local case with fluxes, and the relation to mirror symmetry). Even though compact solutions of the system \ref{eq:IIBintro} are necessarily K\"ahler, non-compact non-K\"ahler solutions may exist, possibly arising as bubbles in the metric degeneration process. Further mathematical motivation for the IIB system is given by generalized geometry, where solutions can be interpreted as steady solitons for the generalized Ricci flow \cite{GFSBook}. Notice that, in particular, any such solution is \emph{Bismut--Hermitian--Einstein} (BHE) in the sense of \cite{GFJS}.

Motivated by this picture, in the present paper we study non-compact solutions of the IIB system on the local models for a conifold transition, given by the
smoothing of the conifold and its two (isomorphic) small resolutions. We will refer to these cases as the \emph{deformed} and \emph{resolved} conifold respectively. Remarkably, these non-compact complex manifolds admit a cohomogeneity
one action of $G=\sunitary{2}^2$, \ie the group $G$ acts on the manifold $M$ in question preserving its complex structure (and a holomorphic volume form) with generic orbits of codimension 1 (the \emph{principal orbits}). We shall adopt here this symmetric ansatz and look for $\sunitary{2}^2$--invariant solutions of the IIB system. The large symmetry group reduces \eqref{eq:IIBintro} to a system of non-linear first-order ODEs and in this paper we provide an essentially complete analysis of its solutions.

We state our two main existence results in the following two theorems. More complete statements are given in Theorems \ref{thm:Coho1:IIB:Deformed} and \ref{thm:Coho1:IIB:Resolved} respectively, which however require notation introduced in the rest of the paper. In the following results, $(\omega_\tu{C},\Omega_\tu{C})$ is the Calabi--Yau cone structure on the conifold, which gives rise to a family $(\omega_\tu{C},\Omega_\tu{C},\phi_\infty)$ of solutions to the IIB system \eqref{eq:IIBintro} by taking a constant dilaton $\phi_\infty\in \R$.

\begin{mtheorem}\label{mthm:IIB:Deformed}
Fix $\kappa>0$ and $\alpha'\neq 0$. Then there exists $p_\ast=p_\ast (\kappa,\alpha')\in \R$ with the following significance. 
\begin{enumerate}
\item There is a 1-parameter family of smooth solutions $\{ (\omega_p,\Omega_p,\phi_p)\}_{p\in\R}$ of the IIB system defined in sufficiently small tubular neighbourhoods of the zero-section in $T^\ast S^3$ satisfying $[e^{-2\phi_p}\Omega_p]=\kappa$ and $[d^c\omega]=\alpha'$ in $H^3(T^\ast S^3;\R)\simeq \R$.
\item For parameter values $p>p_\ast$ the solution $(\omega_p,\Omega_p,\phi_p)$ in (i) extends to a complete solution which is asymptotic at infinity to the cone solution $(\omega_\tu{C},\Omega_\tu{C},\phi_\infty)$ for some $\phi_\infty=\phi_\infty(p)\in \R$ that is monotonically increasing in $p$.
\item The solution $(\omega_p,\Omega_p,\phi_p)$ with $p=p_\ast$ is complete and has a different specific asymptotic geometry that we refer to as a complete CV--MN end. 
\item For parameter values $p<p_\ast$ the solution $(\omega_p,\Omega_p,\phi_p)$ in (i) is incomplete. 
\end{enumerate}
\end{mtheorem}

The solution in (iii) is explicit and was found in the physics literature by Maldacena--Nu\~nez \cite{Maldacena:Nunez} building on work by Chamseddine--Volkov \cite{Chamseddine:Volkov}. Borrowing notation already used in the physics literature, we refer to it as the \emph{CV--MN solution} on the deformed conifold. The CV--MN asymptotic conditions are defined more precisely in Definition \ref{def:CVMN:Ends}: they require $e^\phi\ra 0$ along the complete end and the metric to be of the form
\begin{equation}\label{eq:CV:MN:intro}
ds^2 + s g_{S^2} + g_{S^3}.
\end{equation}
This metric is defined on $(0,\infty)\times\Sigma$, with first factor parametrised by the variable $s$ and $\Sigma=\sunitary{2}^2/\triangle\unitary{1}$ thought of as an $S^3$--bundle over $S^2$; the family of metrics $s g_{S^2} + g_{S^3}$ on $\Sigma$ involves the choice of a connection on this bundle, which is the one induced by the Chern connection of $\mathcal{O}_{\C\PP^1}(1)\ra \C\PP^1\simeq S^2$.

It turns out (see Lemma \ref{lem:No:Smoothness:Resolutions}) that there cannot be any smooth $\sunitary{2}^2$--invariant solution of the IIB system on the resolved conifold (nor on the crepant resolution $K_{\mathbb{CP}^1 x \mathbb{CP}^1}$ of the $\mathbb{Z}_2$-quotient of the conifold, that a priori would also be a candidate to carry $\sunitary{2}^2$-invariant solutions of the IIB system). Instead we use the CV--MN solution to define a distinguished singular behaviour for solutions defined on exterior domains $\tu{C}_R = \{ r>R\}\subset \tu{C}$ in the conifold $\tu{C}$. Here $r$ is the radial distance computed using the Calabi--Yau cone metric on $\tu{C}$. We then obtain the following companion existence result for the complex structure of the conifold $\tu{C}$.

\begin{mtheorem}\label{mthm:IIB:Resolved}
Fix $b,\alpha'\neq 0$. Then there exist $q_\ast=q_\ast (b,\alpha')\in \R$, $R_0=R_0(\alpha',b)>0$ and a decreasing continuous function $R\co \R\ra (0,R_0)$ with the following significance. 
\begin{enumerate}
\item There is a 1-parameter family of solutions $\{ (\omega_q,\Omega_q,\phi_q)\}_{q\in\R}$ of the IIB system defined in sufficiently thin annuli $A_q = \{ R(q)<r<R'\}\subset \tu{C}$ with prescribed incomplete behaviour along the interior boundary $\{ r=R(q)\}$ and satisfying $[d^c\omega_q]=\alpha'$ in $H^3(C_{R(q)};\R)\simeq \R$ and $[\omega_q-\sigma]=b\in H^2(C_{R(q)};\R)\simeq \R$. Here $d\omega_q$ is independent of $q$ and $\sigma$ is an explicit 2-form with $d\sigma=d\omega_q$.
\item For parameter values $q>q_\ast$ the solution $(\omega_q,\Omega_q,\phi_q)$ in (i) extends to a forward complete solution on $\tu{C}_{R(q)}$ which is asymptotic at infinity to the model solution $(\omega_\tu{C},\Omega_\tu{C},\phi_\infty)$ for some $\phi_\infty=\phi_\infty(q)\in \R$ that is monotonically increasing in $q$.
\item The solution $(\omega_q,\Omega_q,\phi_q)$ with $q=q_\ast$ is forward complete on $\tu{C}_{R(q^\ast)}$ and has a complete CV--MN end at infinity. 
\item For parameter values $q<q_\ast$ the solution $(\omega_q,\Omega_q,\phi_q)$ in (i) is forward incomplete. 
\end{enumerate}
\end{mtheorem}

The prescribed incomplete behaviour along $\{ r=R(q)\}$ in part (i) is modelled on (a variation of) the incomplete end $s\ra 0$ of \eqref{eq:CV:MN:intro}. The solution in (iii) was also found by Maldacena--Nu\~nez \cite{Maldacena:Nunez} and we refer to it as the CV--MN solution on the resolved conifold (see \cite{PT} for the relation to complex geometry). 

\begin{remark*}
In Theorems \ref{mthm:IIB:Deformed} and \ref{mthm:IIB:Resolved} the parameter $\alpha'$ is identified with the cohomology class of $d^c\omega=-H$ and in fact determines it explicitly together with the choice of complex structure. In particular, assuming $\alpha'\neq 0$ implies that our solutions cannot be K\"ahler, while solutions with $\alpha'=0$ must necessarily be K\"ahler. In physical terms, we also note that our parameter $\alpha'$ corresponds to a $\tu{NS}5$--brane charge in Type IIB string theory and should not be confused with the slope parameter which appears in the Bianchi identity for the Hull--Strominger system.
\end{remark*}

Because they are asymptotic to the Candelas--de la Ossa metric on the conifold, the generic (forward) complete members of the family in Theorems \ref{mthm:IIB:Deformed} and \ref{mthm:IIB:Resolved} have the property that the Bismut connection has full holonomy $\tu{SU}(3)$. This is in contrast with the compact case \cite{GFJS,ABLS26}, where the holonomy of the Bismut connection of any non-K\"ahler Bismut--Hermitian--Einstein metric (BHE) is strictly contained in $\tu{SU}(n)$, where $n$ is the complex dimension. To our knowledge, Theorem \ref{mthm:IIB:Deformed} provides the first non-compact, complete, non-K\"ahler examples of the BHE system with full holonomy (\cf \cite{ALL26} in the compact case).

\begin{mtheorem}\label{mthm:Bismut}
There exist infinitely many complete, non-compact, non-K\"ahler, Bismut--Hermitian--Einstein metrics in complex dimension $3$ with full holonomy $\tu{SU}(3)$.
\end{mtheorem}

In fact, in Proposition \ref{propo:chs:hol} we also show that there are infinitely many complete examples in complex dimension 2 obtained via the so-called Callan--Harvey--Strominger ansatz \cite{Callan:Harvey:Strominger} (see below). In contrast to the 3-dimensional ones of Theorem \ref{mthm:IIB:Deformed}, these examples have at least two ends.

\begin{remark*}
The symmetric ansatz forces a specific normal form for the Hermitian form
$$
 \omega = \tfrac{1}{2}d\tilde \eta + \omega'.
$$
and the torsion of the Bismut connection
\begin{equation}\label{eq:dcomegaBfieldintro}
H := - d^c \omega = \alpha' H_0 + dB.
\end{equation}
Here, $\tilde \eta$ is an invariant 1-form on the locus $M^*$ of principal orbits and $\omega'$ is a coclosed but not closed decaying harmonic $2$-form with respect to the asymptotically conical Calabi--Yau metric on the (deformed) conifold due to Stenzel/Candelas--de la Ossa.
Furthermore, $H_0$ is co-closed with respect to the homogeneous Sasaki--Einstein structure on $\Sigma$, and the \emph{$b$-field} $B$ is of type $(2,0) + (0,2)$.
\end{remark*}

\begin{remark*}
 Exploiting the analysis of local solutions in Proposition \ref{prop:Coho1:IIB:Deformed:IVP}, one can show that in the limit $\alpha' \to 0$ suitable rescalings of the family of smooth solutions of Theorem \ref{mthm:IIB:Deformed} converges to the Stenzel asymptotically conical Calabi--Yau structure on $T^\ast S^3$ (with period $\kappa$ of the holomorphic volume form) together with a constant dilaton $\phi_\infty\in \R$. We expect the same statement to hold for the family of Theorem \ref{mthm:IIB:Resolved} (with the Candelas--de la Ossa Calabi--Yau metric instead of Stenzel's one and the parameter $b$ corresponding to the choice of K\"ahler class), although the convergence would only be on the complement of the zero-section $S$ because of the presence of singularities. This uses the fact that in Theorem \ref{mthm:IIB:Resolved}(i) the function $R(q)$ satisfies $\lim_{q\ra\infty} R(q)=0$. 
\end{remark*}

\begin{remark*}
Maldacena--Martelli \cite{Maldacena:Martelli}, based on numerical solutions and heuristic matched asymptotics expansions, predicted the existence of a 1-parameter family of solutions on $T^\ast S^3$ that are asymptotically conical, with the CV--MN solution appearing as a limit. Our Theorem \ref{mthm:IIB:Deformed} can be regarded as a rigorous confirmation of this physical expectation.
\end{remark*}

\begin{remark*}
Our existence results were motivated by an analogy with recently established results about the \emph{IIA system}, whose 6-dimensional solutions $(M,\omega,\Omega,\phi)$ are the dimensional reductions of circle-invariant $\tu{G}_2$--holonomy metrics on the total space of a circle bundle over $M$. In fact, the solutions of Theorems \ref{mthm:IIB:Deformed} and \ref{mthm:IIB:Resolved} are expected to be ``mirrors'' to the families $\mathbb{D}_7$ and $\mathbb{B}_7$ of circle-invariant comohomogeneity one $\tu{G}_2$--holonomy metrics constructed in \cite{FHN:Coho1}, whose circle reductions are $\sunitary{2}^2$--invariant solutions of the IIA system with, respectively, Ramond--Ramond 2-form flux on the resolved conifold and a $\tu{D}6$--brane wrapped on the zero-section $S^3$ of the deformed conifold \cite{Vafa,AMV}.
\end{remark*}

\subsection*{Plan of the paper} We conclude this introduction with a brief outline of the content of the paper. As already mentioned, our strategy is to impose $\sunitary{2}^2$ symmetry and study cohomogeneity one solutions of the IIB system.

In Section 2 we derive the ODE system that describes $\sunitary{2}^2$--invariant solutions of the IIB system on the set of principal orbits $\sunitary{2}^2/\triangle\unitary{1}$. The cohomogeneity one problem can be equivalently formulated as a first order system for a pair of functions $(u,\phi)$ or a single second-order ODE for the function $u$ with coefficients that in both cases depend on the choice of complex structure and the cohomology class $\alpha'$ of $d^c\omega=-H$. Along the way we also establish that:
\begin{enumerate}
\item the cohomogeneity one set-up we consider is the only one yielding non-K\"ahler cohomogeneity one solutions of the IIB system, extending work of Alonso--Salvatore \cite{Izar,Izar:correction};
\item under our symmetry assumption, the complex structure and holomorphic complex volume form $e^{-2\phi}\Omega$ can only be the one of the deformed or resolved conifold and that once this is fixed then $H=-d^c\omega$ is completely and explicitly determined by its cohomology class $\alpha'$;
\item every $\sunitary{2}^2$--invariant solution of the IIB system with the complex structure of the resolved conifold (or the crepant resolution $K_{\C\PP^1\times\C\PP^1}$ of the $\Z_2$--quotient of the conifold) cannot extend smoothly over the zero-section and in fact solves the more general version of the \emph{Bianchi identity} 
$dd^c\omega= - \delta$ sourced by a singular current $\delta$ supported on the zero-section.
\end{enumerate}

In Section 3 we discuss explicit singular solutions of the IIB system. We first consider the local model $\C\times (\C^2\setminus\{ 0\})$ (or more generally the complement of finitely many points in $\C^2$) with modified Bianchi identity $dd^c\omega= - \delta$ for $\delta$ a multiple of the current of integration along the holomorphic curve $\C\times \{ 0\}$. Such solutions arise from a simple construction first exploited in the physics literature by Callan--Harvey--Strominger \cite{Callan:Harvey:Strominger} reminiscent of the Gibbons--Hawking construction of 4-dimensional hyperk\"ahler metrics with a triholomorphic circle action. A calculation of the holonomy of the Bismut connection of these examples provides infinitely many BHE metrics in complex dimension $2$ with at least two ends and full holonomy $\tu{SU}(2)$. Attempting to fibre the simplest such local solution over $\C\PP^1$ using the Chern connection of $\mathcal{O}_{\C\PP^1}(1)$ then leads us to recover the explicit CV--MN solutions \cite{Maldacena:Nunez} of parts (iii) in Theorems \ref{mthm:IIB:Deformed} and \ref{mthm:IIB:Resolved}. We also give the definition of CV--MN complete and singular ends.

In Section 4 we establish parts (i) of Theorems \ref{mthm:IIB:Deformed} and \ref{mthm:IIB:Resolved} using the machinery of singular boundary initial value problems that is by now well-established in the cohomogeneity one literature.

Section 5 is the analytic heart of the paper where the proof of Theorems \ref{mthm:IIB:Deformed} and \ref{mthm:IIB:Resolved} is completed. The qualitative analysis of the families of local solutions constructed in Section 4 is based on two main results:
\begin{enumerate}
\item comparison results that allow us to establish the completeness of solutions lying ``above'' the CV--MN solution and the monotonicity of $\phi_\infty$ in parts (ii) of Theorems \ref{mthm:IIB:Deformed} and \ref{mthm:IIB:Resolved};
\item study of the ratio $\frac{u_\ast^{1+\epsilon}}{u}$ between a solution $u$ and the CV--MN solution $u_\ast$ to control the asymptotic growth of $u$ and deduce the asymptotically conical behaviour in parts (ii) and incompleteness statement in parts (iv) of our main theorems. 
\end{enumerate}

Finally, in Section 5.4 we prove Theorem \ref{mthm:Bismut} and collect some properties and observations about the families of
solutions produced in Theorems \ref{mthm:IIB:Deformed} and \ref{mthm:IIB:Resolved}. We speculate that these families provide a metric realisation of the conifold transition at the level of solutions to the IIB system.
We hope to provide a complete answer to this question in future work.

\subsection*{Acknowledgements} After the main results of this work were completed, we became aware of the overlapping work of F. Podest\`a and A. Raffero \cite{PR}, which establishes a version of our Theorem \ref{mthm:IIB:Deformed}, with an essentially equivalent proof, starting from the a priori more general BHE condition. We thank Fabio and Alberto for sharing their paper. The authors also thank Anna Fino for comments on an earlier draft of this paper and Jeff Streets, Vestislav Apostolov and Jason Lotay for interest in this work. LF would also like to thank Thomas Madsen, who was the first to mention to him the interest of considering the holonomy of the Bismut connection of these examples. The authors are grateful to the  Department of Computer Science and Technology of the University of Cambridge, for the hospitality during their visit in the summer 2023, when the results in Section 2 were completed. 

\section{The equations}

In this preliminary section we introduce the ODE system corresponding to $\sunitary{2}^2$--invariant solutions of the IIB system.

\subsection{The IIB system and generalized Ricci solitons}

Let $M$ be a smooth $6$-manifold. An $\tu{SU}(3)$--structure on $M$ is given by a pair $(\omega,\Omega)$, such that $\omega$ is a non-degenerate $2$-form, $\Omega$ is a locally totally decomposable complex $3$-form, with associated almost complex structure $J$, the pair $(\omega,\Omega)$ satisfies the pointwise constraints 
$$
\omega\wedge\Omega = 0, \qquad 2\omega^3=3\Real\Omega\wedge\Imag\Omega,
$$ 
and the induced symmetric tensor $g(\,\cdot\, ,\,\cdot\,) = \omega (\,\cdot\, , J\,\cdot\,)$ is positive definite. Recall that $\Omega$ determines an almost complex structure $J$ on $M$, such that $T^{0,1}_JM$ is the annihilator of $\Omega$ by contraction. Alternatively, a $1$-form $\alpha$ is of type $(1,0)$ with respect to $J$ if and only if $\alpha\wedge\Omega=0$. In particular, the first pointwise constraint above implies that $\omega$ is of type $(1,1)$. For $k$-form $\beta$, we will use the notation $J\beta = (-1)^k \beta(J\,\cdot\,,\dots,J\,\cdot\,)$.

In this paper we are interested in $\tu{SU}(3)$--structures on a six-dimensional manifold solving a natural system of PDE, as given in the following definition.

\begin{definition}\label{def:IIBsystem}
Let $\delta$ be an exact $4$-form on $M$. We say that an $\tu{SU}(3)$--structure $(\omega,\Omega)$ solves the \emph{IIB system with source $\delta$} if the following PDE system is satisfied
\begin{equation}\label{eq:IIB}
d(e^{-2\phi}\Omega)=0, \qquad d(e^{-2\phi}\omega^2)=0, \qquad dd^c\omega= - \delta.
\end{equation}
\end{definition}

The most basic consequence of the IIB system is the integrability of the complex structure $J$, as it follows from the first equation that
$$
d\Omega = 2d\phi \wedge \Omega.
$$
The associated complex manifold is furthermore Calabi--Yau, in the sense that it admits a global holomorphic trivialisation of the canonical bundle, that is, a global holomorphic volume form $e^{-2\phi}\Omega$. The second equation in the system is often known as the \emph{conformally balanced} condition. The specific choice of conformal factor in this equation has important consequences on the Bismut connection of the solution, namely
$$
\nabla^B = \nabla^g - \frac{1}{2}g^{-1}d^c\omega
$$
has holonomy contained in $\tu{SU}(3)$, where $\nabla^g$ denotes the Levi-Civita connection of $g$. Due to its origins in string theory, the torsion $3$-form $H$ of the Bismut connection of a solution has a special significance, in particular solving the equations
$$
H = - d^c\omega, \qquad dH = \delta.
$$
The source $\delta$ on the right hand side of the \emph{Bianchi identity} (second equation above) has to be thought of as fixed, and naturally dictated by the topology and geometry of the background manifold $M$. We will mainly focus on the most canonical case $\delta = 0$, but, as we will see, our analysis will naturally lead us to solutions where $\delta$ is a distributional source located on a compact holomorphic curve on the manifold. This is in fact the type of behaviour expected from string theory, where the support of the distributional source relates to non-perturbative effects, namely, D5--branes or NS5--branes. In the first case, $H$ is interpreted as the Ramond--Ramond field strength, with natural balanced metric $\omega' = e^{-\phi}\omega$, while in the latter $H$ is the more familiar NS--flux, with conformally balanced metric $\omega$ as presented above \cite{Tomasiello,Martucci}.

Further mathematical motivation for the IIB system in the case $\delta = 0$ is given by generalized geometry, where solutions can be interpreted as steady solitons for the generalized Ricci flow \cite{GFSBook}. Notice that, in particular, any such solution is Bismut--Hermitian--Einstein in the sense of \cite{GFJS}, that is, it is a solution of the system
$$
\rho^B(\omega) = 0, \qquad dd^c\omega = 0,
$$
where $\rho^B(\omega)$ is the \emph{Bismut Ricci form} of the Hermitian structure. Recall here that $\rho^B(\omega)$ is defined as the imaginary constant times the curvature of the Bismut connection on the anti-canonical bundle. This fact is a direct consequence of the holonomy reduction produced by the conformally balanced equation. We summarise the basic upshots of this observation in the next proposition, following closely \cite{GFJS,GFSBook}.

\begin{prop}\label{prop:soliton}
Let $(\omega,\Omega)$ be a solution of the \emph{IIB system} with vanishing source $\delta = 0$. Then, the metric $g$ is a \emph{steady gradient generalized Ricci soliton}, that is, it solves the system of equations
\begin{gather}\label{eq:soliton}
 \begin{split}
\tu{Ric}_g  - \tfrac{1}{4} H^2 + \tfrac{1}{2} L_{\theta^{\sharp}} g =& 0,\\
d^* H +  i_{\theta^{\sharp}} H =& 0,
 \end{split}
\end{gather}
where $H = - d^c\omega$ and  $\theta := - d^* \omega \circ J = 2d\phi$ is the Lee form. Consequently, one has
\begin{equation}\label{eq:conservedsolitons}
|H|^2 - 2\Delta \phi - 4 |d\phi|^2 = 0,
\end{equation}
and hence compact solutions of \eqref{eq:IIB} with $\delta = 0$ are necessarily K\"ahler Ricci-flat and have $d\phi = 0$.
\end{prop}

\begin{proof}
The first part of the statement is a direct consequence of the condition $\rho^B(\omega) = 0$ and \cite[Proposition 8.10]{GFSBook}. Equation \eqref{eq:conservedsolitons} follows from the fact that any solution of the IIB system \eqref{eq:IIB} is a solution of the Killing spinor equations in generalized geometry, and hence it has vanishing \emph{generalized scalar curvature} (see \cite[Proposition 2.23]{SGFLS}):
$$
\tu{S}_g  - \tfrac{1}{4} |H|^2 - 2 \Delta \phi - |d\phi|^2 = 0,
$$
where $\Delta$ denotes the Hodge Laplacian. The required identity follows now subtracting the trace of the first equation in \eqref{eq:soliton} (cf. \cite[Proposition 4.33]{GFSBook}). For $M$ compact, the last part of the statement follows from equation \eqref{eq:conservedsolitons} and integration by parts against $e^{-2\phi}\vol_g$.
\end{proof}

An alternative proof of the rigidity of compact solutions of the IIB system is via geometric invariant theory, and follows from \cite[Corollary 4.5]{GFJS}, combining the Bismut--Hermitian--Einstein equations with the existence of a global holomorphic volume form (cf. \cite[Proposition 8.32]{GFSBook}). The existence of non-compact solutions will be, indeed, the main focus of the present work. For further studies on the classification and existence of generalized Ricci solitons, we refer to \cite{GFSBook,StreetsUstinovskiy} and references therein.

\subsection{Cohomogeneity one ansatz and AC Calabi--Yau metrics}

In this paper we are interested in complete non-compact solutions of the IIB system \eqref{eq:IIB} that are invariant under a cohomogeneity one action of a compact Lie group $G$, \ie $G$ acts on $M$ preserving $(\omega,\Omega,\phi)$ and so that the orbit space $M/G$ is $1$-dimensional. In particular, in the next two sections we will reduce the PDE system \eqref{eq:IIB} to an ODE system, which we will carefully study in the subsequent sections of the present work. Our starting point was to make the \emph{a priori} assumption that the underlying smooth manifold admits a background cohomogeneity one asymptotically conical (AC) Calabi--Yau metric, which we expect will provide the asymptotics of our complete solutions at infinity. In fact, building on the main results in \cite{Izar,Izar:correction}, we will show in Remark \ref{def:cohoonebalanced} and Proposition \ref{prop:su3} below that
this assumption is very reasonable, and can be replaced by finiteness of the fundamental group. The present section is devoted to give some background material on cohomogeneity one AC Calabi--Yau metrics in six dimensions.

We start by recalling the structure of a cohomogeneity one manifold $M$, following \cite{Benard} (see also \cite{Izar} and references therein). Let $G$ be a compact Lie group acting on the left on a connected smooth manifold $M$, with an orbit of codimension one. We denote by $K \subset G$ the stabilizer of a point in this orbit. Besides the case $M=S^1\times G/K$, the manifold $M$ contains an open dense subset, the space $M^\ast$ of \emph{principal orbits}, $G$--equivariantly diffeomorphic to $(0,1) \times G/K$. If $M$ is non-compact, then either $M=M^\ast$ or the orbit space is diffeomorphic to $[0,1)$ and $M = M^\ast \sqcup G/K_0$, where the \emph{singular orbit} $G/K_0$ sits over the boundary point $t=0$ of the orbit space. As before, $K_0 \subset G$ denotes the stabilizer of a point in the singular orbit. If $M$ is compact, then the orbit space is diffeomorphic to $[0,1]$ and we have two singular obits $G/K_0$ and $G/K_1$ over the two endpoints. In the non-compact case $M$ has the topology of a disc bundle over the singular orbit $G/K_0$, while in the compact case $M$ is the union of two such disc bundles glued along their common boundary.

In view of finding non-K\"ahler cohomogeneity one solutions to \eqref{eq:IIB} with $\delta = 0$, the most natural choice is a non-compact manifold $M$ with a cohomogeneity one action of either $G=\sunitary{2}^2$, with principal orbit $\Sigma:=G/K=\sunitary{2}^2/\triangle\unitary{1}\simeq S^2\times S^3$, or $G=\sunitary{3}$, with principal orbit $\Sigma:=G/K=\sunitary{3}/\sunitary{2} \simeq S^5$, or finite quotients thereof (see Remark \ref{def:cohoonebalanced}). The requirement that $M$ is non-compact is necessary by Proposition \ref{prop:soliton}. The case $G=\sunitary{3}$ will be ruled out by our analysis in Proposition \ref{prop:su3}, and hence we shall focus on the first case. The existence of non-compact cohomogeneity one Calabi--Yau metrics in six dimensions with symmetry $G=\sunitary{2}^2$ is summarised in the next result. Here, by Calabi--Yau metric we mean a Riemannian metric on a $2n$-dimensional manifold with holonomy contained in $\sunitary{n}$. To give a more precise statement, consider the Calabi--Yau cone with underlying (singular) complex manifold the affine singularity 
$$
\tu{C}=\{ z_1^2+z_2^2+z_3^2+z_4^2 = 0\} \subset\C^4.
$$
The Calabi--Yau metric on $\tu{C}$ is a cone over the homogeneous Sasaki--Einstein structure on $\Sigma=\sunitary{2}^2/\triangle\unitary{1}$, to be recalled later (see \eqref{eq:Conifold}). Using terminology popularised in the physics literature \cite{Candelas:delaOssa}, we refer to this cone as the \emph{conifold}. A complete non-compact Calabi--Yau metric is called \emph{asymptotically conical} (AC) if, away from a compact set, the Calabi--Yau structure $(\omega,\Omega)$ decays to the conical one on (a finite quotient of) $\tu{C}$ with polynomial decay rate.

\begin{theorem}\label{thm:ACCY}
Let $(M^6,\omega,\Omega)$ be a non-compact $\sunitary{2}^2$--invariant AC Calabi--Yau manifold. Then, $M$ is asymptotic to the conifold $\tu{C}$ or its quotient $\tu{C}/\Z_2$ (and in particular it has principal orbits $\sunitary{2}^2/\triangle\unitary{1}$) and either:

\begin{enumerate}
\item $M= T^*S^3$, with singular orbit $\sunitary{2}^2/\triangle\sunitary{2}\simeq S^3$, complex structure given by the \emph{smoothing} of the conifold
$$
\{ z_1^2+z_2^2+z_3^2+z_4^2=t\},
$$ 
depending on the complex parameter $t\in\C^\ast$, and with $g$ an element in the family of Stenzel's AC Calabi--Yau metrics \cite{Candelas:delaOssa,Stenzel}. The phase of $t$ corresponds to the choice of phase of the holomorphic volume form $\Omega$ while the norm of the smoothing parameter $t$ is a scale parameter.

\item $M$ is the total space of the small resolution 
$$
\mathcal{O}(-1)\oplus\mathcal{O}(-1)\ra \C\PP^1,
$$
with singular orbit $\sunitary{2}^2/\unitary{1}\times \sunitary{2} \simeq S^2$ or $\sunitary{2}^2/\sunitary{2}\times\unitary{1} \simeq S^2$, and with metric an element in the 1-parameter family of Candelas--de la Ossa AC Calabi--Yau metrics \cite{Candelas:delaOssa}.

\item $M$ is the total space of $K_{\C\PP^1\times\C\PP^1}$, the canonical line bundle of $\C\PP^1\times\C\PP^1$, with singular orbit $\sunitary{2}^2/T^2 \simeq S^2 \times S^2$, and with metric an element in the 2-parameter family of AC Calabi--Yau metrics asymptotic to $\tu{C}/\Z_2$ \cite{Calabi:Ansatz,PandoZayas:Tseytlin}. The AC Calabi--Yau metrics are parametrised by the K\"ahler class $[\omega]\in H^2(M)$.
\end{enumerate}
\proof
By \cite[\S 2.2]{Sparks:SE} the only Calabi--Yau cones in complex dimension 3 with homogeneous Sasaki--Einstein cross-section are the conifold $\tu{C}$, $\C^3$ or finite quotients thereof. The classification of AC Calabi--Yau metrics by Conlon--Hein \cite{Conlon:Hein} in particular implies, see \cite[Table 1]{Conlon:Hein}, that all AC Calabi--Yau metrics asymptotic to $\tu{C}$ or $\tu{C}/\Z_2$ are $\sunitary{2}^2$--invariant and included in the list of the theorem.  
\endproof
\end{theorem}

\begin{remark}\label{rem:ACCY3}
In the second case, if we fix the isomorphism with the cone $\tu{C}$ at infinity, then there are really two small resolutions distinguished by the sign of $\langle \omega-\omega_{\tu{C}}, [S^2_\infty]\rangle$, where $\omega_{\tu{C}}$ is the conical K\"ahler form and the $[S^2_\infty]$ is the generator of $H^2(\tu{C};\Z)\simeq\Z$ (\cf Section \ref{sec:Parameters}). The map between the two small resolutions is called the flop and in this cohomogeneity one framework it is induced by the map that exchanges the two factors of $\sunitary{2}^2$. In the third case, if $[\omega]$ is compactly supported the metric is due to Calabi \cite{Calabi:Ansatz}, while the general case was first considered in \cite{PandoZayas:Tseytlin}.
\end{remark}

\begin{remark}\label{def:cohoonebalanced}
According to \cite[\S 3 and Remark 4.1]{Izar}, a \emph{simply connected} non-compact (not necessarily complete) $G$--invariant cohomogeneity one $6$-manifold with principal orbits $G/K$ which admits a $G$--invariant $\sunitary{3}$-structure $(\omega,\Omega)$, must satisfy either
\begin{enumerate}
\item $\operatorname{Lie} G = \operatorname{Lie} \sunitary{2}^2, \qquad \operatorname{Lie} K = \operatorname{Lie} \triangle\unitary{1}$, or
\item $\operatorname{Lie} G = \operatorname{Lie} \sunitary{3}, \qquad \operatorname{Lie} K = \operatorname{Lie} \sunitary{2}$.
\end{enumerate}
Furthermore, \cite[Theorem A]{Izar:correction} states that only in the first case there are invariant non-K\"ahler solutions of the equations $d\Omega= 0=d\omega^2$. In fact, the arguments in \cite[\S 4.2]{Izar} show that in case (ii) any invariant $\sunitary{3}$--structure $(\omega,\Omega)$ satisfying the (less restrictive) equations $d(e^{-2\phi}\Omega)=0=d(e^{-2\phi}\omega^2)$ for some positive function $e^\phi$ must also satisfy $d(e^{-\phi}\omega)=0$, \ie it is conformally K\"ahler. As we will see in Proposition \ref{prop:su3} below, imposing the further condition $dd^c\omega=0$ forces that $\phi$ be constant and therefore that $(\omega,\Omega)$ be a K\"ahler Ricci-flat Calabi--Yau structure. 
\end{remark}

\subsection{Integrability conditions and the Bianchi identity}\label{sec:ODEBI}

Let $M$ be a non-compact cohomogeneity one 6-dimensional manifold, as considered in Theorem \ref{thm:ACCY} (the case with symmetry $G=\sunitary{3}$ will be briefly discussed at the end of this section). Our first goal is to rewrite the IIB system \eqref{eq:IIB} on $M^\ast=\R\times\Sigma$ as an ODE system. In this section we will focus on the first and third equations in \eqref{eq:IIB}, corresponding (roughly) to the integrability of the complex structure and the Bianchi identity, and postpone the presentation of the complete ODE system to the next section. We will use the analysis of $\sunitary{2}^2$--invariant $\sunitary{3}$--structures on $M^\ast$ described in \cite[
\S 2]{Foscolo:Haskins}. 

We build on the fact that an invariant $\sunitary{3}$--structure on $M^*$ is equivalent to a 1-parameter family of invariant $\sunitary{2}$--structures on $\Sigma=\sunitary{2}^2/\triangle\unitary{1}$. In order to describe this family more explicitly, we introduce the $\sunitary{2}$--structure corresponding to the canonical homogeneous Sasaki--Einstein geometry of $\Sigma$. From the analysis of the isotropy representation of $\Sigma=\sunitary{2}^2/\triangle\unitary{1}$ one finds, see \cite[Lemma 2.10]{Foscolo:Haskins}, that there is a unique invariant 1-form $\eta^{se}$ up to scale on $\sunitary{2}^2/\triangle\unitary{1}$ and 4 invariant 2-forms $\omega_i^{se}$, $i=0,1,2,3$, satisfying
\begin{equation}\label{eq:Homogeneous:SE}
\begin{gathered}
\omega^{se}_i\wedge\omega_j^{se}=\eta_{ij}\, \omega_1^{se}\wedge\omega_1^{se},\\
d\eta^{se}=2\omega_1^{se}, \qquad d\omega_0^{se}=0=d\omega_1^{se},\\
d\omega^{se}_2=-3\eta^{se}\wedge\omega^{se}_3, \qquad d\omega^{se}_3=3\eta^{se}\wedge\omega^{se}_2,
\end{gathered}
\end{equation}
with $(\eta_{ij})=\tu{diag}(-1,1,1,1)$. In order to be completely explicit, fix bases $\eta_1,\eta_2,\eta_3$ and $\eta'_1,\eta'_2,\eta'_3$ of left-invariant $1$-forms on the two factors of $\sunitary{2}$ such that $d\eta_i=2\eta_j\wedge\eta_k$ and $d\eta'_i=2\eta_j'\wedge\eta'_k$ for $(ijk)$ a cyclic permutation of $(123)$. Then
\begin{equation}\label{eq:Homogeneous:SE:LeftInvariant}
\begin{gathered}
\eta^{se} = \tfrac{2}{3}(\eta_1-\eta'_1),\\
\omega_1^{se} = \tfrac{2}{3}(\eta_2\wedge\eta_3 - \eta'_2\wedge\eta'_3), \qquad \omega_0^{se} = \tfrac{2}{3}(\eta_2\wedge\eta_3 + \eta'_2\wedge\eta'_3),\\
\omega_2^{se} = -\tfrac{2}{3}(\eta_2\wedge\eta'_2 + \eta_3\wedge\eta'_3), \qquad \omega_3^{se} = \tfrac{2}{3}(\eta_2\wedge\eta'_3 + \eta'_2\wedge\eta_3).
\end{gathered}
\end{equation}
Note that $\omega_0^{se}$ and $\eta^{se}\wedge\omega^{se}_0$ are, respectively, a harmonic 2-form and 3-form on $(\Sigma,g^{se})$ representing in cohomology generators for the 1-dimensional cohomology $H^2(\Sigma;\R)$ and $H^3(\Sigma;\R)$. Here $g^{se}$ is the Sasaki--Einstein Riemannian metric defined by the $\sunitary{2}$--structure $(\eta^{se},\omega_1^{se},\omega_2^{se},\omega_3^{se})$.

\begin{remark}\label{rmk:Cohomology:generators}
More precisely, one can calculate that $3[\omega^{se}_0]$ and $\frac{9}{2}[\eta^{se}\wedge\omega_0^{se}]$ generate $H^2(\Sigma;2\pi\Z)$ and, respectively, $H^3(\Sigma;4\pi^2\Z)$.
\end{remark}

\begin{remark}\label{rem:conifoldexp}
The conditions in \eqref{eq:Homogeneous:SE} actually imply that
\begin{equation}\label{eq:Conifold}
\omega_\tu{C} = rdr\wedge\eta^{se}+r^2\omega^{se}_1, \qquad \Omega_\tu{C} = (dr +i r\eta^{se})\wedge r^2 (\omega_2^{se}+i\omega_3^{se})
\end{equation}
defines a conical Calabi--Yau structure, the conifold Calabi--Yau structure on the cone $\tu{C}=\tu{C}(\Sigma)$.
\end{remark}

Let $(\omega,\Omega)$ be a $\sunitary{2}^2$-invariant $\sunitary{3}$--structure on $M^*$. Then, $(\omega,\Omega)$ is equivalent to a 1-parameter family $(\eta,\omega_1,\omega_2,\omega_3)$ of invariant $\sunitary{2}$--structures on $\Sigma=\sunitary{2}^2/\triangle\unitary{1}$, as follows: if we choose a parameter $t$ on $\R$ corresponding to an arc-length parameter on the geodesic meeting all principal orbits orthogonally, we have
\[
\omega = dt\wedge\eta + \omega_1, \qquad \Omega = (dt+i\eta)\wedge(\omega_2+i\omega_3).
\]
Here $\eta$ is a nowhere vanishing 1-form on $\Sigma$, the $4$-dimensional distribution $\ker\eta$ is oriented and $(\omega_1,\omega_2,\omega_3)$ is a basis of $\Lambda^2_+\ker\eta^\ast$, consisting of non-degenerate $2$-forms satisfying $\omega_i\wedge\omega_j =\delta_{ij}\,\omega_1^2$. It follows (see \cite[Proposition 2.11]{Foscolo:Haskins}) that any $1$-parameter family of invariant $\sunitary{2}$--structures $(\eta,\omega_1,\omega_2,\omega_3)$ inducing the given orientation on $\Sigma$ is given by two positive functions $\lambda,\mu$ and an $\tu{SO}_0(1,3)$--valued function $A$ of the independent variable $t$ so that $\eta = \lambda \eta^{se}$ and the triple $\mu^{-1}(\omega_1,\omega_2,\omega_3)$ is given in terms of the basis $\omega_0^{se},\dots,\omega_3^{se}$ by the last three columns of $A$. Moreover, $\mu$ times the first column of $A$ defines an $\sunitary{2}^2$--invariant $2$-form $\omega_0$ on $\ker\eta$ that satisfies $\omega_0\wedge\omega_i=0$ for $i=1,2,3$ and $\omega_0^2=-\omega_1^2$.

With the basic geometric setup in place, we start by analyzing the first equation in the IIB system \eqref{eq:IIB}, which gives integrability of the complex structure and existence of a holomorphic volume form, namely:
\begin{equation}\label{eq:Coho1:IIB:dOmega}
d(e^{-2\phi}\Omega) = 0.
\end{equation}
In the present invariant setting, this equation becomes the evolution equations
\begin{subequations}
\begin{equation}\label{eq:Coho1:IIB:dOmega:Evolution}
\begin{gathered}
e^{2\phi}\partial_t \left( e^{-2\phi}\eta\wedge\omega_2\right) - d\omega_3 = 0= e^{2\phi}\partial_t \left( e^{-2\phi}\eta\wedge\omega_3\right) + d\omega_2,
\end{gathered}
\end{equation}
coupled to the constraints
\begin{equation}\label{eq:Coho1:IIB:dOmega:Static}
d(\eta\wedge\omega_2)=0=d(\eta\wedge\omega_3),
\end{equation}
\end{subequations}
where in the previous formulae $d$ denotes the exterior differential on $\Sigma$. Note that these equations are invariant under the discrete symmetry (this is the symmetry $\tau_1\circ\tau_2$ of \cite[Remark 2.2]{Foscolo:Haskins})
\begin{equation}\label{eq:Coho1:IIB:Discrete:Symmetry}
(t,\eta,\omega_1,\omega_2,\omega_3, e^\phi) \longmapsto (-t, \eta,-\omega_1,-\omega_2,\omega_3, e^\phi),
\end{equation}
which corresponds to a composition of the time reversal $t\mapsto -t$ with the antiholomorphic change $(\omega,\Omega)\mapsto (-\omega,-\overline{\Omega})$ 

\begin{lemma}\label{lem:Coho1:IIB:Fterm}
Up to the discrete symmetry \eqref{eq:Coho1:IIB:Discrete:Symmetry}, a constant change of phase of the complex volume form and the action of the normaliser of $\triangle\unitary{1}$, $\sunitary{2}^2$--invariant solutions to \eqref{eq:Coho1:IIB:dOmega} inducing the same orientation on the hypersurfaces $\{ t=\tu{const}\}$ as the homogeneous Sasaki--Einstein structure are of the form
\begin{equation}\label{eq:normalomega}
\omega_1 = u \omega_1^{se} + u_0 (\nu_3\omega_0^{se}+\nu_0\omega_3^{se}), \qquad \omega_2 = \mu \omega_2^{se}, \qquad \omega_3 = \mu (\nu_0\omega_0^{se}+\nu_3\omega_3^{se}),
\end{equation}
for functions $\lambda,\mu, \nu_0,\nu_3, u, u_0$ satisfying the constraints 
$$
\nu_3^2-\nu_0^2=1, \qquad u^2-u_0^2 = \mu^2, \qquad \lambda,\mu,\nu_3>0, \qquad \nu_0\leq 0.
$$
Furthermore, assuming \eqref{eq:normalomega}, integrability of the complex structure $J$ determined by $\Omega$ is equivalent to the ODE system
\begin{equation}\label{eq:integrableJ}
\lambda\,\partial_t\nu_0 +3\nu_0\nu_3=0, \qquad \lambda\,\partial_t\nu_3 + 3 \nu_3^2=3,
\end{equation}
and the equation \eqref{eq:Coho1:IIB:dOmega} is equivalent to
\[
\begin{gathered}\label{eq:dOmeganormal}
\partial_t (e^{-2\phi}\lambda\mu\nu_0)=0, \qquad e^{2\phi}\partial_t\left( e^{-2\phi}\lambda\mu \nu_3\right) - 3\mu =0=e^{2\phi}\partial_t\left( e^{-2\phi}\lambda\mu \right) - 3\mu\nu_3.
\end{gathered}
\]
\proof
The analysis is similar to \cite[Proposition 2.23]{Foscolo:Haskins}. The normalisation on the orientation of the level sets of $t$ fixes the inequality $\lambda>0$. We first consider the static constraints \eqref{eq:Coho1:IIB:dOmega:Static}: since $d(\eta^{se}\wedge \omega_i^{se})= 2\omega_1^{se}\wedge\omega_i^{se}$ for all $i$, the constraints are equivalent to $\omega_2$ and $\omega_3$ having no component along $\omega_1^{se}$. Secondly, we consider the projections of the two evolution equations in \eqref{eq:Coho1:IIB:dOmega:Evolution} that involve $\omega_2$ and $\omega_3$ along the $\eta^{se}\wedge\omega_0^{se}$--component: they imply that the phase of the $\omega_0^{se}$--component of $\omega_2+i\omega_3$ is constant. Therefore, by a constant change of phase of the complex volume form $\Omega$, we can assume that $\omega_2$ has no $\omega^{se}_0$--component and that the $\omega^{se}_0$--component of $\omega_3$ is non-positive. Finally, the action of the normaliser of $\triangle\unitary{1}$ in $\sunitary{2}^2$, corresponding to an action by equivariant diffeomorphisms on $\Sigma$, acts as a rotation in the $(\omega_2^{se},\omega_3^{se})$--plane. Up to this action we can therefore assume that at a point $t=t_0$ the form $\omega_2$ has also no $\omega_3^{se}$--component and the $\omega_3^{se}$--component of $\omega_3$ is positive. One then checks that the evolution equations that involve $\omega_2$ and $\omega_3$ preserve this choice of normalisation. Finally, the discrete symmetry \eqref{eq:Coho1:IIB:Discrete:Symmetry} allows one to assume that $\omega_2$ is a positive multiple of $\omega_2^{se}$. These facts imply the form of $\omega_2$ and $\omega_3$ in \eqref{eq:normalomega}, while the form of $\omega_1$ follows from the $\tu{SO}_0(1,3)$--relations.

As for the integrability of the complex structure assuming \eqref{eq:normalomega}, we work with the parameter $s$ defined by $\lambda\, ds = dt$ and check that
$$
\Omega' := \lambda^{-1} \mu^{-1}\Omega = (ds+i\eta^{se})\wedge \left(\omega_2^{se} + i(\nu_0\omega_0^{se}+\nu_3\omega_3^{se})\right)
$$
satisfies $d\Omega'=\gamma\wedge\Omega'$ for a complex $1$-form $\gamma$. Without loss of generality we can assume that $\gamma = f ds$ for a complex function $f = f(s)$, and then a direct calculation shows that $d\Omega'=\gamma\wedge\Omega'$ is equivalent to
$$
f = - 3 \nu_3, \qquad 3 - \partial_s \nu_3 = - \nu_3 f, \qquad \partial_s \nu_0 = f \nu_0, 
$$
hence implying \eqref{eq:integrableJ}. The final system of ODE in the statement, equivalent to $d(e^{-2\phi}\Omega)=0$, follows from direct computations.
\endproof
\end{lemma}

In the sequel, we will assume the normal form \eqref{eq:normalomega} for the $\sunitary{2}$--structure $(\eta,\omega_1,\omega_2,\omega_3)$. In the next result we record the action of the complex structure $J$ on the homogeneous $\sunitary{2}$--structure $(\eta^{se},\omega_1^{se},\omega_2^{se},\omega_3^{se})$ and $\omega_0^{se}$.

\begin{lemma}\label{lem:Jformula}
The action of the complex structure $J$ determined by $\Omega$ on $(\eta^{se},\omega_0^{se},\omega_1^{se},\omega_2^{se},\omega_3^{se})$ is given by
\begin{equation}\label{eq:Jexp}
\begin{gathered}
Jdt=\lambda\eta^{se}, \qquad J\eta^{se}=-\lambda^{-1}dt,\\
J\omega_1^{se}=\omega_1^{se}, \qquad J\omega_2^{se}=-\omega_2^{se},\\
J\omega_0^{se} = \left( (\nu_3^2 + \nu_0^2)\omega_0^{se} + 2\nu_0 \nu_3 \omega_3^{se}\right), \qquad J\omega_3^{se} = \left( -2\nu_0\nu_3\omega_0^{se} -(\nu_3^2 + \nu_0^2) \omega_3^{se}\right),
\end{gathered}
\end{equation}

\proof

With respect to the almost complex structure $J$ we have that $dt+i\lambda\eta^{se}$ is of type $(1,0)$, and hence the first line in \eqref{eq:Jexp} follows. As for the action of $J$ on $\omega_j^{se}$, we note that
$$
\omega := (\omega_0,\omega_1,\omega_2,\omega_3) = \mu \omega^{se} \cdot A
$$
where $\omega^{se} := (\omega_0^{se},\omega_1^{se},\omega_2^{se},\omega_3^{se})$ and $A \in \tu{SO}_0(1,3)$ is given by
$$
\mu A = \left( \begin{array}{cccc}
 u\nu_3 & u_0\nu_3 & 0 & \mu \nu_0 \\
 u_0 & u & 0 & 0 \\
 0 & 0 & \mu & 0 \\
 u\nu_0 & u_0\nu_0 & 0 & \mu \nu_3
\end{array}\right),
$$
The complex structure $J$ defines a transverse almost complex structure on $\Sigma$ such that $\omega_0$ and $\omega_1$ are of type $(1,1)$ and $\omega_2+i\omega_3$ is of type $(2,0)$. Hence
\[
J\omega = \omega P
\]
where $P$ is the diagonal matrix $(1,1,-1,-1)$. Consequently,
$$
J\omega^{se} = \mu^{-1} J\omega A^{-1} = \mu^{-1} \omega P A^{-1} = \omega^{se} A P A^{-1}.
$$
A direct calculation now shows that
$$
A P A^{-1} = \left( \begin{array}{cccc}
1 + 2 \nu_0^2 & 0 & 0 & - 2\nu_0\nu_3 \\
0 & 1 & 0 & 0 \\
 0 & 0 & -1 & 0 \\
2 \nu_0\nu_3 & 0 & 0 & 1 - 2 \nu_3^2
\end{array}\right)
$$
and the result follows from the identity $1 + \nu_0^2 = \nu_3^2$.
\endproof
\end{lemma}

We study next the Bianchi identity $dd^c\omega=0$, for the Hermitian $(1,1)$-form
\[
\omega = \lambda\, dt\wedge\eta^{se} + u\, \omega_1^{se} + u_0\, (\nu_3 \omega_0^{se}+\nu_0\omega_3^{se}).
\]
We calculate
\[
d\omega = (\partial_tu - 2\lambda)\, dt\wedge\omega_1^{se} + \partial_t (u_0 \nu_3)\, dt\wedge\omega_0^{se} + \partial_t(u_0 \nu_0)\, dt\wedge\omega_3^{se} + 3u_0 \nu_0\, \eta^{se}\wedge\omega_2^{se}, 
\]
and therefore, using Lemma \ref{lem:Jformula}, 
\begin{equation}\label{eq:dcomegaexp}
\begin{split}
d^c\omega & = \lambda (\partial_t u - 2\lambda)\, \eta^{se}\wedge\omega_1^{se} + \lambda \partial_t (u_0\nu_3)\,  \eta^{se}\wedge\left( (\nu_3^2 + \nu_0^2)\omega_0^{se} + 2\nu_0 \nu_3 \omega_3^{se}\right)\\ 
& - \lambda \partial_t(u_0\nu_0)\, \eta^{se}\wedge  \left( 2\nu_0\nu_3\omega_0^{se} +(\nu_3^2 + \nu_0^2) \omega_3^{se}\right) + 3\lambda^{-1} u_0\nu_0 \, dt\wedge\omega_2^{se}\\
& = \lambda (\partial_t u - 2\lambda)\, \eta^{se}\wedge\omega_1^{se} + \lambda \left( (\nu_3^2+\nu_0^2) \partial_t (u_0\nu_3) - 2\nu_0\nu_3 \partial_t (u_0\nu_0)\right) \eta^{se}\wedge \omega_0^{se}\\
& +\lambda \left( 2\nu_0\nu_3 \partial_t (u_0\nu_3) - (\nu_3^2+\nu_0^2) \partial_t (u_0\nu_0)\right) \eta^{se}\wedge \omega_3^{se} + 3\lambda^{-1}u_0\nu_0\, dt\wedge\omega_2^{se}.
\end{split}
\end{equation}

\begin{lemma}\label{lem:BI}
Assume that the almost complex structure $J$ is integrable, that is, \eqref{eq:integrableJ} is satisfied. Then, the Bianchi identity $dd^c\omega=0$ is equivalent to the existence of a constant $\alpha' \in \R$ such that 
\begin{subequations}
\begin{align}
\partial_tu - 2\lambda & = 0 \label{eq:Coho1:IIB:Bianchia}\\
\nu_3\lambda \partial_tu_0 + 3 u_0\nu_0^2 & = \tfrac{9}{2}\alpha'
\end{align}
\end{subequations}
Furthermore, provided that these two equations are satisfied, one has
\begin{equation}\label{eq:dcomegaBfield}
d^c \omega = \tfrac{9}{2}\alpha' \eta^{se}\wedge \omega_0^{se} - dB
\end{equation}
where 
\begin{equation}\label{eq:Bfield}
B = \tfrac{1}{3} \lambda \left( \nu_0 \partial_t u_0 - u_0 \partial_t \nu_0\right) \omega_2^{se}.
\end{equation}
Consequently, the cohomology class of the closed 3-form $H = - d^c\omega$ restricted to hypersurfaces $\{ t=\tu{const}\}$ is
\[
[H] = - \tfrac{9}{2}\alpha'
[\eta^{se}\wedge \omega_0^{se}] \in \alpha' \, H^3(\Sigma,4\pi^2\mathbb{Z}).
\]

\proof

Using the explicit formula for $d^c\omega$ in \eqref{eq:dcomegaexp}, it follows that 
\begin{align*}
dd^c\omega & = \left( \lambda\partial_tu - 2\lambda^2 \right) \omega_1^{se} \wedge \omega_1^{se}\\
& + \partial_t \left(\lambda \left( (\nu_3^2+\nu_0^2) \partial_t (u_0\nu_3) - 2\nu_0\nu_3 \partial_t (u_0\nu_0)\right)\right) dt \wedge \eta^{se}\wedge \omega_0^{se}
\\
& + \partial_t \left( \lambda\partial_tu - 2\lambda^2\right) dt \wedge \eta^{se}\wedge \omega_1^{se}
\\
& + \left(\partial_t \left( \lambda \left( 2\nu_0\nu_3 \partial_t (u_0\nu_3) - (\nu_3^2+\nu_0^2) \partial_t (u_0\nu_0)\right)\right) + 9 \lambda^{-1} u_0\nu_0 \right) dt \wedge \eta^{se}\wedge \omega_3^{se}
\end{align*}
and therefore $dd^c\omega=0$ is equivalent to
\begin{subequations}
\begin{align}
\partial_tu - 2\lambda & = 0,\label{eq:ddcomegared1}\\
\lambda\,\partial_t \left(\lambda \left( (\nu_3^2+\nu_0^2) \partial_t (u_0\nu_3) - 2\nu_0\nu_3 \partial_t (u_0\nu_0)\right)\right) & =0,\label{eq:ddcomegared2}\\
\lambda\,\partial_t \left( \lambda \left( 2\nu_0\nu_3 \partial_t (u_0\nu_3) - (\nu_3^2+\nu_0^2) \partial_t (u_0\nu_0)\right)\right) + 9 u_0\nu_0 & = 0. \label{eq:ddcomegared3}
\end{align}
\end{subequations}
We claim that the last equation follows from the second equation (they are in fact equivalent), provided that \eqref{eq:integrableJ} is satisfied. For this, taking the variable $\lambda ds = dt$, a direct calculation using \eqref{eq:integrableJ} and $1 + \nu_0^2 = \nu_3^2$, shows that
\begin{subequations}
\begin{align}
\lambda \left( (\nu_3^2+\nu_0^2) \partial_t (u_0\nu_3) - 2\nu_0\nu_3 \partial_t (u_0\nu_0)\right) & = \nu_3 \partial_s u_0 - u_0 \partial_s \nu_3, \\
\lambda \left( 2\nu_0\nu_3 \partial_t (u_0\nu_3) - (\nu_3^2+\nu_0^2) \partial_t (u_0\nu_0)\right) & = \nu_0 \partial_s u_0 - u_0 \partial_s \nu_0, \label{eq:BfieldODE}
\end{align}
\end{subequations}
which implies
\begin{align*}
\lambda\,\partial_t \left(\lambda \left( (\nu_3^2+\nu_0^2) \partial_t (u_0\nu_3) - 2\nu_0\nu_3 \partial_t (u_0\nu_0)\right)\right) & = \nu_3 \partial_s^2 u_0 - u_0 \partial_s^2 \nu_3, \\
\lambda\,\partial_t \left( \lambda \left( 2\nu_0\nu_3 \partial_t (u_0\nu_3) - (\nu_3^2+\nu_0^2) \partial_t (u_0\nu_0)\right)\right) & = \nu_0 \partial_s^2 u_0 - u_0 \partial_s^2 \nu_0.
\end{align*}
Hence, assuming $\nu_3 \partial_s^2 u = u_0 \partial_s^2 \nu_3$, we have
\begin{align*}
\nu_0 \partial_s^2 u_0 - u_0 \partial_s^2 \nu_0 & = u_0 \left(\nu_0 \nu_3^{-1} \partial_s^2 \nu_3 - \partial_s^2 \nu_0\right)\\
& = u_0 \nu_3^{-1}\left(3 \nu_0 \nu_3 \partial_s \nu_3  + 3 \nu_3^3 \partial_s \nu_0 - 6 \nu_0^2 \partial_s \nu_0 \right)\\
& = 3 u_0 \nu_3^{-1}  \left(\nu_3^2 - \nu_0^2 \right)\partial_s\nu_0\\
& = - 9 u_0\nu_0,
\end{align*}
where we have used that $1 + \nu_0^2 = \nu_3^2$ and hence $\nu_0 \partial_s \nu_0 = \nu_3 \partial_s \nu_3$. From the previous discussion, the Bianchi identity therefore reduces to the existence of a constant $\alpha' \in \R$ such that (see \eqref{eq:integrableJ})
$$
\tfrac{9}{2}\alpha' = \nu_3 \partial_s u_0 - u_0 \partial_s \nu_3 = \nu_3\lambda \partial_tu_0 - u_0 (3 - 3 \nu_3^2) = \nu_3\lambda \partial_tu_0 + 3 u_0\nu_0^2,
$$
which concludes the proof of the first part of the statement.

Finally, assuming now that $dd^c\omega = 0$, it follows from \eqref{eq:ddcomegared3} and \eqref{eq:BfieldODE} that
\begin{equation*}
dB = - \lambda \left( 2\nu_0\nu_3 \partial_t (u_0\nu_3) - (\nu_3^2+\nu_0^2) \partial_t (u_0\nu_0)\right) \eta^{se}\wedge \omega_3^{se} - 3\lambda^{-1}u_0\nu_0\, dt\wedge\omega_2^{se}.
\end{equation*}
The last part of the statement is a direct consequence of \eqref{eq:dcomegaexp} and Remark \ref{rmk:Cohomology:generators}.
\endproof
\end{lemma}

\begin{remark}\label{rem:Bfield}
It is interesting to observe that there is a natural closed three form
$$
H_0 = - \tfrac{9}{2}\alpha' \eta^{se}\wedge \omega_0^{se}, 
$$
and $B$-field \eqref{eq:Bfield}, which are forced upon us by the equation $dd^c\omega = 0$ and the cohomogeneity one ansatz. Note also that $H_0$ is furthermore co-closed with respect to the homogeneous Sasaki--Einstein structure on $\Sigma$, and by Lemma \ref{lem:Jformula}, $B$ is of type $(2,0) + (0,2)$ and singular along the singular orbit $S \simeq  S^3$ in the deformed conifold case. In the framework of generalized geometry, this situation corresponds to the \emph{holomorphic gauge}, in which there is a (twisted) holomorphic Courant algebroid determined by the closed complex $3$-form $H_0^{3,0 + 2,1} + i \partial \omega = 2i\partial \omega - dB^{2,0}$ (see \cite{GFJS,PymS}).
\end{remark}

We finish the section with a brief comment on the case with symmetry $G=\sunitary{3}$, where the existence of non-K\"ahler solutions with $\delta = 0$ is easily ruled out exploiting work by Alonso--Salvatore \cite{Izar}.

\begin{prop}\label{prop:su3}
Let $M^\ast$ be the set of principal orbits in a non-compact $\sunitary{3}$--invariant cohomogeneity one manifold with principal orbit $\Sigma=\sunitary{3}/\sunitary{2}\simeq S^5$ or a finite quotient thereof. Then the only (possibly incomplete) $\sunitary{3}$--invariant solutions $(\omega,\Omega)$ on $M^\ast$ of the \emph{IIB system} with vanishing source $\delta = 0$ have $\phi$ constant and therefore are K\"ahler Ricci-flat Calabi--Yau structures.
\proof
As explained in Remark \ref{def:cohoonebalanced}, the discussion in \cite[\S 4.2]{Izar} shows that any $\sunitary{3}$--invariant solution $(\omega,\Omega,e^\phi)$ to $d(e^{-2\phi}\Omega)=0=d(e^{-2\phi}\omega^2)$ on $M^\ast$ satisfies $d(e^{-\phi}\omega)=0$, \ie $d\omega = d\phi\wedge\omega$. Then the Bianchi identity $dd^c\omega=0$ is equivalent to $(dd^c\phi + d\phi\wedge d^c\phi)\wedge\omega=0$ and therefore $dd^c\phi = -d\phi\wedge d^c\phi$ (since wedge-product with $\omega$ is injective on $2$-forms). By $\sunitary{3}$--invariance $\phi=\phi(t)$ depends only on the independent variable $t$. As in the case of symmetry group $\sunitary{2}^2$, there is only one $\sunitary{3}$--invariant $1$-form $\eta^\tu{se}$ on $\Sigma$ (the isotropy representation of $K=\sunitary{2}$ is the sum $\R\oplus \C^2$ of the trivial and the standard representation of $\sunitary{2}$) with $d\eta^{\tu{se}}=2\omega_1^{\tu{se}}\neq 0$ ($\omega_1^{\tu{se}}$ is a multiple of the Fubini--Study metric on $\C\PP^2$). Hence $d^c t = \lambda \eta^{\tu{se}}$ for some nowhere vanishing function $\lambda=\lambda(t)$ and one calculates that $dd^c\omega=0$ if and only if
\[
\ddot{\phi} + \dot{\lambda}\dot{\phi} + \dot{\phi}^2 = 0 = 2\lambda\dot{\phi}.
\]
In particular, $\phi$ is constant and therefore $(\omega,\Omega)$ satisfies $d\omega=0=d\Omega$.
\endproof
\end{prop}

\subsection{ODE system and obstructions via singular sources}\label{sec:ODE}

We finish our analysis by presenting the ODE system that characterises $\sunitary{2}^2$--invariant solutions of the IIB system \eqref{eq:IIB}, in particular obtaining the remaining ODE for the conformally balanced equation $d(e^{-2\phi} \omega^2) = 0$. In order to give a complete statement, we first show that the ODEs defined by the integrability conditions and the Bianchi identity $dd^c\omega = 0$ obtained in the previous section can be solved explicitly. Using this, we will see that the vanishing of the constant $\alpha'$ in \eqref{eq:dcomegaBfield} implies that the solution is K\"ahler Ricci-flat. In the case of the crepant resolutions of $C$ and $C/\mathbb{Z}_2$ we show furthermore that $\alpha' \neq 0$ obstructs the existence of smooth solutions in a neighbourhood of the singular orbit

\begin{lemma}\label{lem:BIexplicit}
Let $(\omega,\Omega,\phi)$ be an $\sunitary{2}^2$--invariant solution of the equations
\begin{equation}\label{eq:IIBweak}
d(e^{-2\phi}\Omega)=0, \qquad dd^c\omega= 0,
\end{equation}
defined on (an open set in) $M^\ast=\R\times \Sigma$ and in normal form \eqref{eq:normalomega}. Then, 
\begin{equation}\label{eq:omeganormal}
 \omega = \tfrac{1}{2}d(u \eta^{se}) + \omega'.
\end{equation}
with $\omega' = u_0\, (\nu_3\omega_0^{se}+\nu_0\omega_3^{se})$, and there are two cases:

\begin{enumerate}
    \item If $\nu_0 = 0$, up to a change of variable the holomorphic volume $\Omega_0 = e^{-2\phi}\Omega$ coincides with the complex volume form of the conifold (see Remark \ref{rem:conifoldexp}). In this case, in the variable $\lambda\, ds = dt$ one has
\[
u_0 = \tfrac{9}{2}\alpha' s + b
\]
for some $b\in \R$ and (see \eqref{eq:dcomegaBfield})
\[
B=0, \qquad d^c \omega = \tfrac{9}{2}\alpha' \eta^{se}\wedge \omega_0^{se}.
\]

\item If $\nu_0< 0$ then $\Omega_0=e^{-2\phi}\Omega$ is the complex volume form of the Stenzel's Calabi--Yau structure on the smoothing of the conifold in Theorem \ref{thm:ACCY}. More precisely, in the variable $\lambda\, ds = dt$, one necessarily has 
\[
e^{-2\phi}\lambda\mu\nu_0 = -\kappa, \qquad e^{-2\phi}\lambda\mu = \kappa \sinh{3s}, \qquad e^{-2\phi}\lambda\mu\nu_3 = \kappa \cosh{3s},
\]
for some $\kappa>0$, which measures the period of $\Omega_0$ along the $0$-section $S^3 \cong \sunitary{2}^2/\triangle \sunitary{2}$. In this case, 
\[
u_0 = \tfrac{3}{2}\alpha' \frac{3s \cosh{(3s)} -\sinh{(3s)}+c\cosh{(3s)}}{\sinh{(3s)}},
\]
for some $c\in \R$, the $B$-field given in \eqref{eq:Bfield} corresponds to
$$
B = -\frac{3}{2}\alpha'\frac{3s + c}{\sinh{(3s)}}\omega_2^{se} 
$$
and $d^c\omega$ is given by \eqref{eq:dcomegaBfield}. Only when $c=0$ (and $\lambda$ and $u$ are an even and, respectively, an odd function of the arclength parameter $t$) $\omega$ extends smoothly over the singular orbit $\sunitary{2}^2/\triangle\sunitary{2}$ to define a $2$-form on $M=T^\ast S^3$. The $2$-form $B$ is singular along the singular orbit.
\end{enumerate}

Consequently, in either case, when $\alpha'  = 0$ the solutions are necessarily K\"ahler. 
    
\proof
We consider a new variable $s$ defined up to an additive constant by $\lambda\, ds = dt$. In this new variable the holomorphic complex volume form $\Omega_0 = e^{-2\phi}\Omega$ is given by
\[
\Omega_0 = (ds + i\eta^{se}) \wedge \left( (e^{-2\phi}\lambda\mu)\omega_2^{se} + i \left( (e^{-2\phi}\lambda\mu\nu_0)\omega_0^{se}+(e^{-2\phi}\lambda\mu\nu_3)\omega_3^{se}\right)\right).
\]
The coefficients $e^{-2\phi}\lambda\mu, e^{-2\phi}\lambda\mu\nu_0, e^{-2\phi}\lambda\mu\nu_3$ satisfy the equations described in Lemma \ref{lem:Coho1:IIB:Fterm}. These can be rewritten as
\[
\partial_s (e^{-2\phi}\lambda\mu\nu_0)=0, \qquad \partial^2_{ss}(e^{-2\phi}\lambda\mu)=9e^{-2\phi}\lambda\mu, \qquad 3e^{-2\phi}\lambda\mu\nu_3=\partial_s (e^{-2\phi}\lambda\mu).
\]
Taking into account the constraints $\lambda,\mu,\nu_3>0$ and $\nu_0\leq 0$ it is not difficult to solve these equations explicitly: either
\begin{enumerate}
\item $\nu_0=0$, $\nu_3=1$ and $e^{-2\phi}\lambda\mu=e^{3s}$, or
\item up to a translation $s\mapsto s+s_0$, 
\[
e^{-2\phi}\lambda\mu\nu_0 = -\kappa, \qquad e^{-2\phi}\lambda\mu = \kappa \sinh{3s}, \qquad e^{-2\phi}\lambda\mu\nu_3 = \kappa \cosh{3s},
\]
for some $\kappa>0$.
\end{enumerate}
In the first case, a further change of variable $r=e^s$ shows that
\[
\Omega_0 = e^{-2\phi}\lambda\mu (ds + i\eta^{se})\wedge (\omega_2^{se}+i\omega_3^{se}) = \Omega_\tu{C}
\]
is the holomorphic volume form of the conifold (see Remark \ref{rem:conifoldexp}). In the second case, it can be shown that $\Omega_0$ extends smoothly at $s=0$ over a singular orbit $S^3=\sunitary{2}^2/\triangle\sunitary{2}$ and defines the complex volume form of Stenzel's Calabi--Yau structure on the smoothing $M=T^\ast S^3$ of the conifold, see \cite[Theorem 2.27(ii)]{Foscolo:Haskins}. Note that the integral of $\Omega_0$ along the zero section is precisely $\tfrac{2}{9}\kappa$, by Remark \ref{rmk:Cohomology:generators}.

We move on to find an explicit expression for $u_0$. When $\nu_0=0$, the Bianchi identity reduces to
$$
\partial_s u_0 = \tfrac{9}{2}\alpha',
$$
and the first part of the statement follows from Lemma \ref{lem:BI}. When $\nu_0<0$, the Bianchi identity reduces to (see Lemma \ref{lem:BI})
\[
\cosh{(3s)}\sinh{(3s)}\, \partial_s u_0 + 3 u_0   = \tfrac{9}{2}\alpha'\sinh^2{(3s)}.
\]
The general solution is
\[
u_0 = \tfrac{3}{2}\alpha' \frac{3s \cosh{(3s)} -\sinh{(3s)} + c\cosh{(3s)}}{\sinh{(3s)}}
\]
for a constant of integration $c\in\R$. Asking that the $2$-form
\[
\omega = \lambda^2 ds\wedge\eta^{se} + u\, \omega_1^{se} + u_0\, (\nu_3 \omega_0^{se}+\nu_0\omega_3^{se}).
\]
is smooth on $T^\ast S^3$ forces $c=0$. This follows from \cite[Lemma 4.2]{Foscolo:Haskins}, which implies, in particular, that $\lambda$ is an even function and $u$, $u_0\nu_3$, and $u_0\nu_0$ are odd functions (see Remark \ref{rem:notationFH}), and hence $u_0$ must be an even function. Similarly, we have that 
$$
f = \lambda \left( \nu_0 \partial_t u_0 - u_0 \partial_t \nu_0\right) = -\frac{9}{2}\alpha'\frac{3s + c}{\sinh{3s}},
$$  
which is an even function but with $f(s) \to -\frac{9}{2}\alpha'$ as $s \to 0$. Hence $B = \tfrac{1}{3} f\omega_2^{se}$ (see \eqref{eq:Bfield}) does not extend to a smooth form on $M$ unless $\alpha' = 0$.

To finish, we note that $\alpha' = 0$ readily implies that $d^c\omega = 0$ when $\nu_0 = 0$, and $u_0 = 0$ when $\nu_0 \neq 0$. In the latter case, the expression \eqref{eq:dcomegaBfield} for $d^c\omega$ in terms of $\alpha'$ and $B$ and the expression \eqref{eq:Bfield} of $B$ in terms of $u_0$ yield $d^c\omega=0$ also in this case.
\endproof
\end{lemma}

\begin{remark}\label{rem:notationFH}
Notice that the triple $(\eta^{se},\omega^{se}_2,\omega_3^{se})$ in the notation of \cite{Foscolo:Haskins} corresponds to the triple $(-\eta^{se}$, $\omega^{se}_3, -\omega_2^{se})$ in the present notation. This is due to a different convention in the definition of conical Calabi--Yau structure in \cite[Definition 2.6]{Foscolo:Haskins}. Our $\omega^{se}_0,\omega_1^{se}$ agree with those in \cite{Foscolo:Haskins}.
\end{remark}

\begin{remark}\label{rem:omega'}
In the case $\nu_0 \neq 0$, the 2-form $\omega'$ in \eqref{eq:omeganormal} is the unique (up to scale) coclosed but not closed decaying harmonic $2$-form with respect to Stenzel's Calabi--Yau structure. In the case $\nu_0 =0$, $\omega'$ is a coclosed, but not closed, harmonic $2$-form with respect to the conical Calabi--Yau metric in the complex structure of the conifold. In this second case, the background $2$-form $\sigma$ in Theorem \ref{mthm:IIB:Resolved} (i) corresponds to $\sigma = \tfrac{9}{2}\alpha' s \omega_0^{se}$, in the variable $\lambda\, ds = dt$.

\end{remark}

In the case $\nu_0 = 0$, the holomorphic volume form of the conifold restricted to $M^\ast$ coincides with that of its small resolution and also, after passing to a $\Z_2$--quotient, with that of the crepant resolution $K_{\C\PP^1\times\C\PP^1}$ of $\tu{C}/\Z_2$. In the next result we prove that $\alpha' \neq 0$ obstructs the existence of smooth solutions of \eqref{eq:IIBweak} in a neighbourhood of the singular orbit $S$ in either the small resolution of $\tu{C}$ or the crepant resolution of $\tu{C}/\mathbb{Z}_2$.

\begin{lemma}\label{lem:No:Smoothness:Resolutions}
Let $M$ be a crepant resolution of $\tu{C}$ or $\tu{C}/\mathbb{Z}_2$, \ie $M$ is one of the two small resolutions $\mathcal{O}_{\C\PP^1}(-1)^{\oplus 2}$ of $\tu{C}$ or $M=K_{\C\PP^1\times\C\PP^1}$, and denote by $M^\ast = M\setminus S$, where $S\subset M$ is the singular orbit. Let $(\omega,\Omega,\phi)$ be a non-K\"ahler cohomogeneity one solution  of \eqref{eq:IIBweak} defined on some open subset of $M^*$. Then $(\omega,\Omega,\phi)$ can never extend smoothly over the singular orbit $S$. In fact, $H=-d^c\omega$ is a smooth closed 3-form on $M^\ast$ uniquely determined by its cohomology class, identified with $\alpha'\in\R\setminus\{0\}$, and $dH$ extends to the globally defined current on $M$ defined as follows.
\begin{enumerate}
\item If $M=\mathcal{O}_{\C\PP^1}(-1)^{\oplus 2}$ is the small resolution of $\tu{C}$ such that $\tfrac{3}{2}(\omega_1^\tu{se}\pm \omega_0^\tu{se})$ represents the integral class that evaluates to $2\pi$ on the singular orbit $S\simeq \C\PP^1$, then
\[
dH = \pm 4\pi^2\alpha' \delta_S,
\]
where $\delta_S$ is the current of integration along $S$.
\item If $M=K_{\C\PP^1\times\C\PP^1}$ is the crepant resolution of $\tu{C}/\mathbb{Z}_2$, then
\[
\langle dH,\tau\rangle = -\pi\alpha' \int_S{ \tau|_S\wedge\tfrac{3}{2}\omega_0^\tu{se}}.
\]
Here $S\simeq \C\PP^1\times\C\PP^1$ and $\frac{3}{2}\omega_0^\tu{se}$ is regarded as a smooth closed 2-form on $S$ representing the cohomology class of the difference of the Fubini--Study K\"ahler forms on the two factors normalised to have period $2\pi$.
\end{enumerate}
\proof
By Lemma \ref{lem:BIexplicit} we know that $d^c\omega =-H= \tfrac{9}{2}\alpha' \eta^\tu{se}\wedge\omega_0^\tu{se}$. If $\omega$ extends from an open subset of $M^\ast$ to an open subset of $M$ containing $S$ then $d^c\omega$ must vanish in cohomology (since in either case $U$ has no third cohomology) and therefore also the restriction of $d^c\omega$ to the principal orbit must vanish in cohomology, \ie $\alpha'=0$.

Conversely, the expression above shows that $H$ is a well defined smooth 3-form on $M^\ast$ and we now show that $dH$ is a well-defined current on $M$. Let $U$ be an open subset of $M$ containing $S$. Consider $\epsilon>0$ sufficiently small and let $U^\ast_\epsilon$ denote the subsets of $U$ of points at distance bigger than $\epsilon$ from the singular orbit $S$. Recall that $dH$ is well defined as a current on $M$ if the limit
\begin{equation}\label{eq:currentabs}
\langle dH,\tau\rangle := \lim_{\epsilon\ra 0}{\int_{U^\ast_\epsilon}H\wedge d\tau}
\end{equation}
exists for every $\tau$ is any smooth compactly supported $2$-form $\tau$ on $U$.

Now, denote by $\Sigma_\epsilon$ the boundary component corresponding to points at distance exactly $\epsilon$ from $S$. We identify $\Sigma_\epsilon$ with $\Sigma$ or $\Sigma/\Z_2$ in the two cases of the Lemma and endow it with the orientation induced by the Sasaki--Einstein metric $g^\tu{se}$. Then, using Stokes' Theorem (being careful about orientations) one has
\[
{\int_{U^\ast_\epsilon}H\wedge d\tau} = {\int_{\Sigma_\epsilon}H\wedge \tau} = -{\int_{\Sigma_\epsilon}\tfrac{9}{2}\alpha'\eta^\tu{se}\wedge\omega_0^\tu{se}\wedge \tau},
\]
for any smooth compactly supported $2$-form $\tau$ on $U$.

Now, $\Sigma_\epsilon$ is the total space of a fibration $\pi_S\co \Sigma_\epsilon\ra S$ over the singular orbit $S$ with fibre either a $3$-sphere or a circle depending whether we work in case (i) or (ii). To see this, we can use a Riemannian normal neighborhood of $S$, using the Riemannian normal bundle. Using \eqref{eq:Homogeneous:SE:LeftInvariant}, one can check that $\tfrac{9}{2} \eta^\tu{se}\wedge (\omega_1^\tu{se}\mp \omega_0^\tu{se})$ and $6\eta^\tu{se}$ are forms on $\Sigma_\epsilon$ that integrate to, respectively, $4\pi^2$ and $2\pi$ on the fibres of $\pi_S$ (in the first case, the sign depends on the choice of small resolution of the conifold, see Remark \ref{rem:ACCY3}). Finally, since $\tau$ is a smooth form on $U$, the restriction of $\tau$ to $\Sigma_\epsilon$ satisfies $\tau|_{\Sigma_\epsilon} = \pi_S^\ast(\tau|_S) + O(\epsilon)$ as $\epsilon\ra 0$. Therefore the limit for $\epsilon\ra 0$ in \eqref{eq:currentabs} exists for any $\tau$ defining the current $dH$, and is given as in the statement of the Lemma (rewriting $2\omega_0^\tu{se} = \mp (\omega_1^\tu{se}\mp \omega_0^\tu{se}) \pm (\omega_1^\tu{se}\pm \omega_0^\tu{se})$ in the case of one of the small resolutions of the conifold).
\endproof
\end{lemma}

We present next the main result of this section.

\begin{prop}\label{prop:Coho1:IIB}
Up to the discrete symmetry \eqref{eq:Coho1:IIB:Discrete:Symmetry}, a constant change of phase of the complex volume form, and the action of $\sunitary{2}^2$--equivariant diffeomorphisms, $\sunitary{2}^2$--invariant solutions of the IIB system \eqref{eq:IIB}, inducing the same orientation on the hypersurfaces $\{ t=\tu{const}\}$ as the homogeneous Sasaki--Einstein structure, are in one-to-one correspondence with solutions $(u,\phi)$ of the ODE system
\begin{equation}\label{eq:ODEfinal}
\partial_s u = \frac{2V^2 e^{4\phi}}{u^2-u_0^2}, \qquad \partial_s \phi = - \frac{u_0\, \partial_s u_0}{u^2-u_0^2},
\end{equation}
satisfying
\begin{equation}\label{eq:ODEpositive}
u > 0, \qquad \partial_s u > 0, \qquad u^2-u_0^2 >0,
\end{equation}
where, either:

\begin{enumerate}
\item $e^{-2\phi}\Omega$ is the complex volume form of the conifold and
\[
V = e^{3s}, \qquad u_0 = \tfrac{9}{2}\alpha' s + b
\]
for some constants $\alpha', b\in \R$;

\item $e^{-2\phi}\Omega$ is the complex volume form of the smoothing $T^\ast S^3$ of the conifold, and
\[
V = \kappa \sinh{3s}, \qquad u_0=\tfrac{3}{2}\alpha' \frac{3s \cosh{(3s)} -\sinh{(3s)}+ c\cosh{(3s)}}{\sinh{(3s)}}
\]
for some some constants $\alpha', c\in \R$ and $\kappa> 0$. Moreover, if $(\omega,\Omega, e^\phi)$ extends smoothly to $T^\ast S^3$ then necessarily $c=0$.
\end{enumerate}

In either case, solutions with $\alpha'=0$ are K\"ahler Calabi--Yau. Moreover, given a solution $(u,\phi)$ of \eqref{eq:ODEfinal} satisfying \eqref{eq:ODEpositive}, the associated $\tu{SU}(3)$--structure is of the form \eqref{eq:normalomega}, where the function $\mu$ is defined by $\mu^2=u^2-u_0^2$, the function $\lambda$ is defined by $\partial_s u = 2\lambda^2$ and the arc-length parameter $t$ is implicitly defined by $\lambda\, ds = dt$. 
\proof
Work in the parameter $s$ defined by $\lambda\, ds=dt$. Thanks to Lemma \ref{lem:BIexplicit}, the description of the triple $(\omega,\Omega,e^\phi)$ in Lemma \ref{lem:Coho1:IIB:Fterm} depends on the explicitly known functions $V=e^{-2\phi}\lambda\mu$, $u_0$, $\nu_0, \nu_3$ and unknown functions $u,\phi,\lambda,\mu$. In fact, given $V$ and $u_0$, $u$ and $\phi$ alone determine $\lambda$ and $\mu$ via $\mu^2=u^2-u_0^2$ and $\lambda^2\mu^2 = V^2 e^{4\phi}$.

Now, imposing $dd^c\omega=0$ yields the additional constraint $\partial_s u=2\lambda^2$ of \eqref{eq:Coho1:IIB:Bianchia}. One further computes that $d(e^{-2\phi}\omega^2)=0$ if and only if
\begin{equation}\label{eq:confbalanceds}
\partial_s \left( e^{-2\phi}(u^2-u_0^2)\right) -4e^{-2\phi}\lambda^2 u=0.
\end{equation}
The system given by the two equations \eqref{eq:confbalanceds} and $\partial_s u=2\lambda^2$ is equivalent to \eqref{eq:ODEfinal} once we substitute the expressions for $\lambda, \mu$ in terms of $V,u_0, u,\phi$.

Finally, if $\alpha'=0$ then $\partial_su_0=0$ so $\phi$ is constant and $(\omega,\Omega)$ a K\"ahler Calabi--Yau structure.
\endproof
\end{prop}

\begin{remark}\label{rmk:Coho1:IIB}
The ODE system in the statement of the theorem is equivalent to the single second-order ODE for $u$
\[
\frac{\partial^2_{ss}u}{\partial_s u} = \partial_s \log{\left( V^2\right)} - \frac{\partial_s (u^2+u_0^2)}{u^2-u_0^2}
\]
with known given coefficients $V$ and $u_0$.
\end{remark}

\begin{remark}\label{rmk:Coho1:IIB:Infinity}
Observe that in case (ii) for $s\gg 1$ we have
\[
V \approx \tfrac{\kappa}{2}e^{3s} = e^{3\left( s+\tfrac{1}{3}\log{\tfrac{\kappa}{2}}\right)}, \qquad u_0\approx \tfrac{9}{2} \alpha' s  - \tfrac{3}{2} \alpha' = \tfrac{9}{2} \alpha' \left( s+\tfrac{1}{3}\log{\tfrac{\kappa}{2}}\right)  - \tfrac{3}{2} \alpha' \left( 1+\log{\tfrac{\kappa}{2}}\right)
\]
up to exponentially decaying terms. Hence, up to a translation $s\mapsto s+\tfrac{1}{3}\log{\frac{\kappa}{2}}$, the ODE system of Proposition \ref{prop:Coho1:IIB} in case (ii) has coefficients exponentially close to the one in case (i) with $b=-\tfrac{3}{2}\alpha'\left( 1+\log{\frac{\kappa}{2}}\right)$.
\end{remark}

\subsection{Parameters and rescalings}\label{sec:Parameters}
We conclude this section by discussing the parameters involved in Proposition \ref{prop:Coho1:IIB} and natural symmetries of the ODE system. First of all, note that a consequence of our analysis is that, for any solution, $H=-d^c\omega$ is uniquely determined by $\alpha'$, which determines the cohomology class $[H]\in H^3(\Sigma)$. We have already observed that $\alpha'=0$ implies that $(\omega,\Omega)$ is a K\"ahler Calabi--Yau structure. We will therefore focus on the non-K\"ahler case $\alpha'\neq 0$. Note that the ODE system of Proposition \ref{prop:Coho1:IIB} is invariant under $(u,e^\phi,u_0)\mapsto (u,e^\phi,-u_0)$. Hence without loss of generality we can assume that $\alpha'>0$ so that $u_0>0$ for $s\gg 1$.

In case (i) of Proposition \ref{prop:Coho1:IIB}, \ie in the complex structure of the conifold, the parameter $b$ is the choice of a cohomology class in $H^2(\Sigma)$ associated to the choice of the primitive $\omega$ for $JH$. This parameter does not appear in case (ii) when we make the further assumption that solutions extend smoothly from $M^\ast$ to $T^\ast S^3$: the latter has no second cohomology and therefore there is no cohomological freedom in the choice of primitive for $JH$. (If we do not assume smoothness, then the additional choice corresponds to the choice of integration constant $c$ in the expression for $u_0$).

The parameter $\kappa$ in case (ii) of Proposition \ref{prop:Coho1:IIB}, \ie in the complex structure of the smoothing of the conifold, is the cohomology class $[e^{-2\phi}\Omega]\in H^3(\Sigma;\C)$: a priori it is a complex number, but it can always be assumed to lie in $(0,\infty)$ by changing the phase of the complex volume form, as we have done. In case (i) this cohomological invariant vanishes.

There are natural scaling freedoms in the equation arising from scaling the $\sunitary{3}$--structure $(\omega,\Omega)$ by $(\Lambda^2\omega, \Lambda^3 \Omega)$ for some $\Lambda>0$ and independently translating the dilaton $\phi$ by a constant $\psi\in\R$. Under these symmetries, if $(u,\phi)$ is a solution of the system of Proposition \ref{prop:Coho1:IIB} with parameters $(\alpha',\kappa,b)$ ($\kappa=0$ in case (i) and $b=0$ in case (ii) when we assume smoothness on $T^\ast S^3$), then:
\begin{enumerate}
\item if we work in the complex structure of the conifold, then up to a translation $s\mapsto s-\log{\Lambda}+\tfrac{2}{3}\psi$, $(\Lambda^2 u,\phi+\psi)$ is a solution with parameters
\begin{subequations}\label{eq:Coho1:IIB:Scaling}
\begin{equation}
\left( \Lambda^2\alpha',0,\Lambda^2 b-\tfrac{9}{2}\Lambda^2\alpha'\left( \log{\Lambda}-\tfrac{2}{3}\psi \right)\right);
\end{equation}
\item if we work in the complex structure of the smoothing of the conifold, then $(\Lambda^2 u,\phi+\psi)$ is a solution with parameters
\begin{equation}
\left( \Lambda^2\alpha',\Lambda^3e^{-2\psi}\kappa,0\right).
\end{equation}
\end{subequations}
\end{enumerate}
In each case, we can choose $\Lambda$ and $\psi$ to normalise the choice of $(\alpha',b)$ and, respectively, $(\alpha',\kappa)$. This becomes relevant when studying limits of families of solutions of the IIB system.

\section{Explicit solutions and singular sources}

We have seen in Lemma \ref{lem:No:Smoothness:Resolutions} that on the crepant resolutions of the conifold or its $\Z_2$--quotient there cannot be any smooth non-K\"ahler $\sunitary{2}^2$--invariant solution of the IIB system. The purpose of this section is to discuss some explicit solutions of the IIB system \eqref{eq:IIB} that will serve as models for natural singular behaviour for the ODE system of Proposition \ref{prop:Coho1:IIB} when no smooth solutions are allowed. We will also introduce a family of complete solutions on the deformed conifold, called \emph{CV--MN solutions}, which will correspond to the case $p=p_*$ and $q=q_\ast$ in Theorems \ref{mthm:IIB:Deformed} and \ref{mthm:IIB:Resolved}.

\subsection{$5$-brane sources and the CV--MN models}

In physics jargon, the solutions we discuss in this section have $5$-brane sources: let $S$ be a (possibly disconnected) smooth holomorphic curve in the complex manifold $(M,J)$; we look for solutions of the system on the complement $M\setminus S$ of  $S$ and which satisfy the global equation
\[
dd^c\omega = -4\pi^2\alpha'\,
\sum_{i=1}^n{k_i\,\delta_{S_i}},
\]
where $S_1,\dots, S_n$ are the connected components of $S$, $k_i\in\R$ are fixed weights and $\delta_{S_i}$ denotes the current of integration along $S_i$. The first part of Lemma \ref{lem:No:Smoothness:Resolutions} explains why this natural physical set-up is relevant for our analysis.

We will first discuss the ``local model'' case where $M=\C^3=\C\times \C^2$ and $S=\bigsqcup_{i=1}^n\C\times\{ p_i\}$ for distinct points $p_1,\dots, p_n\in\C^2$. In the simplest situation ($n=1$ and a natural constant of integration is set to $0$), the solution of the IIB system on the complement of $S$ is given by the isometric product of the flat metric on $\C$ and the Boothby metric on $\C^2\setminus\{ 0\}$ (that is, the cylindrical metric on $\C^2\setminus\{ 0\} \cong \R \times S^3$). This solution is well known both in the physics literature, \cf the ``neutral solution'' \cite{Callan:Harvey:Strominger}, and in the mathematics literature, where it appears more commonly after passing to a compact quotient as a canonical non-K\"ahler metric on the Hopf surface. Superpositions of neutral solutions centred at different points in $\C^2$ were considered in \cite[Section 5.1]{VAHS}, in applications of the IIB system to vertex algebras. The construction can be considered as a IIB analogue of the Gibbons--Hawking Ansatz for 4-dimensional hyperk\"ahler metrics with a triholomorphic circle action, which give rise to models for solutions of the IIA system with 6-brane sources.

We will then attempt to write a ``global solution'' of the IIB system as a fibration over $\C\PP^1$ with fibre the neutral solution using the Chern connection of $\mathcal{O}_{\C\PP^1}(-1)$ (the resulting complex manifold is then naturally identified with the complement of the 0-section in the small resolution $\mathcal{O}_{\C\PP^1}(-1)\oplus\mathcal{O}_{\C\PP^1}{O}(-1)$ of the conifold). This attempt will lead us naturally to recover an explicit singular solution of the ODE system of Proposition \ref{prop:Coho1:IIB} in the complex structure of the conifold that was first discussed in the physics literature by Maldacena--Nu\~nez \cite{Maldacena:Nunez}. The non-vanishing curvature of the Chern connection on $\mathcal{O}_{\C\PP^1}(-1)$ has interesting effects on this geometry, which has a complete end and an incomplete singularity of an exotic kind. Explicit complete metrics on $T^\ast S^3$ with the same asymptotics at infinity were discussed by Maldacena--Nu\~nez \cite{Maldacena:Nunez} building on work by Chamseddine--Volkov \cite{Chamseddine:Volkov}. We will refer to these explicit solutions of the cohomogeneity one IIB system of Proposition \ref{prop:Coho1:IIB} as the \emph{CV--MN solutions}, either on the resolved or the deformed conifold. The solution is singular in the former case and smooth in the latter. There is in fact a family of CV--MN solutions, depending on parameters $\alpha'$ and $b$ or $\kappa$ according to the two cases, but the rescalings of \eqref{eq:Coho1:IIB:Scaling} can be used to fix these parameters.

Later Maldacena--Martelli \cite{Maldacena:Martelli}, based on numerical solutions and heuristic matched asymptotics expansions, speculated on the existence of a 1-parameter family of solutions on $T^\ast S^3$ that are asymptotically conical, with the CV--MN solution appearing as a limit. The purpose of the rest of the paper is to establish the existence of this family rigorously (see Theorem \ref{mthm:IIB:Deformed} in the Introduction), as well as the existence of an analogous family of solutions on exterior domains in the conifold with prescribed incomplete behaviour at the interior boundary and asymptotically conical behaviour at infinity (our Theorem \ref{mthm:IIB:Resolved} in the Introduction).

\subsection{The Callan--Harvey--Strominger ansatz}

Let $h$ be a positive harmonic function on an open set $U$ of $\C^2$ and consider the $\sunitary{3}$--structure $(\omega,\Omega)$ and function $e^{\phi}$ on $\C\times U$ defined by
\begin{equation}\label{eq:IIB:GH}
\omega=\omega_\C +h\,\omega_{\C^2}, \qquad \Omega= h\,\Omega_{\C^3}, \qquad e^\phi = \sqrt{h}.
\end{equation}
Here $\omega_{\C^m}$ and $\Omega_{\C^m}$ denote, respectively, the standard K\"ahler form and holomorphic volume form on $\C^m$. It is immediate to check that $d(e^{-2\phi}\Omega)=0=d(e^{-2\phi}\omega^2)$: the first equation is obvious, while the second one follows from the fact that $e^{2\phi}=h$ only depends on the coordinates on $U$, so that $h\omega_{\C^2}^2$ is trivially closed. On the other hand,
\[
dd^c\omega = dJ(dh\wedge\omega_{\C^2})= (dJdh)\wedge\omega_{\C^2}= -(\triangle h) \tu{vol}_{\C^2}=0.
\]
Hence the triple $(\omega,\Omega,e^{2\phi})$ determined by $h$ in \eqref{eq:IIB:GH} provides a solution of the IIB system. This is analogous to the Gibbons--Hawking Ansatz for 4-dimensional hyperk\"ahler metrics with a circle symmetry, which can be regarded as solutions of the IIA system on an open set $\R^3\times U\subset \R^3 \times \R^3\simeq\C^3$. In the related context of the heterotic system, an ansatz equivalent to \eqref{eq:IIB:GH} has been considered by Strominger \cite{Strominger} and Callan--Harvey--Strominger \cite{Callan:Harvey:Strominger}, so we will refer to \eqref{eq:IIB:GH} as the \emph{Callan--Harvey--Strominger ansatz} and to its complete solutions below as the \emph{CHS solutions}.  

Now, positive harmonic functions on the complement of a finite set in $\C^2$ must necessarily be a sum of a constant and multiples of Green's functions on $\C^2$ with singularities at a collection of points $p_i$ in $\C^2$. We consider here a finite number of points, though in principle one could also consider infinitely many of them. We consider then harmonic functions of the form  
\begin{equation}\label{eq:IIB:GH:harmonic}
h=c+\alpha'\sum_{i=1}^n{ \frac{k_i}{|z-p_i|^2}}
\end{equation}
for constants $c,\alpha'\geq 0$, distinct points $p_1, \dots, p_n$ in $\C^2$ and weights $k_1,\dots,k_n>0$. Then $h$ is harmonic on $U=\C^2\setminus\{ p_1,\dots,p_n\}$ while as distributions on $\C^2$
\[
\triangle h = 4\pi^2\alpha'\sum_{i=1}^n{k_i \delta_{p_i}},
\]
where $\delta_{p_i}$ is the Dirac delta in $p_i$. The resulting solution $(\omega,\Omega,e^{2\phi})$ of the IIB system is therefore defined on $\C\times (\C^2\setminus \{ p_1,\dots,p_n\})\subset\C^3$ and as currents on $\C^3$ it satisfies
\[
dd^c\omega = -4\pi^2\alpha' \sum_{i=1}^n{k_i\delta_{S_i}},
\]
where $\delta_{S_i}$ is the current of integration on the holomorphic curve $S_i=\C\times \{ p_i\}$ in $\C^3$. Solutions with $\alpha'=0$ correspond to the flat K\"ahler structure on $\C^3$, so in the following we will assume without further notice that $\alpha'>0$.

The simplest non-K\"ahler solution corresponds to the choice $c=0$, $n=1$. Then $h=\alpha' r^{-2}$ (for $k_1=1$, say) where $r$ is a radial coordinate on $\C^2$ centred at $p_1$ and the resulting metric is isometric to the product of the Euclidean metric on $\C$ with the Boothby (\ie cylinder) metric $g_{cyl}=dt^2 + \alpha' g_{\Sph^3}$ on $\C^2\setminus\{ 0\} \cong \R \times S^3$. Here $t=\sqrt{\alpha'}\log{r}$ and $g_{\Sph^3}$ is the round metric on the unit sphere $\Sph^3$ in $\C^2$. This solution is Bismut flat, \ie the Bismut connection $\nabla^\tu{B}=\nabla^{\tu{LC}}-\frac{1}{2}g^{-1}d^c\omega$ has vanishing curvature.

\begin{remark*}
The Boothby metric appears more commonly in the literature as a natural example of (Bismut flat) non-K\"ahler metric on the Hopf surface $(\C^2\setminus\{0\})/\Z\simeq S^1\times S^3$.
\end{remark*}

Now, since for $|z-p_i|\ra 0$ we have
\[
h\approx \alpha' k_i \frac{1}{|z-p_i|^2},
\]
by \eqref{eq:IIB:GH:harmonic}, and hence the solution $(\omega,\Omega, e^{2\phi})$ corresponding to the generic choice of parameters in \eqref{eq:IIB:GH:harmonic} has a cylindrical end near $S_i$ with asymptotic cylinder of the form
\[
g_\C + (dt^2 +  \alpha' k_i g_{\Sph^3}),
\]
for $t = \sqrt{\alpha'k_i}\log |z-p_i|$, \ie the cross-section $S^3$ of the cylinder has radius $\sqrt{\alpha' k_i}$. As $|z|\ra \infty$ the asymptotic behaviour of the solution depends on the parameter $c$. If $c=0$ we also have a cylindrical end at infinity with cross-section $S^3$ of the cylinder of radius $\sqrt{\alpha'(k_1+\dots+k_n)}$. If $c>0$ the metric has instead a Euclidean end.

\subsubsection{Bismut holonomy of the CHS solutions}

In this section we calculate the holonomy of the Bismut connection for the solutions in the CHS ansatz. In order to state the result we keep the notation just introduced. 

\begin{prop}\label{propo:chs:hol}
Let $(\omega,\Omega,e^\phi)$ be a CHS solution \eqref{eq:IIB:GH} with $h$ as in
\eqref{eq:IIB:GH:harmonic}, on $M=\C\times U$ with
$U=\C^2\setminus\{p_1,\dots,p_n\}$.  Then the holonomy of the Bismut
connection of $(\omega,\Omega,e^\phi)$ is
\[
\operatorname{Hol}(\nabla^B)=
\begin{cases}
\{1\} & \text{if } \alpha' = 0 \text{ (flat solution)} \text{ or } c=0  \text{ and } n=1 \quad \text{(Boothby solution)},\\[2pt]
\sunitary{2} & \text{otherwise},
\end{cases}
\]
where $\sunitary{2}\subset\sunitary{3}$ acts trivially on the flat $\C$ factor.
In particular, every CHS solution other than the flat metric and the Boothby metric is not Bismut-flat and has full holonomy $\sunitary{2}$. 
\end{prop}
The Proposition in particular shows that there are infinitely many complete Bismut--Hermitian--Einstein metrics in complex dimension $2$ with full holonomy $\sunitary{2}$. Since $\alpha'\neq 0$, note that these examples have necessarily at least two complete ends, at least one of which must be cylindrical.

Now, the metric determined by \eqref{eq:IIB:GH} is the Riemannian product
\begin{equation}\label{eq:chs:metric}
g = g_{\C}\oplus g_h, \qquad g_h := h\,\delta ,
\end{equation}
where $\delta$ denotes the flat metric of $\C^2 \cong \R^4$, and therefore we focus our analysis on the four-dimensional factor $(U,g_h)$. Note that the torsion $3$-form of the Bismut connection $\nabla^{\tu B,h}$ of $g_h$ is
\beq\label{eq:chs:H}
H \;=\; -\ast_\delta dh \;=\; -\ast_{g_h} d\log h.
\eeq
It is not difficult to see that $g_h$ is Bismut--Hermitian--Einstein, and therefore the holonomy of its Bismut connection is contained in $\sunitary{2}$, since $U$ is simply connected.

Let $x_0,x_1,x_2,x_3$ be the standard coordinates on $\R^4$, with
$z_1=x_0+ix_1$, $z_2=x_2+ix_3$, and write $e^a:=dx_a$, $\partial_a:=\frac{\partial}
{\partial x_a}$. We use the natural orientation, so that $\vol_\delta = e^{0123}$ and
\[
\omega_{\C^2}=e^{01}+e^{23},
\]
where $e^{ab}:=e^a\wedge e^b$ and so on. The Hodge star operator of $g_h$ agrees with that of $\delta$ on $2$-forms, and we let
\[
\Lambda^2 = \Lambda^2_+\oplus\Lambda^2_-
\]
be the corresponding (conformally invariant) splitting into self-dual and anti-self-dual forms.  A basis of
$\Lambda^2_-$ is given by
\begin{equation}\label{eq:chs:sigma}
\sigma^-_1 = e^{01}-e^{23},\qquad
\sigma^-_2 = e^{02}-e^{31},\qquad
\sigma^-_3 = e^{03}-e^{12}.
\end{equation}
We work in the global $g_h$-orthonormal frame
\begin{equation}\label{eq:chs:frame}
f_a := h^{-1/2}\partial_a,\qquad f^a := h^{1/2}e^a.
\end{equation}
We denote by $\nabla^0$ the unique
connection for which $f_0,\dots,f_3$ are parallel.  Explicitly, writing $\nabla^\delta$ for the flat Euclidean connection
\begin{equation}\label{eq:chs:nabla0}
\nabla^0 \;=\; \nabla^\delta+\tfrac12\,d\log h\otimes\operatorname{Id}.
\end{equation}
For later use we record that
\beq\label{eq:chs:LCvsnabla0}
\nabla^{g_h}_XY = \nabla^0_XY + \tfrac12 \Big(d\log h(Y)\,X-g_h(X,Y) (d\log h)^{\sharp_{g_h}}\Big).
\eeq
We denote the connection one-form of the Bismut connection $\nabla^{\tu B,h}$ of $g_h$ in the trivialisation \eqref{eq:chs:frame} by
\beq\label{eq:chs:A}
A := \nabla^{\tu B,h}-\nabla^0 \;\in\; \Omega^1\big(U,\mathfrak{so}(TU,g_h)\big)
\cong \Omega^1\big(U,\Lambda^2T^\ast U\big).
\eeq
We use the convention $\mathfrak{so} \to \Lambda^2T^\ast \colon P \to \alpha_P$ with $\alpha_P(X,Y)=\delta(X,PY)$. We will write $X^\flat=\delta(X,\cdot)$ throughout.

\begin{lemma}\label{lem:chs:connection}
In the trivialisation \eqref{eq:chs:frame} the connection form
\eqref{eq:chs:A} of the Bismut connection of $g_h$ is
\beq\label{eq:chs:connectionform}
A(X) = \big[\,X^\flat\wedge d\log h\,\big]_-
\eeq
where $[\,\cdot\,]_-$ is the anti-self-dual projection.
\end{lemma}

\begin{proof}
A standard calculation using \eqref{eq:chs:LCvsnabla0} shows that the connection one-form of the Levi-Civita connection of $g_h$, in the given frame, is
\beq\label{eq:chs:ALC}
A^{g_h}(X) = \tfrac12 X^\flat\wedge d\log h.
\eeq
Using now  $H=-\ast_{g_h} d \log h$, so that  $ \iota_XH = h \ast_{\delta}(X^\flat\wedge d \log h)$, 
we obtain
\[
A(X)=\tfrac12 X^\flat\wedge d\log h -\tfrac12 \ast(X^\flat\wedge d \log h)
= [X^\flat\wedge d \log h]_-.\qedhere
\]
\end{proof}

The curvature of $\nabla^{\tu B,h}$ is
$R^{\tu B, h}=dA+A\wedge A$, where
$(A\wedge A)(X,Y)=[A(X),A(Y)]$. To calculate an explicit formula, define
the symmetric $2$--tensor
\beq\label{eq:chs:Q}
Q \;:=\; h\,\Hess_\delta\!\big(h^{-1}\big)
\;=\;-\Hess_\delta(\log h)+d\log h\otimes d\log h.
\eeq
Using that $h$ is harmonic, $\Delta_\delta\log h = -\tr_\delta \Hess_\delta (\log h) =\abs{d\log h}_\delta^2$, so
\beq\label{eq:chs:trQ}
\tr_\delta Q = 2\abs{d\log h}_\delta^2 ,
\eeq
and we write
\beq\label{eq:chs:Qzero}
Q_0 \;:=\; Q-\tfrac14(\tr_\delta Q)\,\delta
\;=\;-\Hess_\delta(\log h)+d\log h\otimes d\log h-\tfrac12\abs{d\log h}^2_\delta\,\delta
\eeq
for its trace-free part. Note that $Q_0$ is also the $g_h$--trace-free part of $Q$, since $\tr_{g_h}Q=h^{-1}\tr_\delta Q$. For a symmetric $2$-tensor $S$, we define the $\Lam^2_-$--valued $2$-form
\beq\label{eq:chs:wedgeop}
\mathcal{R}(S)(X,Y) = \big[\,X^\flat\wedge S(Y,\cdot)-Y^\flat\wedge S(X,\cdot)\,\big]_-.
\eeq
Note that if $S$ is traceless, then $\mathcal{R}(S)$ is a linear map $\Lambda^2_+\ra \Lambda^2_-$, \ie $\mathcal{R}(S)$ vanishes when eveluated on anti-self-dual forms (a well-known fact from the decomposition of the curvature operator of Riemannian 4-manifolds).

\begin{lemma}\label{lem:chs:curvature}
The curvature of the Bismut
connection of $g_h$ is
\beq\label{eq:chs:curvature}
R^{\tu B,h}(X,Y) = \big[\,X^\flat\wedge Q_0(Y,\cdot)-Y^\flat\wedge Q_0(X,\cdot)\,\big]_-
= \mathcal{R}(Q_0)(X,Y),
\eeq
with $Q_0$ as in \eqref{eq:chs:Qzero}.
\end{lemma}

\begin{proof}
A straightforward calculation shows that
\beq\label{eq:chs:dA}
dA = -\mathcal{R}\big(\Hess_\delta(\log h)\big).
\eeq
On the other hand, it is not difficult to see that for a $1$--form $L$ and any vector fields $X,Y$,
\beq\label{eq:chs:bracket}
\big[\,[X^\flat\wedge L]_-\,,\,[Y^\flat\wedge L]_-\,\big]
\;=\;\mathcal{R}\big(L\otimes L-\tfrac12\abs{L}^2_\delta \delta \big)(X,Y).
\eeq
Hence, applying Lemma \ref{lem:chs:connection},
\beq\label{eq:chs:AA}
(A\wedge A)(X,Y)=\big[A(X),A(Y)\big]
= \mathcal{R}\big(d\log h\otimes d\log h-\tfrac12\abs{d\log h}^2_\delta \delta \big)(X,Y),
\eeq
and the proof concludes by adding \eqref{eq:chs:dA}.
\end{proof}

As a corollary, we obtain the following:

\begin{corollary}\label{cor:chs:flat}
Let $h>0$ be harmonic on a connected open set $U\subseteq\R^4$.  The Bismut
connection of $g_h$ is flat if and only if 
\[
h\equiv\text{const}\qquad\text{or}\qquad h=\frac{a}{\abs{x-p}^2}
\ \ \text{ for some }a>0,\ p\in\R^4 .
\]
Furthermore, $R^{\tu B, h}$ is a
\emph{self-dual} $2$-form with values in $\Lam^2_-$, that is, it defines an $\sunitary{2}$-instanton for $g_h$ with respect to the opposite orientation.
\end{corollary}

\begin{proof}
By Lemma \ref{lem:chs:curvature}, $R^{\tu B,h}=0$ if and only if
$\mathcal{R}(Q_0)=0$. We make a change of positively oriented orthonormal basis, so that $Q_0$ is diagonalised, with eigenvalues $q_0,q_1,q_2,q_3$. Then, a straightforward calculation shows that
\beq\label{eq:chs:Rab}
\mathcal{R}(Q_0)(e_a,e_b)=(q_a+q_b)\,[e^{ab}]_-.
\eeq
From this, it follows that $R^{\tu B,h}$ is an instanton for $g_h$ with respect to the opposite orientation, and that $R^{\tu B,h} = 0$ implies $Q_0 = 0$. Now $Q_0=0$ is equivalent to $\Hess_\delta(h^{-1})=\lambda\,\delta$ for a function
$\lambda$. differentiating, $\partial_k\lambda\,\delta_{ij}$ must be totally
symmetric, and therefore $\lambda$ is constant and $h^{-1}=\tfrac{\lambda}{2}\abs{x}^2+B\cdot x+C$. Now, we have
$$
\Delta_\delta h = -(2|B|^2 - 4 \lambda C)h^3,
$$
and, since $h>0$ is harmonic, $|B|^2 = 2\lambda C$. If $\lambda = 0$, then $B = 0$ and $h$ is constant. If $\lambda \neq 0$, completing squares, $h=\tfrac{2}{\lambda }\abs{x + \lambda^{-1}B}^{-2}$, as claimed. The converse is trivial by Lemma \ref{lem:chs:curvature}.
\end{proof}

Fix $h$ as in \eqref{eq:IIB:GH:harmonic} and a centre $p_i$. Write
$\rho:=\abs{z-p_i}$, $x:=z-p_i$, and split
\beq\label{eq:chs:split}
h=\frac{a_i}{\rho^2}+w_i,\qquad
a_i:=\alpha'k_i>0,\qquad
w_i(z):=c+\alpha'\sum_{j\neq i}\frac{k_j}{\abs{z-p_j}^2},
\eeq
so that $w_i$ is harmonic and smooth near $p_i$.  Set
\beq\label{eq:chs:bi}
b_i:=w_i(p_i)=c+\alpha'\sum_{j\neq i}\frac{k_j}{\abs{p_i-p_j}^2}\;\geq\;0 .
\eeq
Fix $\rho_0>0$ smaller than the distance from $p_i$ to the other centres, and let $\hat x:=x/\rho$ be the outward radial \emph{vector} field on the punctured ball $B_{\rho_0}(p_i)\setminus\{p_i\}$, of unit length with respect to $\delta$. 

\begin{lemma}\label{lem:chs:asymptotics}
With the notation \eqref{eq:chs:split}--\eqref{eq:chs:bi}, as $z\ra p_i$ we have the following expansion, uniformly in the direction $\hat x$:
\beq\label{eq:chs:Qasymp}
Q_0 \;=\; -\frac{8b_i}{a_i}\Big(\hat x^\flat\otimes\hat x^\flat - \tfrac14\delta\Big)+O(\rho).
\eeq
\end{lemma}

\begin{proof}
Write $w_i=b_i+O(\rho)$.  Then
\[
h^{-1}=\frac{\rho^2}{a_i}\Big(1+\frac{w_i\rho^2}{a_i}\Big)^{-1}
=\frac{\rho^2}{a_i}-\frac{b_i}{a_i^2}\rho^4+O(\rho^5).
\]
Now $\Hess_\delta(\rho^2)=2\delta$ is pure trace, while
$\Hess_\delta(\rho^4)=8\,x^\flat\otimes x^\flat +4\rho^2\delta$ has trace-free part
$8\big(x^\flat \otimes x^\flat -\tfrac14\rho^2\delta\big)$.  Hence the trace-free part of
$\Hess_\delta(h^{-1})$ is $-\tfrac{8b_i}{a_i^2}\big(x^\flat\otimes
x^\flat-\tfrac14\rho^2\delta\big)+O(\rho^3)$, and multiplying by
$h=a_i\rho^{-2}(1+O(\rho^2))$ gives \eqref{eq:chs:Qasymp}, since $x^\flat=\rho d\rho$ and therefore
$x^\flat\otimes x^\flat=\rho^2 \hat x^\flat\otimes\hat x^\flat$.
\end{proof}

\begin{proof}[Proof of Proposition \ref{propo:chs:hol}]

By \eqref{eq:chs:metric} the solution of the IIB system in $\C \times U$ is a Riemannian product of flat $\C$ with
$(U,g_h)$, the complex structure is a product complex structure and $H$ is pulled back from $U$. Hence, the Bismut holonomy of the solution is $\Hol(\nabla^{\tu B,h}) \subset \sunitary{3}$, acting trivially on the $\C$ factor. We assume that $\alpha' \neq 0$, since the case $\alpha'= 0$ is trivial. Fix a centre $p_i$. By definition, $b_i>0$ unless $c = 0$ and $n=1$. This case is flat by Corollary \ref{cor:chs:flat}, and hence the statement also holds. From now on we  assume this is not the case.  Recall that $\nabla^{\tu B, h}$ has holonomy contained in $\sunitary{2}$, and therefore 
$$
R^{\tu B,h}_z(e_a,e_b) = \mathcal{R}(Q_0)(e_a,e_b) \in \mathfrak{su}(2) \cong \Lambda^2_-.
$$
Furthermore, by the Ambrose--Singer theorem, the holonomy algebra $\hol_z(\nabla^{\tu B,h})$ contains all the endomorphisms $R^{\tu B,h}_z(e_a,e_b)$. We claim that, for $z$ sufficiently close to $p_i$, $\hol_z(\nabla^{\tu B,h}) = \mathfrak{su}(2)$. To prove this, it suffices to show that the endomorphisms $\mathcal{R}(Q_0(z))(e_a,e_b)$ span a three-dimensional subspace. We set $\gamma:=-8b_i/a_i \in \R_{<0}$. Fix a point $z_0 \in B_{\rho_0}(p_i)\setminus\{p_i\}$ and denote by $\hat x_0$ the associated radial vector field of unit length with respect to $\delta$. We choose the $\delta$-orthonormal basis $e_j$ such that $e_0 = \hat x_0$, and positively oriented. For points $z$ along the radius $\ell = \{p_i + \rho \hat x_0 \; | \; 0 < \rho < \rho_0\}$, by Lemma \ref{lem:chs:asymptotics} we have 
$$
Q_0(z) = T + O(\rho),
$$
where $T:=\gamma\big(
\hat x_0^\flat\otimes \hat x_0^\flat-\tfrac14\delta\big)$. The eigenvalues of $T$ are given by
\[
q_0=\tfrac34\gamma ,\qquad
q_1=q_2=q_3=-\tfrac14\gamma.
\] 
Furthermore, applying \eqref{eq:chs:Rab}, we have
\begin{equation*}
\mathcal{R}(T)(e_0,e_j)=(q_0+q_j)\,[e^{0j}]_- = \tfrac12\gamma[e^{0j}]_- \in \Lambda^2_-,
\end{equation*}
and these span $\Lambda^2_-$. Since spanning a three-dimensional subspace is an open condition, the claim follows. Finally $U$ is simply connected, so
$\Hol(\nabla^{\tu B,h})=\Hol^0(\nabla^{\tu B,h})$ is the connected subgroup with Lie algebra
$\mathfrak{su}(2)$, that is, $\Hol(\nabla^{\tu B,h})=\sunitary{2}$.
\end{proof}

\subsection{Explicit solutions of the $\sunitary{2}^2$--invariant IIB system} We will now describe some explicit solutions of the $\sunitary{2}^2$--invariant IIB system, \ie explicit solutions $(u,\phi)$ of the ODE system of Proposition \ref{prop:Coho1:IIB} in either of the two complex structures. In fact for each choice of complex structure and parameters $(\alpha',b,0)$ and, respectively, $(\alpha',0,\kappa)$ we will find an explicit solution $(u_\ast,\phi_\ast)$, which we will refer to as the \emph{Chamseddine--Volkov/Maldacena--Nu\~nez (CV--MN) solution}, since as far as we know these solutions first appeared in the work of these authors.

\subsubsection{The CV--MN solution on the resolved conifold}

We will derive the CV--MN solution on the resolved conifold by attempting to fibre the Boothby metric on $\C^2\setminus\{0\}$ over $\C\PP^1$ using the Chern connection of $\mathcal{O}(-1)=\mathcal{O}_{\C\PP^1}(-1)$.

Let $\eta_1,\eta_2,\eta_3$ and $\eta'_1, \eta'_2, \eta'_3$ be left invariant $1$-forms on two copies of $S^3\simeq \tu{SU}(2)$ satisfying $d\eta_i = 2\eta_{j}\wedge\eta_k$ and $d\eta'_i=2\eta'_j\wedge\eta'_k$ for $(ijk)$ a cyclic permutation of $(123)$. Identify $\C\PP^1\simeq S^2$ with $\sunitary{2}/\unitary{1}$, where $\sunitary{2}$ is the first copy of $S^3$. Note that $[2\eta_1\wedge\eta_2\wedge\eta_3]$ is the generator of $H^3(S^3;4\pi^2\Z)$ and $2\eta_2\wedge\eta_3$ the generator of $H^2(S^2;2\pi\Z)$. Then we can think of $\eta_1$ as the unique rotationally invariant connection $1$-form on the principal circle bundle associated with $\mathcal{O}(-1)\ra\C\PP^1=\sunitary{2}/\unitary{1}$ and $2\eta_2\wedge\eta_3$ is the Fubini--Study K\"ahler form normalised to have area $2\pi$.

\begin{remark*}
In order to resolve some apparent confusion, observe that $\eta_1$ is a real $1$-form, so the Chern connection on $\mathcal{O}(-1)$ is induced by $i\eta_1$. The curvature of this connection is $F=d(i\eta)=2i\eta_2\wedge\eta_3$ and indeed $\frac{i}{2\pi}F$ has integral $-1$ over $\C\PP^1$.
\end{remark*}

Regard instead the second copy of $S^3$ as the unit sphere in $\C^2\setminus\{0\}$. In order to align with the conventions of the previous section, where we consider cohomogeneity one 6-manifolds with principal orbits $\sunitary{2}^2/\triangle\unitary{1}$, we need to orient Hopf fibres by $-\eta_1'$. Letting $t$ denote the radial parameter, note that the Boothby $\sunitary{2}$--structure on $\C^2\setminus\{ 0\}$ of the previous section can be written in spherical coordinates as
\begin{equation}\label{eq:Boothby:rotational}
\omega_{\tu{B}}=-\frac{dt}{t}\wedge\eta_1' - \eta_2'\wedge\eta_3', \qquad \Omega_{\tu{B}} = \frac{1}{t^2}(dt-it\eta_1')\wedge t (\eta_2'-i\eta_3').
\end{equation}

We consider the complement $M$ of the $0$-section in the vector bundle $\mathcal{O}(-1)\times_{\tu{U}(1)}\C^2=\mathcal{O}(-1)\oplus\mathcal{O}(-1)$ over $\C\PP^1$. The fibres of $M\ra\C\PP^1$ are copies of $\C^2\setminus\{ 0\}$. We can use the Chern connection of $\mathcal{O}(-1)$ to fibre the Boothby metric over $\C\PP^1$. This simply means replacing $-\eta'_1$ in \eqref{eq:Boothby:rotational} with $\eta=-\eta_1'+\eta_1$, yielding the $\sunitary{3}$--structure on $M$ given by
\[
\omega' = c\,\eta_2\wedge\eta_3 +   \alpha' \left( \frac{dt}{t}\wedge\eta - \eta'_2\wedge\eta'_3\right), \qquad \Omega' = e^{2\phi'}(dt+it\eta)\wedge t(\eta'_2-i\eta'_3)\wedge (\eta_2+i\eta_3)
\]
for positive constants $c$ and $\alpha'$. The function $e^{2\phi'}$ is determined by the $\sunitary{3}$--structure volume forms constraint and one can easily calculate that $e^{2\phi'}$ is a constant multiple of $t^{-2}$. The metric induced by $(\omega',\Omega')$ is a Riemannian submersion with base metric a multiple of the Fubini--Study metric on $\C\PP^1$ and fibres given by the Boothby metric on $\C^2\setminus\{0\}$ (rescaled by a factor $\alpha'$). One can check that $d(e^{-2\phi'}\Omega)=0$: in fact the induced complex structure $J$ is the standard one on $\mathcal{O}(-1)\oplus\mathcal{O}(-1)$. However, $(\omega',\Omega')$ is not a solution of the IIB system since neither $d^c\omega'$ nor $t^2\omega'\wedge\omega'$ are closed.

In a way reminiscent to other bundle constructions of special metrics starting from Calabi's construction of K\"ahler Ricci-flat metrics \cite{Calabi:Ansatz}, we therefore modify our ansatz introducing a positive function $f=f(t)$. Define an $\sunitary{3}$--structure $(\omega,\Omega)$ on the open set of $M$ where $f>0$ by
\begin{subequations}\label{eq:5brane:P1}
\begin{equation}
\omega = f\, \eta_2\wedge\eta_3 + \alpha' \left( \frac{dt}{t}\wedge\eta - \eta'_2\wedge\eta'_3\right), \qquad \Omega = e^{2\phi}(dt+it\eta)\wedge t(\eta'_2-i\eta'_3)\wedge (\eta_2+i\eta_3).
\end{equation}
Here $e^{2\phi}$ is uniquely defined so that $(\omega,\Omega)$ is an $\sunitary{3}$--structure. A rapid computation shows that
\begin{equation}
e^{2\phi}=\alpha' t^{-2}\sqrt{f(t)}.
\end{equation}
Using $Jdt=t\eta$ and the fact that $\eta_2\wedge\eta_3$ and $\eta'_2\wedge\eta'_3$ are of type $(1,1)$ with respect to $J$, we then calculate
\[
d^c\omega = (tf' -2\alpha')\,\eta\wedge\eta_2\wedge\eta_3 + 2\alpha' \eta\wedge\eta'_2\wedge\eta'_3,
\]
so that $dd^c\omega=0$ if and only if $f' = 4\alpha't^{-1}$, \ie
\begin{equation}\label{eq:}
f(t) = \tfrac{4}{3}\left(b-3\alpha'\log{\tfrac{2}{3}}+\tfrac{3}{4}\alpha'\right)+ 4\alpha' \log{t}
\end{equation}
\end{subequations}
for some constant $b$. (The reason for the particular way of writing the constant term of $f$ will become clear momentarily.) Here we used $d\eta=2\eta_2\wedge\eta_3- 2\eta_2'\wedge\eta'_3$ and the fact that each of the two summands on the right-hand side of this formula are closed. It turns out that the definition of $f$ is also equivalent to $d(e^{-2\phi}\omega^2)=0$. Indeed, we calculate
\[
\tfrac{1}{2}e^{-2\phi}\omega^2 = tf^\frac{1}{2} dt\wedge\eta\wedge\eta_2\wedge\eta_3 - \alpha' t f^{-\frac{1}{2}}dt\wedge\eta\wedge\eta'_2\wedge\eta'_3 - t^2 f^\frac{1}{2}\eta_2\wedge\eta_3\wedge\eta'_2\wedge\eta'_3.
\]
We conclude that \eqref{eq:5brane:P1} with $\alpha'>0$ defines a solution of the IIB system defined on the open set $\{ f(t) > 0\}$ of $M$.

In order to study the geometry of this solution, introduce the new variable
\[
\tau = \sqrt{\alpha'}\log{t} +\tfrac{1}{3\sqrt{\alpha'}}\left(b-3\alpha'\log{\tfrac{2}{3}}+\tfrac{3}{4}\alpha'\right),
\]
the arc-length parameter along a geodesic that meets all hypersurfaces $t=\tu{const}$ orthogonally. The metric induced by $(\omega,\Omega)$ takes the form
\[
d\tau^2 + 4\sqrt{\alpha'} \tau (\eta_2^2 + \eta_3^2) + \alpha' (\eta^2 + (\eta'_2)^2 + (\eta'_3)^2).
\]
For $\tau\ra 0$ this metric is incomplete, while for $\tau\ra\infty$ it is complete. The (in)completeness behaviour of this metric gives rise to an asymptotic geometry of an exotic kind, with an end foliated by parallel hypersurfaces that are Riemannian submersions over a $2$-sphere of radius proportional to the square root of the distance and constant 3-sphere fibres. Here the metric on the 3-sphere fibres is a round metric, but we can generalise this ansatz replacing the round metric on $S^3$ with (given the $\sunitary{2}^2$--invariant constraint within which we work in this paper) any Berger metric: given $\alpha',c>0$ define
\begin{equation}\label{eq:CVMN:end}
g_{\tu{CVMN}}(\alpha',c)= d\tau^2 + 4\sqrt{\alpha'} \tau (\eta_2^2 + \eta_3^2) + \alpha' \left(\eta^2 + c\left((\eta'_2)^2 + (\eta'_3)^2\right)\right).
\end{equation}

\begin{definition}\label{def:CVMN:Ends}
Let $(M,g)$ be a Riemannian manifold diffeomorphic to $I\times S^2\times S^3$ with $I\subset \R_{>0}$. We say that $(M,g)$ is
\begin{enumerate}
\item a \emph{complete CV--MN end} with parameters $(\alpha',c)$ if $I=(\tau_0,\infty)$ for some $\tau_0>0$ and $g-g_{\tu{CVMN}(\alpha',c)}=o(1)$ as $\tau\ra\infty$;
\item an \emph{incomplete CV--MN end} with parameters $(\alpha',c)$ if $I=(0,\tau_0)$ for some $\tau_0>0$ and $g-g_{\tu{CVMN}}(\alpha',c)=o(1)$ as $\tau\ra 0$.
\end{enumerate}    
\end{definition}

Now, there is a natural action of $\sunitary{2}^2$ on $M$ and the solution $(\omega,\Omega)$ of \eqref{eq:5brane:P1} is clearly invariant. We conclude this discussion by expressing this solution in terms of the parametrisation of Proposition \ref{prop:Coho1:IIB}.

First of all, in terms of the coframe $\{ \eta_i,\eta'_i\}$ the Sasaki--Einstein $\sunitary{2}$--structure $(\eta^\tu{se},\omega_1^\tu{se},\omega_2^\tu{se},\omega_3^\tu{se})$ and primitive $(1,1)$--form $\omega_0^\tu{se}$ on $\sunitary{2}^2/\triangle\unitary{1}$ can be written as in \eqref{eq:Homogeneous:SE:LeftInvariant}. Then after the change of variable
\begin{equation}\label{eq:SE:to:5brane}
4r^3=9t^2
\end{equation}
for $t>0$, the holomorphic volume form
\[
\Omega_0 = e^{-2\phi}\Omega =(dt+it\eta)\wedge t(\eta'_2-i\eta'_3)\wedge (\eta_2+i\eta_3) 
\]
is nothing but the holomorphic volume form $\Omega_\tu{C}$ of the conifold. Using \eqref{eq:SE:to:5brane} (note that in particular $2t^{-1}dt=3r^{-1}dr$) we can also write
\[
\omega = \tfrac{9}{4} r^{-1}dr\wedge\eta^{\tu{se}} + \tfrac{3}{4}(f+\alpha')\omega_1^\tu{se} + \tfrac{3}{4}(f-\alpha')\omega_0^\tu{se}.
\]
In other words, in the notation of Proposition \ref{prop:Coho1:IIB} we have
\[
u_0 = \tfrac{3}{4}(f-\alpha'), \qquad u = \tfrac{3}{4}(f+\alpha').
\]

\begin{definition}\label{def:CVMN:Resolved}
    The \emph{CV--MN solution on the resolved conifold} with parameters $(\alpha',b)$ is the solution $(u_\ast,\phi_\ast)$ of the ODE system of Proposition \ref{prop:Coho1:IIB} with $V=e^{3s}$ and $u_0 = \frac{9}{2}\alpha' s + b$ given by
    \[
    u_\ast = \tfrac{9}{2}\alpha' s + b + \tfrac{3}{2}\alpha', \qquad e^{2\phi_\ast} = \tfrac{1}{\sqrt{3}}\tfrac{9}{2}\alpha'e^{-3s} \sqrt{\tfrac{9}{2}\alpha' s + b +\tfrac{3}{4}\alpha'}.
    \]    
\end{definition}

Note that one has $\phi_\ast\ra -\infty$ as $s\ra\infty$, in contrast to the asymptotic behaviour $\phi\ra\phi_\infty$ for some $\phi_\infty\in\R$ displayed by AC solutions, \ie solutions asymptotic to $(\omega_\tu{C},\Omega_\tu{C},\phi_\infty)$.

\subsubsection{The CV--MN solution on the deformed conifold}

As discovered in the physics literature by Chamseddine--Volkov \cite{Chamseddine:Volkov,Chamseddine:Volokv:II} and Maldacena--Nu\~nez \cite{Maldacena:Nunez}, the ODE system of Proposition \ref{prop:Coho1:IIB} also has an explicit solution when $V$ and $u_0$ are the ones corresponding to the complex structure of the defomed conifold.

\begin{definition}\label{def:CVMN:Deformed}
    The \emph{CV--MN solution on the deformed conifold} with parameters $(\alpha',\kappa)$ is the solution $(u_\ast,\phi_\ast)$ of the ODE system of Proposition \ref{prop:Coho1:IIB} with $u_0=\tfrac{3}{2}\alpha' \frac{3s \cosh{(3s)} -\sinh{(3s)}}{\sinh{(3s)}}$ and $V = \kappa \sinh{3s}$ given by
    \[
    u_\ast = \tfrac{9}{2}\alpha' s , \qquad e^{2\phi_\ast} = \frac{(\tfrac{9}{2}|\alpha'|)^\frac{3}{2}}{3\sqrt{2}\kappa}\frac{ \sqrt{9s^2\sinh^2{(3s)} -(3s \cosh{(3s)}-\sinh{(3s)})^2}}{ \sinh^2{(3s)}}.
    \]    
\end{definition}

Checking that $u_\ast>|u_0|$ for all $s>0$ and that $(u_\ast,\phi_\ast)$ is a solution of the ODE system are explicit computations. We briefly discuss asymptotics of the solution as $s\ra\infty$ and $s\ra 0$.

For $s\ra\infty$ we calculate
\[
u_\ast = \tfrac{9}{2}\alpha'\, s, \qquad e^{2\phi_\ast} \approx \tfrac{2}{\sqrt{3}\kappa}\left( \tfrac{9}{2}\alpha'\right)^{\frac{3}{2}}e^{-3s}\sqrt{s-\tfrac{1}{6}}
\]
up to exponentially decaying terms. Together with the exponential decay (in the variable $s$) of the complex structure of the deformed conifold to the one of the conifold up to a change of variable $s\mapsto s + \frac{1}{3}\log{\frac{\kappa}{2}}$, we therefore see that the CV--MN solution on the deformed conifold has a complete CV--MN end with parameters $(\alpha',1)$ (and $b=-\frac{3}{2}\alpha'(1+\log{\frac{\kappa}{2}})$).

It will follow from the analysis of initial conditions in the next section that the solution determined by $(u_\ast,\phi_\ast)$ extends smoothly across the zero section, \ie $(u_\ast,\phi_\ast)$ in Definition \ref{def:CVMN:Deformed} defines a complete solution of the IIB system on $T^\ast S^3$, see Remark \ref{rmk:SmoothLCVMN:Deformed} below. For now we record the following expansion near $s=0$:
\begin{equation}\label{eq:Dilaton0:CVMN:deformed}
\phi_\ast = \tfrac{3}{4}\log{\tfrac{9}{2}|\alpha'|} - \tfrac{1}{2}\log{3\sqrt{2}\kappa} +O(s^2).
\end{equation}

\section{Local solutions}

In this section we study natural singular initial value problems for the ODE system of Proposition \ref{prop:Coho1:IIB} that arise when studying solutions that extend smoothly over a singular orbit or solutions with a prescribed CV--MN singular end. In either case, the main technical tool is the following existence result for singular initial value problems, \cf \cite[Theorem 7.1]{Malgrange}, \cite[\S 4]
{Ferus:Karcher} and \cite[\S 5]{Eschenburg:Wang}.

\begin{theorem}[{\cf \cite[Theorem 4.7]{Foscolo:Haskins}}]\label{thm:Singular:BVP}
Consider the singular initial value problem
\begin{equation}\label{eq:Singular:IVP}
z'= \frac{1}{s}\, M_{-1}(z) + M(s,z), \qquad z(0)=z_0,
\end{equation}
where $z$ is a function of the independent variable $s$ with values in $\R^k$, $z'$ denotes its derivative, $M_{-1}\co \R^k\ra\R^k$ is a smooth function of $z$ in a neighbourhood of $z_0$ and $M\co \R\times\R^k\ra \R^k$ is smooth in $(s,z)$ in a neighbourhood of $(0,z_0)$. Assume that
\begin{enumerate}
\item $M_{-1}(z_0)=0$;
\item $d_{z_0}M_{-1}-h\,\tu{id}_{\R^k}$ is invertible for all $h\in\N_{\geq 1}$.
\end{enumerate}
Then there exists a unique solution $z(s)$ of \eqref{eq:Singular:IVP}. Furthermore $z$ depends continuously on $z_0$ satisfying (i) and (ii).
\end{theorem}

The main results of this section are Propositions  \ref{prop:Coho1:IIB:Deformed:IVP} and \ref{prop:Coho1:IIB:Resolved:IVP}. The former establishes the existence, for each fixed choice of parameters $(\alpha',\kappa)$, of a 1-parameter family of smooth solutions to the IIB system defined in a tubular neighbourhood of the zero-section in $T^\ast S^3$. Working with the complex structure of the conifold, topological reasons forbid the existence of smooth non-K\"ahler solutions on either small resolution of the conifold and on $K_{\C\PP^1\times\C\PP^1}$, but Proposition \ref{prop:Coho1:IIB:Resolved:IVP} establishes the existence, for each fixed choice of parameters $(\alpha',b)$, of a 1-parameter families of CV--MN singular ends.

\subsection{Smooth solutions in a neighbourhood of a singular orbit}

In Theorem \ref{thm:ACCY} we have seen that there are four smooth $\sunitary{2}^2$--invariant Calabi--Yau manifolds up to scaling: $T^\ast S^3$, the two small resolutions of the conifold and $K_{\C\PP^1\times\C\PP^1}$. Each of these manifolds contains a dense open set, the set of \emph{principal orbits}, diffeomorphic to $\R\times \sunitary{2}^2/\triangle\unitary{1}$ (up to a double cover in the case of $K_{\C\PP^1\times\C\PP^1}$). The space of principal orbit is partially compactified by adding a smaller orbit for the $\sunitary{2}^2$--action, called the \emph{singular orbit}. The singular orbit in each of the four cases is, respectively,
\[
\sunitary{2}^2/\triangle\sunitary{2}, \qquad \sunitary{2}^2/\unitary{1}\times\sunitary{2}, \qquad \sunitary{2}^2/\sunitary{2}\times\unitary{1}, \qquad \sunitary{2}^2/\unitary{1}^2.
\]
Solutions of the ODE system of Proposition \ref{prop:Coho1:IIB} describe solutions of the IIB system on the set of principal orbits; suitable boundary conditions will then determine whether the $\sunitary{3}$--structure extends to a smooth solution of the IIB system across the relevant singular orbit.

In Lemma \ref{lem:No:Smoothness:Resolutions} we observed that for topological reasons there are no non-K\"ahler smooth $\sunitary{2}^2$--invariant solutions on either of the two small resolutions of the conifold or on $K_{\C\PP^1\times\C\PP^1}$. In other words, $\alpha' \neq 0$ obstructs the existence of smooth solutions of \eqref{eq:IIBweak} in a neighbourhood of the singular orbit of any crepant resolution of the conifold $\tu{C}$ and its quotient $\tu{C}/\mathbb{Z}_2$. In view of this observation, we focus our attention on smooth solutions of the IIB system defined in a tubular neighbourhood of the 0-section in $T^\ast S^3$. In the next result we will therefore work with the complex structure of the deformed conifold, case (ii) of Proposition \ref{prop:Coho1:IIB}. 

\begin{prop}\label{prop:Coho1:IIB:Deformed:IVP}
For each choice of parameters $\kappa>0$ and $\alpha'\in \R$, there exists a $1$-parameter family of solutions to the ODE system of Proposition \ref{prop:Coho1:IIB} with
\[
V = \kappa \sinh{3s}, \qquad u_0=\tfrac{3}{2}\alpha' \frac{3s \cosh{(3s)} -\sinh{(3s)}}{\sinh{(3s)}}
\]
which corresponds to a 1-parameter family of smooth solutions of the IIB system defined in a tubular neighbourhood of the 0-section $\sunitary{2}^2/\triangle\sunitary{2}$ in $T^\ast S^3$. The family can be parametrised by the value of $\phi$ on the singular orbit.
\proof
We will reformulate the smooth extension across the singular orbit as a singular boundary value problem for an ODE system of the form considered in Theorem \ref{thm:Singular:BVP}.

The necessary and sufficient conditions for a $\sunitary{2}^2$--invariant non-degenerate 2-form to extend over the singular orbit $\sunitary{2}^2/\triangle\sunitary{2}$ are described in \cite[Lemma 4.2]{Foscolo:Haskins}. After taking into account the different conventions from that paper, this conditions require that, with respect to the arc-length parameter $t$, $u, u_0\nu_3, u_0\nu_0$ are odd, $\lambda$ is even with $\lambda(0)>0$ and
\[
u_0 (\nu_3+\nu_0)=O(t^3), \qquad u(t)=2\lambda(0) t + O(t^3).
\]
Note the last condition also arises from the equation $\partial_t u = 2\lambda$, while the condition that the derivative of $u_0 (\nu_3+\nu_0)$ vanishes at the origin was already used to determine a constant of integration in the proof of Lemma \ref{lem:BIexplicit}. Since $\lambda(0)\neq 0$, the relation $\lambda\, ds = dt$ implies that we can work with these same smoothness conditions with the parameter $s$ in which the coefficients of the ODE system become explicit. Finally, since the complex structure and holomorphic volume form $\Omega_0$ is already smooth, there are no additional smoothness conditions arising from smoothness of the function $\phi$ since $e^\phi$ is defined implicitly by $\omega$ and $\Omega_0$. More directly, since $t$ is a radial parameter on the fibres of $T^\ast S^3\ra S^3$, it is clear that smoothness of $\phi$ requires $\phi$ to be an even function of $t$ (equivalently, of $s$).

Now, given $\phi_0\in\R$, write
\[
u(s) = s X(s), \qquad \phi = \phi_0 + s^2 Y(s)
\]
for functions $X,Y$. Using
\[
e^{-2\phi}\lambda\mu \approx 3\kappa s, \qquad u_0 \approx \tfrac{9}{2}\alpha' s^2,
\]
for $s\ra 0$ we calculate that the ODE system of Proposition \ref{prop:Coho1:IIB} has the form
\[
s\, \partial_s z = M_{-1}(z) + M(s,z)
\]
for the pair $z=(X,Y)$, where 
\[
M_{-1}(z) = \left( -X + 18\kappa^2 e^{4\phi_0}X^{-2}, -2Y -2\left( \tfrac{9}{2}\alpha'\right)^2 X^{-2}\right)
\]
and $M(s,z)$ is a real analytic function with $M(0,z)=0$.

According to Theorem \ref{thm:Singular:BVP}, in order to understand the existence of local solutions of this system we have to parametrise all vectors $z_0$ such that $M_{-1}(z_0)=0$ and then check the non-resonance condition
\[
\det{\left( d_{z_0}M_{-1}-h\tu{Id}\right)} \neq 0,\qquad \forall h\in \mathbb{N}_{\geq 1}.
\]
A direct calculation implies that the initial condition $z_0=(X_0,Y_0)$ is uniquely given by
\[
X_0 = (18\kappa^2)^{\frac{1}{3}}e^{\frac{4}{3}\phi_0}, \qquad Y_0 = -\left( \tfrac{9}{2}\alpha'\right)^2(18\kappa^2)^{-\frac{2}{3}}e^{-\frac{8}{3}\phi_0}
\]
once the parameters $\alpha',\kappa,\phi_0$ are given. One then calculates that
\[
\det\,\left( d_{z_0}M_{-1}-h\,\tu{id}\right) = (h+3)(h+2)
\]
so that Theorem \ref{thm:Singular:BVP} guarantees the existence of a solution of the ODE system defined for sufficiently small $s$ (depending on $\alpha',\kappa$ and $\phi_0$). Finally, uniqueness of the solution implies that $u$ is in fact an odd function of $s$ and $\phi$ an even function of $s$, implying that the smoothness conditions for extension across the singular orbit are automatically satisfied.
\endproof
\end{prop}

\begin{remark}\label{rmk:Coho1:IIB:Deformed:IVP}
For later use, we record here higher order expansions in terms of the parameter $s$:
\[
u = \left( 18\kappa^2\right)^{\frac{1}{3}} e^{\frac{4}{3}\phi_0} \left( 1 + \tfrac{3}{5}\left( 1- \left( \tfrac{9}{2}\alpha'\right)^2 \left( 18\kappa^2\right)^{-\frac{2}{3}} e^{-\frac{8}{3}\phi_0}\right)s^2 + O(s^4)\right)\, s.  
\]
The consequence we will need is that the quantities $u$, $\partial_s u$ and
\[
\frac{\partial_s u}{u} = \frac{1}{s}+ \tfrac{3}{5}\left( 1- \left( \tfrac{9}{2}\alpha'\right)^2 \left( 18\kappa^2\right)^{-\frac{2}{3}} e^{-\frac{8}{3}\phi_0}\right)s + O(s^3)
\]
calculated at any fixed time $s>0$ sufficiently small are all increasing in $\phi_0$.
\end{remark}

\begin{remark}\label{rmk:SmoothLCVMN:Deformed}
By uniqueness of the solution in Theorem \ref{thm:Singular:BVP} the solution with $\phi_0 = \phi_\ast(0)$ of \eqref{eq:Dilaton0:CVMN:deformed} must coincide with the CV--MN solution of Definition \ref{def:CVMN:Deformed}. As promised, in particular this shows that the latter defines a smooth complete non-K\"ahler solution of the IIB system on $T^\ast S^3$.
\end{remark}

\begin{remark}\label{rmk:mThmAi}
Setting $p=\phi_0$, Proposition \ref{prop:Coho1:IIB:Deformed:IVP} yields a proof of part (i) of Theorem \ref{mthm:IIB:Deformed} in the Introduction, with $p_\ast = \phi_\ast (0)$.
\end{remark}

\subsection{Solutions with CV--MN incomplete behaviour}

By Lemma \ref{lem:No:Smoothness:Resolutions} there are no non-K\"ahler smooth solutions of the cohomogeneity one IIB system on any crepant resolution of the conifold or of its $\Z_2$--quotient. The explicit CV--MN solution of Definition \ref{def:CVMN:Resolved} however provides an interesting incomplete solution. We will now establish the existence of a 1-parameter family of solutions with singular CV--MN ends with parameters $(\alpha',c)$ in the sense of Definition \ref{def:CVMN:Ends}, parametrised by the squashing parameter $c>0$ of the Berger metric on the $S^3$--fibres of the CV--MN end.

Work in the complex structure of the conifold and consider the generic $\sunitary{2}^2$--invariant non-degenerate $2$-form
\[
\omega = \tfrac{1}{2}u' \, ds\wedge\eta^{\tu{se}} + u\, \omega_1^{\tu{se}} + u_0\, \omega^{\tu{se}}_0.
\]
Using \eqref{eq:SE:to:5brane} we rewrite
\[
\omega = \tfrac{1}{3} u' \, ds\wedge\eta - \tfrac{2}{3}(u-u_0)\, \eta_2'\wedge\eta_3' + \tfrac{2}{3}(u+u_0)\, \eta_2\wedge\eta_3.
\]
Fix $c>0$ and define $u=u_c$ where
\begin{equation}\label{eq:CVMN:End:u}
u_c = u_0 + \tfrac{3}{2}c\alpha'.
\end{equation}
Then
\[
\omega = \tfrac{3}{2}\alpha'\, ds\wedge\eta - c\alpha' \, \eta_2'\wedge\eta_3' + \tfrac{4}{3}\left( \tfrac{9}{2}\alpha' s + b + \tfrac{3}{4}c\alpha'\right)\, \eta_2\wedge\eta_3.
\]
After a change of variable $\sqrt{\alpha'}\tau = \frac{1}{3}\left( \frac{9}{2}\alpha' s + b + \frac{3}{4}c\alpha'\right)$ the metric induced by $\omega$ and the complex structure of the conifold is seen to be precisely the model metric $g_\tu{CVMN}(\alpha',c)$ of \eqref{eq:CVMN:end}.

It is convenient to reparametrise by a translation in $s$ in terms of a parameter $\sigma$ such that $\sigma=0$ corresponds to the value $s_c$ of $s$ for which $u_c+u_0=0$: in other words, $\sigma = s+\left(\frac{9}{2}\alpha'\right)^{-1}\left( b+\frac{3}{4}c\alpha'\right)$. Then 
\begin{equation}\label{eq:u0uc}
u_0 = \frac{9}{2}\alpha'\sigma-\frac{3}{4}c\alpha', \qquad u_c = \frac{9}{2}\alpha'\sigma+\frac{3}{4}c\alpha'.
\end{equation}
and we define a function $\phi_c$ implicitly by
\[
\frac{2e^{6s_c}e^{6\sigma}e^{4\phi_c}}{u_c^2-u_0^2}=\tfrac{9}{2}\alpha'.
\]
Note that $u_c^2-u_0^2=\frac{2}{3}c\, \left(\frac{9}{2}\alpha'\right)^2 \sigma$ and $\dot{\phi}_c = -\frac{3}{2}+\frac{1}{4\sigma}$. Then
\[
\dot{u}_c = \frac{2e^{6s_c}e^{6\sigma}e^{4\phi_c}}{u_c^2-u_0^2}, \qquad \dot{\phi}_c + \frac{u_0\dot{u}_0}{u_c^2-u_0^2} = -\frac{3}{2}\left( 1- c^{-1}\right)
\]
so that $(u_c,\phi_c)$ is a solution to the system of Proposition \ref{prop:Coho1:IIB} if and only if $c=1$, \ie $(u_c,\phi_c)=(u_\ast,\phi_\ast)$ is the CV--MN solution on the resolved conifold of Definition \ref{def:CVMN:Resolved}.

\begin{prop}\label{prop:Coho1:IIB:Resolved:IVP}
For each choice of parameters $b\in\R$ and $\alpha'>0$, there exists a $1$-parameter family of solutions to the ODE system of Proposition \ref{prop:Coho1:IIB} with
\[
V=e^{3s}, \qquad u_0 = \tfrac{9}{2}\alpha' s + b,
\]
which corresponds to a 1-parameter family of solutions of the IIB system in the complex structure of the conifold with a singular CV--MN end with parameters $(\alpha',c)$ for some $c>0$. The family can be parametrised by the squashing parameter $c>0$.
\proof
Let $(u_c,\phi_c)$ be the functions defined above. We look for a solution $(u,\phi)$ of the ODE system of the form $u=u_c+\sigma^2 X$, $\phi = \phi_c +\sigma Y$ for functions $X,Y$ defined for small $\sigma>0$, \ie a solution $(u,\phi)$ that has the singular behaviour $(u_c,\phi_c)$ as $\sigma\ra 0$. We calculate
\begin{align*}
\sigma^2 \dot{X} + 2\sigma X &= -\tfrac{9}{2}\alpha' +  \tfrac{9}{2}\alpha' \left( 1+ 4\sigma Y+O(\sigma^2)\right) \left( 1 - \tfrac{1}{2}(\tfrac{9}{2}\alpha')^{-1}\sigma X + O(\sigma^2)\right)\\
&= 4\left( \tfrac{9}{2}\alpha' \right) \sigma Y - \tfrac{1}{2}\sigma X + O(\sigma^2)
\end{align*}
and
\begin{align*}
\sigma \dot{Y} + Y &= \tfrac{3}{2}-\tfrac{1}{4\sigma} -\frac{\tfrac{9}{2}\alpha'\sigma -\tfrac{3}{4}c\alpha'}{3c\alpha'\sigma} \left( 1 - \tfrac{1}{2}(\tfrac{9}{2}\alpha')^{-1}\sigma X + O(\sigma^2)\right)\\
&= \tfrac{3}{2}\left( 1-c^{-1}\right) -\tfrac{1}{8}\left( \tfrac{9}{2}\alpha'\right)^{-1} X + O(\sigma)
\end{align*}
where $O(\sigma^h)$ denotes a real analytic function in $X,Y$ with coefficients that depend real analytically on $\sigma$ and vanish to order $h$ at $\sigma=0$.  We conclude that $z=(X,Y)$ satisfies an ODE system of the form
\[
\sigma \dot{z} = M_{-1}(z) + M_1(\sigma,z), \qquad M_{-1}(z) = \left( -\tfrac{5}{2}X +4\left( \tfrac{9}{2}\alpha'\right) Y, \tfrac{3}{2}\left( 1-c^{-1}\right) -\tfrac{1}{8}\left( \tfrac{9}{2}\alpha'\right)^{-1} X - Y \right),
\]
and with $M_1(\sigma,z)=O(\sigma)$. There is a unique solution $z_0=(X_0,Y_0)$ to $M_{-1}(z_0)=0$, given by
\[
X_0 = 2 \left( \tfrac{9}{2}\alpha'\right)  \left( 1-c^{-1}\right) ,\qquad Y_0= \tfrac{5}{4} \left( 1-c^{-1}\right).
\]
We then compute
\[
d_{z_0}M_{-1} = \left( \begin{array}{cc} -\tfrac{5}{2} & 4\left( \tfrac{9}{2}\alpha'\right)\\ -\tfrac{1}{8}\left( \tfrac{9}{2}\alpha'\right)^{-1} & -1  \end{array}\right), \qquad \det \left( d_{z_0}M_{-1} - h\tu{id}\right) = h^2 + \tfrac{7}{2}h+3.
\]
By Theorem \ref{thm:Singular:BVP} we conclude that for every $c>0$ there exists a unique solution $(u,\phi) = (u_c,\phi_c) + (\sigma^2 X, \sigma Y)$ of the ODE system as claimed.
\endproof
\end{prop}

\begin{remark}\label{rmk:Coho1:IIB:Resolved:IVP}
Setting $q=\log{c}$, the proposition yields a proof of part (i) of Theorem \ref{mthm:IIB:Resolved} in the Introduction. In terms of the notation introduced in the statement of that theorem, via the change of variable $e^s=r$ we calculate $q_\ast=0$, $\log{R_0} = -\frac{2}{9}\frac{b}{\alpha'}$ and $\log{R(q)}=-\frac{1}{6}e^q +\log{R_0}$. In particular, $R(q)\ra 0$ as $q\ra\infty$ and $R(q)\ra R_0$ as $q\ra -\infty$.
\end{remark}

\section{Forward complete solutions}

Our goal now is to study forward completeness of the local solutions on the deformed conifold $T^\ast S^3$ given by Proposition \ref{prop:Coho1:IIB:Deformed:IVP} and of the incomplete solutions on the conifold given in Proposition \ref{prop:Coho1:IIB:Resolved:IVP}. For either choice of complex structure, these propositions establish the existence of 1-parameter families of local solutions with interesting initial conditions; the CV--MN solutions yield a distinguished point in each 1-parameter family. The main results of this section are the complete asymptotically conical behaviour for members of the 1-parameter families lying ``above'' the relevant CV--MN solution, and incompleteness of members of the family lying ``below'' the latter.

\begin{theorem}\label{thm:Coho1:IIB:Deformed}
For each choice of $\alpha'\in\R\setminus\{0\}$ and $\kappa>0$, consider the 1-parameter family of $\sunitary{2}^2$--invariant local solutions of the IIB system \eqref{eq:IIB} on the deformed conifold constructed in Proposition \ref{prop:Coho1:IIB:Deformed:IVP}. The solutions are parametrised by a real parameter $\phi_0\in \R$. Set
\[
\phi_0 (\alpha',\kappa) = \tfrac{3}{4}\log{\tfrac{9}{2}|\alpha'|} - \tfrac{1}{2}\log{3\sqrt{2}\kappa}.
\]
\begin{enumerate}
\item Solutions corresponding to $\phi_0> \phi_0 (\alpha',\kappa)$ are complete and asymptotically conical, \ie for $t\ra \infty$ they approach the solution $(\omega_\tu{C},\Omega_\tu{C},e^{\phi_\infty})$ for some $\phi_\infty\in \R$.
\item For $\phi_0> \phi_0 (\alpha',\kappa)$, the asymptotic value $\phi_\infty$ of the dilaton is increasing in $\phi_0$.
\item The solution $(u_\ast,\phi_\ast)$ corresponding to $\phi_0= \phi_0 (\alpha',\kappa)$ is the CV--MN solution on the defomed conifold of Definition \ref{def:CVMN:Deformed}.
\item Solutions with $\phi_0< \phi_0 (\alpha',\kappa)$ are incomplete.
\end{enumerate}
\end{theorem}

The analogous result for the conifold is as follows.

\begin{theorem}\label{thm:Coho1:IIB:Resolved}
For each choice of $\alpha'\in\R\setminus\{0\}$ and $b$, consider the 1-parameter family of $\sunitary{2}^2$--invariant local singular solutions of the IIB system \eqref{eq:IIB} on the conifold constructed in Proposition \ref{prop:Coho1:IIB:Resolved:IVP}. The solutions are parametrised by a real parameter $c >0$.
\begin{enumerate}
\item Solutions corresponding to $c> 1$ are forward complete and asymptotically conical, \ie for $t\ra \infty$ they approach the solution $(\omega_\tu{C},\Omega_\tu{C},e^{\phi_\infty})$ for some $\phi_\infty\in \R$.
\item For $c>1$, the asymptotic value $\phi_\infty$ of the dilaton is increasing in $c$.
\item The solution $(u_\ast,\phi_\ast)$ corresponding to $c= 1$ is the CV--MN solution on the resolved conifold of Definition \ref{def:CVMN:Resolved}
\item Solutions with $0<c< 1$ are forward incomplete.
\end{enumerate}
\end{theorem}
Together with the local existence results of Propositions \ref{prop:Coho1:IIB:Deformed:IVP} and \ref{prop:Coho1:IIB:Resolved:IVP}, these two theorems recover Theorems \ref{mthm:IIB:Deformed} and \ref{mthm:IIB:Resolved} in the Introduction.

The proof, essentially the same for both theorems, is the analytic heart of this paper. The crucial role in our analysis is played by comparison with the explicit CV--MN solutions. There are two main steps of our analysis.
\begin{enumerate}
\item We first consider forward long-time existence of solutions of the ODE system. We find that the necessary and sufficient condition for existence forward in time is $|u-u_0|>0$. A comparison argument with the CV--MN solution $u_\ast$ guarantees long time existence of solutions lying ``above'' $u_\ast$.
\item Secondly, motivated by the fact that AC asymptotics can be deduced as soon as $u\gg u_0\approx u_\ast$ for $s\ra\infty$, we study the growth behaviour of $u$ by studying ratios $\frac{u}{u_\ast ^{1+\delta}}$ for small $\delta$. It turns out that good control of these ratios can be achieved if $\log{u}-\log{u_\ast}$ has good monotonicity properties, so part of our analysis requires proving this monotonicity must eventually hold for solutions that exists for all forward times lying ``above'' $u_\ast$. For solutions lying ``below'' $u_\ast$ the opposite monotonicity eventually forces the solution to blow up in finite time.
\end{enumerate}

\subsection{Comparison results}

We begin with some general analytic results about solutions $(u,\phi)$ of the ODE system of Proposition \ref{prop:Coho1:IIB} in either of the two possible complex structures. In this subsection we consider criteria for the maximal existence time of solutions and a simple but crucial comparison result for solutions of the system. In the next subsection we will prove a more technical growth estimate for solutions that are defined on an unbounded interval $[s_0,\infty)$.

Without any loss of generality, we work under the \textbf{standing assumption} that $\alpha'>0$. In particular, $\dot{u}_0>0$. By direct inspection, one also checks that in the complex structure of $T^\ast S^3$, \ie case (ii) in Proposition \ref{prop:Coho1:IIB}, we then have $u_0>0$ for all $s>0$. In case (i) in Proposition \ref{prop:Coho1:IIB} this is not the case and we would like to restrict to the set $\{ s\geq s_\ast\}$ where $u_0\geq 0$. Here $s_\ast$ is defined by $ \tfrac{9}{2}\alpha' s_\ast+b=0$. Hence as a first step of our analysis we prove that solutions must exists at least until time $s_\ast$.

\begin{lemma}\label{lem:IIB:Completeness:Conifold}
Any solution $(u,\phi)$ of the ODE system in case (i) of Proposition \ref{prop:Coho1:IIB} starting at some time $s_0<s_\ast$ exists forward in time with $u,\dot{u}, u^2-u_0^2>0$ up and including time $s_\ast$.
\proof
Let $(u,\phi)$ be defined (with $u,\dot{u}, u^2-u_0^2>0$) on some interval $[s_0,s]$ with $s_0<s\leq s_\ast$. We will show that we have a priori estimates independent of $s$.

Firstly, since $u_0\dot{u}_0\leq 0$ for $s\leq s_\ast$ we have $\dot{\phi}\geq 0$. Hence $e^{4\phi}$ is increasing and therefore bounded below. Note that we also have $\partial_s(u^2-u_0^2)> 0$ in this range (since $u\dot{u}>0$) and therefore $u^2-u_0^2$ is also bounded below. The latter point immediately implies that $u$ is bounded below. On the other hand, the lower bound on $u^2-u_0^2$ implies that the positive function $\dot{\phi}$ is bounded above and therefore so is $e^{4\phi}$ by integration. Taking into account that $(e^{-2\phi}\lambda\mu)^2=e^{6s}$ is also bounded, we deduce that $\dot{u}$ is bounded above and therefore so is $u$.
\endproof
\end{lemma}

We then make the additional \textbf{standing assumption} that \textbf{any solution $(u,\phi)$ is defined on an interval where $u_0>0$}.

\begin{lemma}\label{lem:IIB:Completeness}
Solutions $(u,\phi)$ of the ODE system of Proposition \ref{prop:Coho1:IIB} defined on a subset of $\{ u_0>0\}$ exist as long as $u-u_0$ stays bounded away from $0$.
\proof
Suppose that a solution exists on an interval $[s_0,s_1)$ and $u-u_0 \geq \theta>0$ for all $s\in [s_0,s_1)$. Then $u^2-u_0^2 \geq \theta (2u_0+\theta)$. Observe that $\phi$ is decreasing since $u_0\dot{u}_0>0$ by our standing assumption. Then
\[
|\dot{u}| \leq \frac{2(e^{-2\phi}\lambda\mu)^2 e^{4\phi_0}}{\theta (2u_0+\theta)}, \qquad |\dot{\phi}|\leq \frac{u_0\dot{u}_0}{\theta (2u_0+\theta)}
\]
are uniformly bounded in $[s_0,s_1)$ and $(u,\phi)$ can be extended up to and including time $s_1$.
\endproof
\end{lemma}

The main result of this subsection is the following comparison result.

\begin{prop}\label{prop:IIB:Comparison:CV:MN}
Let $(u_1,\phi_1)$ and $(u_2,\phi_2)$ be two solutions of the ODE stystem of Proposition \ref{prop:Coho1:IIB} defined on a subset of $\{ u_0>0\}$ with $\dot{u}_i>0,u_i>u_0$. Suppose that there exists $s_0$ such that
\begin{equation}\label{eq:IIB:CV-MN:Comparison}
u_1>u_2, \qquad \dot{u}_1>\dot{u_2}.
\end{equation}
Then the inequalities \eqref{eq:IIB:CV-MN:Comparison} persist for all $s\geq s_0$ such that both solutions exist. Moreover, $\phi_1>\phi_2$ in the same interval.
\proof
The second inequality of \eqref{eq:IIB:CV-MN:Comparison} is the first one that could possibly fail. At a time $s_\ast>s_0$ such that $u_1>u_2$ and $\dot{u}_1=\dot{u}_2$ we have, using the formulation of the ODE system as the second order ODE of Remark \ref{rmk:Coho1:IIB},
\[
\ddot{u}_1-\ddot{u}_2 = 2\dot{u}_1^2 \left( \frac{u_2}{u_2^2-u_0^2} - \frac{u_1}{u_1^2-u_0^2}\right) + 2u_0 \dot{u}_0\dot{u}_1  \left( \frac{1}{u_2^2-u_0^2} - \frac{1}{u_1^2-u_0^2}\right).
\]
Since the functions $x\mapsto \frac{x}{x^2-u_0^2}$ and $x\mapsto \frac{1}{x^2-u_0^2}$ are both decreasing in $x$ and $u_0\dot{u}_0>0$ by our standing assumptions, we conclude that $\ddot{u}_1-\ddot{u}_2>0$ at $s_\ast$, yielding a contradiction.

In order to obtain the final statement, simply note that
\[
e^{4\phi_1} = \frac{(u_1^2-u_0^2)\dot{u}_1}{2\left( e^{-2\phi}\lambda\mu\right)^2} > \frac{(u_2^2-u_0^2)\dot{u}_2}{2\left( e^{-2\phi}\lambda\mu\right)^2}=e^{4\phi_2}
\]
by the first equation in Proposition \ref{prop:Coho1:IIB}.
\endproof
\end{prop}

\subsubsection{Comparison results for the solutions of Proposition \ref{prop:Coho1:IIB:Resolved:IVP}}

We have explained how even in case (i) of Proposition \ref{prop:Coho1:IIB} we can assume without loss of generality that our solutions $u$ are defined on an interval where $u_0>0$. However, despite Lemma \ref{lem:IIB:Completeness:Conifold} guarantees that the solutions constructed in Proposition \ref{prop:Coho1:IIB:Resolved:IVP} extend until $u_0>0$, Proposition \ref{prop:Coho1:IIB:Resolved:IVP} only gives us information about the solutions near the initial condition, which is always in the region where $u_0<0$. For this reason we need some refined comparison results for this family.

Firts of all, we compare a solution $(u,\phi)$ from Proposition \ref{prop:Coho1:IIB:Resolved:IVP} with its leading order terms $(u_c,\phi_c)$ defined in \eqref{eq:CVMN:End:u}.

\begin{lemma}\label{lem:Coho1:IIB:Resolved:IVP}
Let $(u,\phi)$ be the solution of the system of case (i) of Proposition \ref{prop:Coho1:IIB} constructed in Proposition \ref{prop:Coho1:IIB:Resolved:IVP} with parameter $c>0$.
\begin{enumerate}
\item If $c>1$ then
\[
u>u_c, \qquad \dot{u}>\dot{u}_c
\]
as long as the solution is defined.
\item If $c<1$ then
\[
u<u_c, \qquad \dot{u}<\dot{u}_c
\]
as long as the solution is defined.
\end{enumerate}
\proof
Recall that, in the complex structure of the conifold,
\[
\frac{\ddot{u}}{\dot{u}} = 6 - \frac{2u\dot{u}+2 u_0\dot{u}_0}{u^2-u_0^2}.
\]
The inequalities in the statement of the Lemma are initially satisfied because in Proposition \ref{prop:Coho1:IIB:Resolved:IVP} we constructed $u$ with $u=u_c + \sigma^2 X_0 + O(\sigma^3)$ with the sign of $X_0$ the same as the sign of $c-1$. Here recall that $\sigma$ is simply a translate of $s$.

In either case (i) or (ii), in each pair of inequalities the first one that could possibly fail is the second one. However, at a time where $\dot{u}=\dot{u}_c=\dot{u}_0$ and $u-u_c$ has the same sign as $c-1$ we find
\[
\frac{\ddot{u}}{\dot{u}} = 6 - \frac{2\dot{u}_0}{u-u_0} > 6 - \frac{2\dot{u}_0}{u_c-u_0} = 6(1- c^{-1})>0
\]
if $c>1$ or
\[
\frac{\ddot{u}}{\dot{u}} = 6 - \frac{2\dot{u}_0}{u-u_0} < 6 - \frac{2\dot{u}_0}{u_c-u_0} = 6(1-c^{-1})<0
\]
if $c<1$. In either case, this leads to a contradiction since $\ddot{u}=\partial_s(\dot{u}-\dot{u}_c)$.
\endproof
\end{lemma}

In particular, for solutions corresponding to $c>1$ we have that $u>u_0$ and $\dot{u}>\dot{u}_0$ in the whole maximal existence interval (since $u_c>u_0$ and $\dot{u}_0=\dot{u}_c$). In this case, a different rearrangement of the second derivatives allows us to extend the comparison result of Proposition \ref{prop:IIB:Comparison:CV:MN} also when $u_0$ is not assumed to be positive.

\begin{prop}\label{prop:IIB:Comparison:CV:MN:Resolved}
Let $(u_1,\phi_1)$ and $(u_2,\phi_2)$ be two solutions of the ODE stystem of Proposition \ref{prop:Coho1:IIB} with $\dot{u}_i>0,u_i>|u_0|$ and $\dot{u}_i>\dot{u}_0$. Suppose that there exists $s_0$ such that
\begin{equation}\label{eq:IIB:Smoothing:Comparison}
u_1>u_2, \qquad \dot{u}_1>\dot{u}_2.
\end{equation}
Then the inequalities \eqref{eq:IIB:Smoothing:Comparison} persist for all $s\geq s_0$ such that both solutions exist. Moreover, $\phi_1>\phi_2$ in the same interval.
\proof
We argue as in Proposition \ref{prop:IIB:Comparison:CV:MN} but first rewrite
\[
\frac{\ddot{u}}{\dot{u}} = \partial_s \log{V^2}
- \frac{2u\dot{u}+2 u_0\dot{u}_0}{u^2-u_0^2} = \partial_s \log{V^2} - 2 \frac{u(\dot{u}-\dot{u}_0) + \dot{u}_0 (u+u_0)}{u^2-u_0^2}
\]
so that, at a time where $\dot{u}_1=\dot{u}_2$ we have
\[
\ddot{u}_1-\ddot{u}_2 = 2\dot{u}_1 (\dot{u}_1-\dot{u}_0) \left( \frac{u_2}{u_2^2-u_0^2} - \frac{u_1}{u_1^2-u_0^2}\right) + 2\dot{u}_0 \dot{u}_1  \left( \frac{1}{u_2-u_0} - \frac{1}{u_1-u_0}\right).
\]
Since $\dot{u}_1 (\dot{u}_1-\dot{u}_0)>0$ by assumption and $\dot{u}_0>0$ since $\alpha'>0$, we conclude as in Proposition \ref{prop:IIB:Comparison:CV:MN}.
\endproof
\end{prop}

\subsection{Asymptotic behaviour}

We continue to work in the general setting of Proposition \ref{prop:Coho1:IIB} for either complex structure. We let $u_\ast$ be the explicit solution given by Definitions \ref{def:CVMN:Resolved} and \ref{def:CVMN:Deformed} for the complex structure in case (i) and, respectively, (ii) of Proposition \ref{prop:Coho1:IIB}. We continue to assume without loss of generality that $\alpha'>0$ and $u_0>0$.

\begin{remark}\label{rmk:Growth:CV:MN}
We make the following explicit observations about the asymptotic behaviour of $u_\ast$ and $u_0$, which follow by direct inspection. The function $u_\ast$ is an affine function with $\dot{u}_\ast = \tfrac{9}{2}\alpha'$ and $u_0\approx u_\ast - \tfrac{3}{2}\alpha'$ (up to exponentially decaying terms). Thus we calculate
\[
\frac{2u_0}{u_\ast^2-u_0^2}\left( \dot{u}_0 + u_0 \frac{\dot{u}_\ast}{u_\ast}\right) \approx 6  + O(s^{-1})
\]
as $s\ra\infty$. This will be used in the following discussion.
\end{remark}

The following proposition establishes growth behaviour, measured in terms of the explicit solution $u_\ast$, for solutions of the second order ODE of Remark \ref{rmk:Coho1:IIB} defined for all time $[s_0,\infty)$ in either choice of complex structure, case (i) and case (ii) in Proposition \ref{prop:Coho1:IIB}.  

\begin{prop}\label{prop:Growth:CV:MN}
Let $u$ be a solution of the second order ODE of Remark \ref{rmk:Coho1:IIB} with $\alpha'>0$ and $\dot{u},u-u_0>0$ for all $s\in [s_0,\infty)$.
\begin{enumerate}
\item If $u>u_\ast$ for all $s\in [s_0,\infty)$ and at time $s_0$ we have 
\[
\frac{\dot{u}}{u} > \frac{\dot{u}_\ast}{u_\ast}
\]
then 
\[
R_\delta = \frac{u_\ast^{1+\delta}}{u}
\]
is a decreasing function for all $\delta$ sufficiently small.
\item If $u<u_\ast$ for all $s\in [s_0,\infty)$ and at time $s_0$ we have 
\[
\frac{\dot{u}}{u} < \frac{\dot{u}_\ast}{u_\ast}
\]
then 
\[
R_\delta = \frac{u_\ast^{1+\delta}}{u}
\]
is an increasing function for all $\delta$ sufficiently small.
\end{enumerate}
\begin{proof}
We calculate
\[
S_\delta:= R_\delta^{-1}\dot{R}_\delta = (1+\delta) \frac{\dot{u}_\ast}{u_\ast} - \frac{\dot{u}}{u}.
\]
Our assumptions say that either
\begin{enumerate}
\item $R_0\in (0,1)$ for all $s$ and $S_0<0$ at time $s_0$, or
\item $R_0>1$ for all $s$ and $S_0>0$ at time $s_0$.
\end{enumerate}

At any point where $S_\delta=0$ we calculate
\begin{align*}
(1+\delta)^{-1}\frac{u_\ast}{\dot{u}_\ast}\dot{S}_\delta|_{S_\delta=0} &= \frac{\ddot{u}_\ast}{\dot{u}_\ast} - \frac{\ddot{u}}{\dot{u}} + \delta\, \frac{\dot{u}_\ast}{u_\ast}\\
& = 2u_0 \dot{u}_0 \left( \frac{1}{u^2-u_0^2}-\frac{1}{u_\ast^2-u_0^2}\right) + 2\frac{\dot{u}_\ast}{u_\ast} \left( \frac{u^2}{u^2-u_0^2}-\frac{u_\ast^2}{u_\ast^2-u_0^2}\right) + \delta\, \frac{\dot{u}_\ast}{u_\ast} \left(\frac{2u^2}{u^2-u_0^2}+1\right)\\
&= -\frac{2u_0 (u^2-u_\ast^2)}{(u^2-u_0^2)(u_\ast^2-u_0^2)}\left( \dot{u}_0 + 2u_0 \frac{\dot{u}_\ast}{u_\ast}\right) + \delta\, \frac{\dot{u}_\ast}{u_\ast} \left(\frac{2u^2}{u^2-u_0^2}+1\right)\\
& = -\frac{2u_0 }{(u_\ast^2-u_0^2)}\left( \dot{u}_0 + u_0 \frac{\dot{u}_\ast}{u_\ast}\right) \frac{1-R_0^2}{1-\left(\frac{u_0}{u}\right)^2} + 2\delta\,  \frac{\dot{u}_\ast}{u_\ast} \frac{1}{1-\left(\frac{u_0}{u}\right)^2}+ \delta\,  \frac{\dot{u}_\ast}{u_\ast}.
\end{align*}
Observe that
\[
\frac{2u_0 }{(u_\ast^2-u_0^2)}\left( \dot{u}_0 + u_0 \frac{\dot{u}_\ast}{u_\ast}\right)>0, \qquad 1-\left(\frac{u_0}{u}\right)^2>0
\]
for all $s\in [s_0,\infty)$, while for large $s$ by Remark \ref{rmk:Growth:CV:MN} we have
\[
\frac{2u_0 }{(u_\ast^2-u_0^2)}\left( \dot{u}_0 + u_0 \frac{\dot{u}_\ast}{u_\ast}\right) \approx 6 + O(s^{-1}), \qquad \frac{\dot{u}_\ast}{u_\ast}  = O(s^{-1}).
\]

Now, first consider the case $\delta=0$.
\begin{enumerate}
\item If $R_0<1$ for all $s$, we have $\dot{S}_0<0$ at any time where $S_0=0$. Thus $S_0<0$ for all $s\in [s_0,\infty)$ and in particular $R_0$ is strictly decreasing.
\item If $R_0>1$ for all $s$, we have $\dot{S}_0>0$ at any time where $S_0=0$. Thus $S_0>0$ for all $s\in [s_0,\infty)$ and in particular $R_0$ is strictly increasing.
\end{enumerate}
Fix $s_\ast$ sufficiently large so that the asymptotic expansions of Remark \ref{rmk:Growth:CV:MN} are valid for all $s\geq s_\ast$. For $\delta$ sufficiently small the sign of $S_\delta$ on $[s_0,s_\ast]$ is the same as the sign of $S_0$. Moreover if at a later time $s>s_\ast$ we have $S_\delta =0$ then
\[
(1+\delta)^{-1} \frac{u_\ast}{\dot{u}_\ast}\dot{S}_\delta \approx \left( -6(1-R_0^2) + O(s^{-1})\right) \frac{1}{1-\left( \frac{u_0}{u}\right)^2} + O(s^{-1}).
\]
Thus the sign of $\dot{S}_\delta$ is the same as the sign of $R_0-1$. We conclude that the sign of $S_\delta$ is the same as the sign of $S_0$ for all $s\geq s_0$, \ie $R_\delta$ has the same monotonicity of $R_0$. 
\end{proof}
\end{prop}

By Remark \ref{rmk:Coho1:IIB:Deformed:IVP} the additional assumptions on the sign of $\frac{\dot{u}}{u}-\frac{\dot{u}_\ast}{u_\ast}$ at some time $s_0$ in Proposition \ref{prop:Growth:CV:MN} are satisfied for small enough $s>0$ for all members of the family given by Proposition \ref{prop:Coho1:IIB:Deformed:IVP}, where the sign of $\phi_0-\phi_0(\alpha',\kappa)$ determines which inequality is satisfied. In order to establish the same result for the family of solutions produced by Proposition \ref{prop:Coho1:IIB:Resolved:IVP} we need to work harder. In fact, it suffices to prove that the additional assumptions on the sign of $\frac{\dot{u}}{u}-\frac{\dot{u}_\ast}{u_\ast}$ at some time $s_0$ in Proposition \ref{prop:Growth:CV:MN} are eventually satisfied  for \emph{any} solution that exists for all time.

\begin{prop}\label{prop:Growth:CV:MN:I}
Let $u$ be a solution of the second order ODE of Remark \ref{rmk:Coho1:IIB} with $\alpha'>0$ and $\dot{u},u-u_0>0$ for all $s\in [s_0,\infty)$.
\begin{enumerate}
\item If $u>u_\ast$ and $\dot{u}>\dot{u}_\ast$ for all $s\in [s_0,\infty)$ then there exists $s_1\geq s_0$ such that at time $s=s_1$
\[
\frac{\dot{u}}{u} > \frac{\dot{u}_\ast}{u_\ast}.
\]
\item If $u<u_\ast$ and $\dot{u}<\dot{u}_\ast$ for all $s\in [s_0,\infty)$ then there exists $s_1\geq s_0$ such that at time $s=s_1$
\[
\frac{\dot{u}}{u} < \frac{\dot{u}_\ast}{u_\ast}.
\]
\end{enumerate}
\proof
Consider case (i). By contradiction, assume that
\[
\frac{\dot{u}}{u} \leq \frac{\dot{u}_\ast}{u_\ast}
\]
for all $s\in [s_0,\infty)$. In other words $\left(\log{u}\right)' \leq \left(\log{u_\ast}\right)'$ and by integration we conclude that $u$ has at most linear growth. We are going to deduce a contradiction by showing that $\left(\log{\dot{u}}\right)'\geq \kappa>0$, which implies by integration that $\dot{u}$ and $u$ are at least exponentially growing.

Using the second order ODE of Remark \ref{rmk:Coho1:IIB}, we calculate
\[
\left(\log{\dot{u}}\right)' = \frac{\ddot{u}}{\dot{u}} = \partial_s \log{\left( e^{-2\phi}\lambda\mu\right)^2} - 2\frac{u\dot{u}+u_0\dot{u_0}}{u^2-u_0^2}.
\]
We study the second term. Under our assumptions, we have $\frac{\dot{u}}{u}<\frac{\dot{u}_\ast}{u_\ast}$ and $u>u_\ast +\tfrac{3}{2}\alpha'\theta$ for $\tfrac{3}{2}\alpha'\theta = u(s_0)-u_\ast(s_0)>0$. Thus, using $\dot{u}_0=\dot{u}_\ast$ and writing $u=(u^2-u_0^2)u^{-1}+u_0^2 u^{-1}$ in the numerator,
\begin{align*}
\frac{u\dot{u} + u_0\dot{u_0}}{u^2-u_0^2}  &= \left( 1 + \frac{u_0^2 }{u^2-u_0^2} \right)\frac{\dot{u}}{u} + \frac{u_0u_\ast}{u^2-u_0^2}\frac{\dot{u}_\ast}{u_\ast} \leq \left( 1+\frac{u_0 (u_0+u_\ast)}{u^2-u_0^2}\right)\frac{\dot{u}_\ast}{u_\ast}\\
&<  \left( 1+\frac{u_0 (u_0+u_\ast)}{(u_\ast+u_0+\tfrac{3}{2}\alpha'\theta)(u_\ast -u_0 + \tfrac{3}{2}\alpha'\theta)}\right)\frac{\dot{u}_\ast}{u_\ast} \stackrel{s\ra\infty} {\longrightarrow} \frac{3}{1+\theta}.
\end{align*}
Using that $\partial_s \log{\left( e^{-2\phi}\lambda\mu\right)^2}\approx 6$ for $s\ra\infty$ we conclude that for large enough $s$
\[
\left(\log{\dot{u}}\right)' \geq \tfrac{6\theta}{1+\theta} -\epsilon>0.
\]

The proof in case (ii) is analogous. By contradiction, assume that
\[
\frac{\dot{u}}{u} \geq \frac{\dot{u}_\ast}{u_\ast}
\]
for all $s\in [s_0,\infty)$ so that $u$ has at least linear growth. We obtain a contradiction by showing that $\left(\log{\dot{u}}\right)'\leq -\kappa<0$, which implies by integration that $\dot{u}$ is exponentially decaying and therefore $u$ is bounded. The assumptions now imply that $u<u_\ast -\tfrac{3}{2}\alpha'\theta$ for some $\theta\in (0,1)$ (since $u_0<u<u_\ast$ and $u-u_\ast$ is decreasing). All the inequalities in the previous estimates for $\left(\log{\dot{u}}\right)'$ are reversed, concluding that for large enough $s$
\[
\left(\log{\dot{u}}\right)' \leq -\tfrac{6\theta}{1-\theta} +\epsilon <0. \qedhere
\]
\endproof
\end{prop}

\subsection{Proof of Theorems \ref{thm:Coho1:IIB:Deformed} and \ref{thm:Coho1:IIB:Resolved}}

We are now ready to explain the proof of the main results of this section, Theorems \ref{thm:Coho1:IIB:Deformed} and \ref{thm:Coho1:IIB:Resolved}.

We begin with Theorem \ref{thm:Coho1:IIB:Deformed} and therefore consider the one parameter family of solutions of the ODE system of case (ii) of Proposition \ref{prop:Coho1:IIB} given by Proposition \ref{prop:Coho1:IIB:Deformed:IVP}. The solutions are parametrised by $\phi_0$ and by Remark \ref{rmk:Coho1:IIB:Deformed:IVP} we know that at some time $s_0>0$ sufficiently small
\begin{enumerate}
\item if $\phi_0 > \phi_0 (\alpha',\kappa)$ then the solution $u$ satisfies the inequalities
\[
u>u_\ast, \qquad \dot{u}>\dot{u}_\ast, \qquad \frac{\dot{u}}{u}>\frac{\dot{u}_\ast}{u_\ast};
\]
\item if $\phi_0 < \phi_0 (\alpha',\kappa)$ then the solution $u$ satisfies the inequalities
\[
u<u_\ast, \qquad \dot{u}<\dot{u}_\ast, \qquad \frac{\dot{u}}{u}<\frac{\dot{u}_\ast}{u_\ast}.
\]
\end{enumerate}
Proposition \ref{prop:IIB:Comparison:CV:MN} implies that the first two inequalities in each case are preserved forward in time as long as the solution $u$ exists.

In particular, when $\phi_0>\phi_0(\alpha',\kappa)$ we have $u>u_\ast>u_0$ and therefore the solution $u$ is defined for all $s\in [0,\infty)$ by Lemma \ref{lem:IIB:Completeness}. Proposition \ref{prop:Growth:CV:MN} then implies that for $\delta>0$ we have
\[
u>u_\ast^{1+\delta}
\]
for all $s>0$, \ie $u$ has superlinear growth at infinity. In particular, using that fact that $u_0$ grows instead linearly in $s$, we deduce that
\[
\lim_{s\ra\infty}{\frac{u_0}{u}}= 0, \qquad u^{-1}\dot{u}_0\in L^1.
\]
Consider now the ODE system of Proposition \ref{prop:Coho1:IIB}. The previous observations imply
\[
|\dot{\phi}| =  \frac{u_0}{u}\frac{1}{1-\left( \frac{u_0}{u}\right)^2} \frac{\dot{u}_0}{u} \in L^1
\]
and $\phi$ has a finite limit $\phi_\infty$ as $s\ra\infty$. More precisely, $\phi = \phi_\infty + O(s^{-2\delta})$. Using Remark \ref{rmk:Coho1:IIB:Infinity}, we then deduce
\[
u^2\dot{u} \approx \tfrac{1}{2}\kappa^2 e^{4\phi_\infty} e^{6s}.
\]
Integrating we obtain that $u$ has growth
\[
u\approx \left( \left(\tfrac{\kappa}{2}\right)^{\frac{1}{3}} e^{\frac{2}{3}\phi_\infty} e^s\right)^2.
\]
By a change of variable $\left(\tfrac{\kappa}{2}\right)^{\frac{1}{3}} e^{\frac{2}{3}\phi_\infty} e^s=r$ we see that geometrically $(\omega,\Omega,\phi)$ is asymptotic to the conical solution $(\omega_{\tu{C}},\Omega_\tu{C}, \phi_\infty)$.

This concludes the proof of part (i) of Theorem \ref{thm:Coho1:IIB:Deformed}. Part (ii) follows from Proposition \ref{prop:IIB:Comparison:CV:MN} and Remark \ref{rmk:Coho1:IIB:Deformed:IVP} and part (iii) is Definition \ref{def:CVMN:Deformed}.

We now prove part (iv). Assume by contradiction that a solution $u$ corresponding to $\phi_0<\phi_0 (\alpha',\kappa)$ is defined for all $s\in [0,\infty)$. Then we can apply the second part of Proposition \ref{prop:Growth:CV:MN} concluding that
\[
u<u_\ast^{1+\delta}
\]
for all small $\delta$. Taking $\delta<0$ we deduce that $u$ has sublinear growth at infinity and therefore the condition $u>u_0$ could not possibly be satisfied for $s$ sufficiently large.

The proof of Theorem \ref{thm:Coho1:IIB:Resolved} is similar, but we use Lemma \ref{lem:Coho1:IIB:Resolved:IVP} and Proposition \ref{prop:Growth:CV:MN:I} instead of Remark \ref{rmk:Coho1:IIB:Deformed:IVP} to guarantee that
\[
u-u_\ast, \qquad \dot{u}-\dot{u}_c, \qquad \frac{\dot{u}}{u} -\frac{\dot{u}_\ast}{u_\ast} 
\]
have the same sign as $c-1$ for some sufficiently large $s_0$. Indeed, Lemma \ref{lem:Coho1:IIB:Resolved:IVP} implies that the sign of $u-u_c$ and $\dot{u}-\dot{u}_c$ is the same as the sign of $c-1$ for all times. Since on the common interval of definition the sign of $u_c-u_1=u_c-u_\ast$ is also the same as the sign of $c-1$ and $\dot{u}_c$ is independent of $c$, we obtain the statements about the sign of $u-u_\ast$ and $\dot{u}-\dot{u}_\ast$. We can then apply Proposition \ref{prop:Growth:CV:MN:I} to control the sign of $\frac{\dot{u}}{u} -\frac{\dot{u}_\ast}{u_\ast}$ for sufficiently large times. Finally, part (ii) of Theorem \ref{thm:Coho1:IIB:Resolved} follows from Proposition \ref{prop:IIB:Comparison:CV:MN:Resolved}. Indeed, note that since $u_c(s)$ is increasing in $c$ for every fixed $s$, $\dot{u}_c$ is constant in $c$ and the solution $u$ corresponding to parameter $c$ has $u=u_c$ and $\dot{u}=\dot{u}_c$ at the initial time, Lemma \ref{lem:Coho1:IIB:Resolved:IVP} implies that $u(s)$ and $\dot{u}(s)$ are also increasing in $c$, so that the hypotheses of Proposition \ref{prop:IIB:Comparison:CV:MN:Resolved} hold. This is shown as follows: take $c_1>c_2>1$ and let $u_1$ and $u_2$ be the corresponding solutions. At the time $s_2$ where $u_2(s_2)=u_{c_2}(s_2)=-u_0(s_2)$ we have $u_1(s_2)>u_{c_1}(s_2)>u_{c_2}(s_2)=u_2(s_2)$ and $\dot{u}_1(s_2)>\dot{u}_{c_1}(s_2)=\dot{u}_{c_2}(s_2)=\dot{u}_2(s_2)$ (where equalities are understood as one sided limits $s \downarrow s_2$). By continuity, $u_1>u_2$ and $\dot{u}_1>\dot{u}_2$ at some time $s_0>s_2$, and it follows from this that the inequalities of Proposition \ref{prop:IIB:Comparison:CV:MN:Resolved} are satisfied.

\subsection{Discussion}

We close the paper by collecting some properties and observations about the solutions produced in Theorems \ref{thm:Coho1:IIB:Deformed} and \ref{thm:Coho1:IIB:Resolved}.

\subsubsection{Genuine 1-parameter families of new non-K\"ahler solutions}

First of all, we note that our existence theorems yield 1-parameter families of genuinely distinct solutions of the IIB system up to the symmetries of Section \ref{sec:Parameters}. Indeed, for Theorem \ref{thm:Coho1:IIB:Deformed} we can simply use the dilaton-translation freedom $\phi\mapsto \phi+\psi$ to fix the parameter $\phi_0$. Since $(\alpha',\kappa)$ have a cohomological interpretation, solutions corresponding to different values of these parameters cannot be related by a diffeomorphism that fixes the asymptotic cone at infinity (up to discrete symmetries that change the signs of $\alpha'$ and $\kappa$). The scaling symmetry of Section \ref{sec:Parameters} can then be used to reduce to a single genuine parameter. For Theorem \ref{thm:Coho1:IIB:Resolved} we can instead use the scaling and dilaton-translation symmetries of Section \ref{sec:Parameters} to fix the parameters $(\alpha',b)$. Solutions corresponding to distinct parameter $c$ cannot differ by a diffeomorphism because they have CV--MN singular ends with different parameters.

It is also clear that the solutions cannot be homogeneous because the dilaton $\phi$ is never constant and, as already observed in the Introduction, the solutions are never K\"ahler for $\alpha'\neq 0$ simply because then $[d^c\omega]\neq 0$ in cohomology.

\subsubsection{Asymptotically conical behaviour and the Bismut connection} We will now explain how to deduce Theorem \ref{mthm:Bismut} in the Introduction from our existence results.

In Theorems \ref{thm:Coho1:IIB:Deformed} and \ref{thm:Coho1:IIB:Resolved} we showed that the generic (forward) complete solution is asymptotically conical in the sense that $\phi$ converges to a constant $\phi_\infty$ and after a change of variable $u\approx r^2$. By rescaling and a shift of the dilaton $\phi$ by a constant, for simplicity we can always assume that $\phi_\infty=0$ and $\kappa=2$ in the complex structure of the deformed conifold. In the new variable $r$, we then have $V=r^3(1+ O(r^{-6}))$ and $u_0 = O(\log{r})$.

Now, consider the IIB ODE system of Proposition \ref{prop:Coho1:IIB}. Integrating the ODE for $\phi$ we obtain that $\phi = O(r^{-4}(\log{r})^2)$ and therefore the ODE for $u$ reads
\[
\partial_r u^3 = 6r^5 \left(1+ O(r^{-4}(\log{r})^2)\right).
\]
Integrating yields $u=r^2 (1+O(r^\nu))$ for any $\nu>-3$, with similar estimates for all derivatives. This implies that $\omega -\omega_{\tu{C}}=O(r^{\tau})$ for any $\tau>-2$ (because of the logarithmic behaviour of $u_0$).

In particular, the Levi--Civita connection $\nabla^{\tu{LC}}$ of the metric induced by $(\omega,\Omega)$ is of the form $\nabla^{\tu{LC}} = \nabla^{\tu{C}}+O(r^{\tau-1})$, where $\nabla^{\tu{C}}$ is the Levi--Civita connection of the Calabi--Yau cone metric on $\tu{C}$ and $d^c\omega=O(r^{\tau-1})$ (in fact, the explicit expressions \eqref{eq:dcomegaBfield} and \eqref{eq:Bfield} show that $d^c\omega = O(r^{-3})$). It follows that the Bismut connection $\nabla^{\tu{B}}$ of the generic complete solution in Theorems \ref{thm:Coho1:IIB:Deformed} and \ref{thm:Coho1:IIB:Resolved} is of the form $\nabla^{\tu{B}} = \nabla^{\tu{C}} + O(r^{-1+\tau})$ for all $\tau>-2$ (in particular we can certainly take $\tau<0$). 

With these analytical facts in place, we turn to the characterization of the Bismut holonomy. Let $V \subset \Lambda^*T^*M$ denote the bundle of $\sunitary{3}$--invariant differential forms, generated fibre-wise as an algebra by $\omega$, $\operatorname{Re} \Omega$, and $\operatorname{Im} \Omega$, and let $V^\perp$ its orthogonal complement.
The list of proper connected closed subgroups of $\sunitary{3}$ is: $\tu{S}(\tu{U}(1) \times \tu{U}(2)) \cong \tu{U}(2)$, $\tu{SO}(3)$, $\sunitary{2}$, $\tu{T}^2$ and $\unitary{1}_{p,q}=\{ (e^{ip\theta},e^{iq\theta},e^{-i(p+q)\theta})\,|\, e^{i\theta}\in\unitary{1}\}$.
Assume that the holonomy of $\nabla^{\tu B}$ is a proper subgroup $G \subset \sunitary{3}$. Then, by direct inspection, in any of the cases above the holonomy representation at a point $x \in M$ must fix $\sigma_x \in V_x^\perp \subset \Lambda^*T_x^*M$ of degree $2$ or degree $3$. By parallel transport, $\sigma_x$ extends to a non-zero section $\sigma$ of $V^\perp$, which is Bismut parallel. In particular, $\sigma$ has non-vanishing constant norm $c = |\sigma|_g \neq 0$. By convergence of the Bismut connection to the Levi--Civita connection of the Calabi--Yau cone metric on the conifold, one can then show that $\sigma$ must be asymptotic to a $ \nabla^\tu{C}$--parallel form lying in the orthogonal complement of the bundle of $\sunitary{3}$--invariant differential forms on the cone. Since $\nabla^\tu{C}$ has full holonomy $\sunitary{3}$, the limit must vanish, which contradicts that $c \neq 0$. Therefore, we conclude that the holonomy of $\nabla^{\tu B}$ is the full group $\tu{SU}(3)$.

\begin{remark*}
Recall here that $\nabla^{\tu{C}}$ has full holonomy $\sunitary{3}$ by Gallot's Theorem (which states that the metric cone over a complete Riemannian manifold is either flat or irreducible) combined with Berger's classification of Riemannian holonomy groups, which states that an irreducible, non-locally-symmetric Riemannian $6$-manifold has holonomy either $\tu{SO}(6)$, $\unitary{3}$, or $\sunitary{3}$.
\end{remark*}

\begin{remark*}
As observed in the Introduction, Theorem \ref{mthm:Bismut} is in contrast with the compact case, where the holonomy of the Bismut connection of any non-K\"ahler BHE metric is contained in $\tu{SU}(n-1) \subset \sunitary{n}$, where $n$ is the complex dimension \cite{ABLS26,GFJS}.
\end{remark*}

\subsubsection{Asymptotics of the dilaton functional}\label{sec:Dilaton:Functional}

In this section we study the \emph{dilaton functional}
\begin{equation}\label{eq:Dilaton:Functional}
\mathcal{D}(\omega,\phi) := \int{e^{-2\phi}\,\frac{\omega^3}{6}}
\end{equation}
as introduced in \cite{GFRST}, on the families of solutions produced by Theorems
\ref{thm:Coho1:IIB:Deformed} and \ref{thm:Coho1:IIB:Resolved}. The functional $\mathcal{D}$ diverges on the different domains of the solutions, and hence we will introduce cut-offs and calculate the corresponding asymptotics. In the present setup, the functional \eqref{eq:Dilaton:Functional} is very natural, as $e^{-2\phi}\vol_g = e^{-2\phi}\omega^3/6$ is the weighted measure with respect to which a solution of the IIB
system is a gradient steady generalized Ricci soliton (in particular, it is used in the integration by parts of Proposition \ref{prop:soliton}).

Let $(\omega,\Omega,\phi)$ be an $\sunitary{2}^2$--invariant solution of the
IIB system on $M^\ast$ in the normal form \eqref{eq:normalomega}, and let
$V=e^{-2\phi}\lambda\mu$ be as in Proposition \ref{prop:Coho1:IIB}. Then
$$
e^{-2\phi}\frac{\omega^3}{6} = \tfrac12 V^2 e^{2\phi} ds\wedge\eta^{se}\wedge(\omega_1^{se})^2,
$$
and consequently, for any interval $I$ of the variable $s$,
\begin{equation}\label{eq:Dilaton:1D}
\mathcal{D}\bigl( \{ s\in I\}\bigr) := \frac{16\pi^3}{27}\int_I{V^2e^{2\phi}\,ds}.
\end{equation}
By Proposition \ref{prop:Coho1:IIB} the integrand of
\eqref{eq:Dilaton:1D} is
\[
V^2e^{2\phi} =
\begin{cases}
e^{6s}e^{2\phi} & \qquad\text{in case (i)},\\[2pt]
\kappa^2\sinh^2{(3s)} e^{2\phi} & \qquad\text{in case (ii)},
\end{cases}
\]
so that everything reduces to analyse the behaviour of the dilaton.

We begin with the two borderline members, $p=p_\ast$ and $q=q_\ast$. The two cases behave rather differently at the level of primitives, while having the same growth. For the CV--MN solution on the resolved conifold (see Definition \ref{def:CVMN:Resolved}), with parameters $(\alpha',b)$ and
$W:=\tfrac92\alpha's+b+\tfrac34\alpha' = u_\ast -\tfrac34\alpha'$, one has
\begin{equation}\label{eq:Dilaton:CVMN:Resolved}
V^2e^{2\phi_\ast} \; = \; \frac{3\sqrt3}{2}\,\alpha'\,e^{3s}\sqrt{W},
\end{equation}
and, with $x:=\sqrt{2W/3\alpha'}$,
\begin{equation}\label{eq:Dilaton:CVMN:Resolved:Primitive}
\int{V^2e^{2\phi_\ast}\,ds}
\; = \; \frac{3\,(\alpha')^{3/2}}{\sqrt2}\,
e^{-\frac{2b}{3\alpha'}-\frac12}
\left( \tfrac12\,x\,e^{x^2} - \frac{\sqrt\pi}{4}\operatorname{erf}(x)\right),
\end{equation}
where $\operatorname{erf}$ denotes the \emph{error function}. In particular the primitive is not elementary, and
$\mathcal{D} (\{ s\leq S\}) \sim \tu{const}\cdot e^{3S}\sqrt S$ as $S\ra\infty$.

On the other hand, for the CV--MN solution on the deformed conifold with parameters $(\alpha',\kappa)$,
$$
V^2e^{2\phi_\ast} \; = \; \frac{\kappa}{3\sqrt2}\left( \tfrac92\alpha'\right)^{3/2}\sqrt{D}
\; = \; \tfrac94\,\kappa\,(\alpha')^{3/2}\sqrt D,
$$
where
$$
D := 9s^2\sinh^2{(3s)} - \left( 3s\cosh{(3s)}-\sinh{(3s)}\right)^2 = 3s\sinh{(6s)} - 9s^2 - \tfrac12\left( \cosh{(6s)}-1\right).
$$
Hence, as $s\ra\infty$, $D\sim\tfrac14e^{6s}(6s-1)$ and
\[
V^2e^{2\phi_\ast} \sim \frac{\kappa}{2\sqrt3}\left(\tfrac92\alpha'\right)^{3/2}
e^{3s}\sqrt{s-\tfrac16},
\]
so that again $\mathcal{D} (\{ s\leq S\})\sim \tu{const}\cdot e^{3S}\sqrt S$.

\begin{remark*}
As $s\ra 0$ one has $V^2e^{2\phi_\ast} = \tfrac{81}{4}\kappa (\alpha')^{3/2}s^2 + O(s^4)$ and $\mathcal{D}$ is finite near the zero-section.
\end{remark*}

We turn to the members $p>p_\ast$ and $q>q_\ast$, which by Theorems
\ref{thm:Coho1:IIB:Deformed} and \ref{thm:Coho1:IIB:Resolved} are
asymptotic to the cone solution $(\omega_{\tu{C}},\Omega_{\tu{C}},e^{\phi_\infty})$.
We normalise the radius by $u\simeq r^2$, that is, $r=e^se^{2\phi_\infty/3}$. Let $(\omega,\Omega,\phi)$ be an asymptotically conical member of either
family, with asymptotic dilaton $\phi_\infty$. Then the dilaton
$\phi$ has the expansion
$$
\phi = \phi_\infty + (c_1s+c_0)\,e^{-4s} + O(s^2e^{-8s}),
\qquad
c_1 = \tfrac{81}{16}(\alpha')^2e^{-8\phi_\infty/3},
\qquad
\frac{c_0}{c_1} = \frac14+\frac{2b}{9\alpha'},
$$
and, consequently, one can prove that
\begin{equation}\label{eq:Dilaton:AC}
\mathcal{D}\bigl(\{ r\leq R\}\bigr) = 
\frac{8\pi^3}{81}\,e^{-2\phi_\infty}R^6 + 3\pi^3(\alpha')^2e^{-2\phi_\infty}\,R^2\log{R}
 +  O(R^2).
\end{equation}
Hence, in this case, the dilaton functional is polynomial growth of degree $6$, with leading coefficient strictly decreasing in the parameter $\phi_\infty$, and proportional to the volume of the cone 
$$
\int_{\{ r\leq R\}} \frac{\omega_{\tu{C}}^3}{6}= \tfrac{16\pi^3}{27}\int_0^R{t^5\,dt}
= \tfrac{8\pi^3}{81}R^6.
$$

\bibliographystyle{amsinitial}
\bibliography{IIB}

\end{document}